\documentclass[twoside,11pt]{article}

\usepackage{jmlr2e}

\jmlrheading{1}{2022}{1-30}{4/00}{10/00}{usmanova}{Ilnura Usmanova, Yarden As, Maryam Kamgarpour, and Andreas Krause}

\ShortHeadings{Log Barriers for Smooth Safe Optimization}{Usmanova, As, Kamgarpour, and Krause}
\firstpageno{1}

\usepackage{amsmath}      
  {
      \newtheorem{assumption}{Assumption}
       \newtheorem{fact}{Fact}
  }

\usepackage{bbm}
\usepackage{bm}
\usepackage{algorithm}
\usepackage{algorithmic}
\makeatletter
\def\BState{\State\hskip-\ALG@thistlm}
\usepackage{fullpage}
\usepackage{url}
\usepackage[utf8]{inputenc}
\usepackage[T1]{fontenc}  
\usepackage{hyperref}      
\usepackage{url}           
\usepackage{booktabs}      
\usepackage{amsfonts}      
\usepackage{nicefrac}      
\usepackage{microtype}     
\usepackage{cleveref}
\usepackage{graphbox}
\usepackage{graphicx}
\usepackage{xcolor}
\usepackage{import}
\usepackage{multicol, blindtext}
\usepackage[font=small,labelfont=bf]{caption}
\usepackage{natbib}
\usepackage{wrapfig}
\usepackage{lipsum}
\usepackage{svg}
\usepackage{paralist}
\usepackage{tabularx}

\RequirePackage{filecontents}
\makeatletter
\def\BState{\State\hskip-\ALG@thistlm}
\makeatother
\newcommand{\be}{\begin{equation}}
\newcommand{\ee}{\end{equation}}
\newcommand{\e}{\varepsilon}
\newcommand{\la}{\langle}
\newcommand{\ra}{\rangle}

\newcommand{\E}{\mathbb E}
\newcommand{\R}{\mathbb R}

\newcommand{\de}{\delta}
\newcommand{\LBM}{LB-SGD}

\newcommand{\Prob}{\mathbb P}

\newcommand{\I}{\mathcal{I}}

\crefname{assumption}{assumption}{assumptions}
\crefname{lemma}{Lemma}{Lemmas}

\begin{document}

\title{
Log Barriers for Safe Black-box Optimization with Application to Safe Reinforcement Learning
}

\author{\name 
    Ilnura Usmanova 
      \email{ilnurau@control.ee.ethz.ch} \\
        \addr    
        {Automatic Control Laboratory, D-ITET,\\ ETH Zürich, 8092 Zurich, Switzerland}
    \AND 
    Yarden As 
      \email{yarden.as@inf.ethz.ch} \\
        \addr 
        Institute for Machine Learning,
        D-INFK,\\ ETH Zürich, 8092 Zurich, Switzerland 
    \AND Maryam Kamgarpour* \email{maryam.kamgarpour@epfl.ch}\\
        \addr
        STI-IGM-Sycamore, 
 EPFL, 1015 Lausanne, Switzerland
    \AND Andreas Krause* \email{ krausea@ethz.ch}\\
        \addr 
        Institute for Machine Learning,
        D-INFK,\\ ETH Zürich, 8092 Zurich, Switzerland 
}
\editor{}
\maketitle
\begin{abstract}%
    Optimizing noisy functions online, when evaluating the objective requires experiments on a deployed system, is a crucial task arising in manufacturing, robotics and many others. Often, constraints on safe inputs are unknown ahead of time, and we only obtain noisy information, indicating how close we are to violating the constraints. Yet, safety must be guaranteed at all times, not only for the final output of the algorithm. 
    
    We introduce a general approach for seeking a stationary point in high dimensional non-linear stochastic optimization problems in which maintaining \emph{safety} during learning is crucial.  
    Our approach called \LBM\,is based on applying stochastic gradient descent (SGD) with a carefully chosen adaptive step size to a logarithmic barrier approximation of the original problem. 
    We provide a complete convergence analysis of non-convex, convex, and strongly-convex smooth constrained problems, with first-order and zeroth-order feedback.
    Our approach yields efficient updates and scales better with dimensionality compared to existing 
    approaches.
    
    We empirically compare the sample complexity and the computational cost of our method with existing safe learning approaches. 
    Beyond synthetic benchmarks, we demonstrate the effectiveness of our approach on minimizing constraint violation in 
    policy search tasks in safe reinforcement learning (RL). \footnote{* Equal supervision}
\end{abstract}

\section{Introduction}
    Many optimization tasks in robotics, manufacturing, health sciences, and finance require minimizing a loss function under constraints and uncertainties. In several applications, these constraints are unknown at the outset of optimization, and one can infer the feasibility of inputs only from noisy measurements. 
    For example, in manufacturing, the learner may want to tune the parameters of a machine. However, they can only observe  noisy measurements of the constraints. 
    Alternatively, in learning-based control tasks, e.g., in robotics, one may want to iteratively collect measurements and improve a pre-trained control policy in new environments. 
    In such cases, during the optimization process, it is crucial to only query points (decision vectors) that satisfy the \textcolor{black}{safety} constraints, i.e., lie inside the feasible set, since querying infeasible points could lead to harmful consequences \citep{kirschner2019adaptive,7487170}.     \textcolor{black}{In such settings, even if the state constraints are known in advance, the learner may only have an approximate model of the true 
    dynamics, e.g., through a simulator or a learned model. This implies that in the control policy space, exact constraints are also unknown, which makes non-violation of safety constraints while learning a challenging and important task. 
    }  In the manufacturing example, one wants to sequentially update the parameters and take measurements of the machine performance while not violating the temperature limit during the learning not to break the machine.    
    In the robotics example, while learning the new policy in the new environment, one wants to perform only safe policy updates, avoiding dangerous situations on the road. 
    This problem is known as \emph{safe learning}. 
    
    In this work, we consider two general settings of safe learning. In the first case, we can only access the objective and constraints from noisy value measurements. This is referred to the zeroth-order (black-box) noisy information setting. In the second case, noisy gradient measurements are also available. This is referred to as the first-order noisy information setting. 
    
    The literature on safe learning has focused mainly on the black-box setting. To compare various methods in this setting, we can estimate their sample and computational complexity. The sample complexity represents the number of oracle queries in total that the learner has to make during the optimization to achieve the specific final accuracy $\e>0$. Computational complexity is the total number of arithmetical operations the algorithm requires to achieve accuracy $\e$. 
    
    In the safe learning case,
a large body of work is based on safe Bayesian optimization (BO) \citep{pmlr-v37-sui15,berkenkamp2020bayesian}. These approaches are typically based upon fitting a Gaussian process (GP) as a surrogate of the unknown objective and constraints functions based on the collected measurements. GPs are built using the predefined kernel function.

Although safe BO algorithms are provably safe and globally optimal, without simplifying assumptions,  BO methods suffer from \emph{the curse of dimensionality} \citep{https://doi.org/10.48550/arxiv.1807.02811,https://doi.org/10.48550/arxiv.1902.10675,https://doi.org/10.48550/arxiv.2103.00349}. 
That is, their sample complexity might depend exponentially on the dimensionality for most  commonly used kernels, including the squared exponential kernel \citep{srinivas12information}. 
	Moreover, their computational cost can also scale exponentially on the dimensionality since BO methods have to solve a non-linear programming (NLP) sub-problem at each iteration, which is, in general, an NP-hard problem. 
	Together with this challenge, their computational cost scales cubically with the number of measurements, which provides a significant additional restriction on the number of measurements.
	These challenges make BO methods harder to employ on medium-to-big scale problems. \footnote{Few works are extending BO approaches to high dimensions, see  \citet{https://doi.org/10.48550/arxiv.1502.05700,kirschner2019adaptive}.}
	
    Thus, our primary motivation is to find an algorithm such that
    {\em 1)} its computational complexity  scales  efficiently to a large number of data points;
    {\em 2)} its sample complexity scales efficiently to high dimensions; and
    {\em 3)} it keeps optimization iterates within the feasible set of parameters with high probability. 
    
    We propose {\em Log Barriers SGD
    (\LBM)}, an algorithm that addresses the safe learning task by minimizing the log barrier approximation of the problem. This minimization is done by using Stochastic Gradient Descent (SGD) with a carefully chosen adaptive step size. We prove safety and derive the convergence rate of the algorithm for the convex and non-convex case (to the stationary point) and demonstrate \LBM's performance compared to other safe BO optimization algorithms on a series of experiments with various scales. 
     
    \paragraph{Our contributions}
	We summarize our contributions  below:
    \begin{itemize}
        \item We propose a unified approach for safe learning given a zeroth-order or first-order stochastic oracle. 
         We prove that our approach generates feasible iterations with high probability and converges to a stationary point. 
         Each iteration of the proposed method is computationally cheap and does not require solving any subproblems.
        \item In contrast to our past work on the log barrier approach \citep{usmanova2020safe}, we develop a less conservative adaptive step size based  on the smoothness constant instead of the Lipschitz constant of the constraints, enabling a tighter analysis.
        
        \item We provide a unified analysis, deriving the convergence rate of our algorithm for the stochastic non-convex, convex, and strongly-convex problems. We establish convergence despite the non-smoothness of the log barrier and the increasingly high variance of the log barrier gradient estimator.
        \item We empirically demonstrate that our method can scale to problems with high dimensions, in which previous methods fail. 
        Moreover, we show the effectiveness of our approach in minimizing constraint violation in policy search in a high-dimensional constrained reinforcement learning (RL) problem.
    \end{itemize}

    \paragraph{Related work} 
    
     Although  first-order stochastic optimization is widely explored \citep{Nemirovsky_Yudin, juditsky:hal-00853911,
     Lan_book},  we are not aware of any work addressing \emph{safe} learning for first-order stochastic optimization. 
    Therefore, even though our work also covers the first-order information case, we focus our review on the zeroth-order (black-box) optimization. Most relevant to our work are two areas: 1) \emph{feasible} optimization approaches addressing smooth problems with {\em known} constraints; 2) existing \emph{safe} approaches addressing smooth 
    \emph{unknown} objectives and constraints (including  linear constraints). Here, by \emph{unknown} constraints, we mean that we only have access to a noisy zeroth-order oracle of the constraints while solving the constrained optimization problem. By feasible optimization approaches, we refer to constrained optimization methods that generate a feasible optimization trajectory. For example, the Projected Gradient Descent and Frank-Wolfe are feasible, whereas dual approaches such as Augmented Lagrangian are infeasible. 

    \textcolor{black}{There also exists another large body of work addressing \emph{probabilistic} or \emph{chance constraints} \citep{ShapDentRusz09}. This line of work aims to solve an optimization problem with probabilistic constraints in the form $\Prob\{F^i(x,\xi)\leq 0\}\geq 1-\de.$ However, the main difference is that in our safe learning task, we aim to satisfy the uncertain constraints with high probability \emph{during the learning process}, not only in the end. Another issue is that the chance constraints problem, in general, is also a complex task with no universal solution. Typically, to address it, one requires either a significant number of (unsafe) measurements (e.g., in the scenario approach) or assumes some knowledge about the structure of the constraints a priori (e.g., in robust optimization). In contrast, in the current work, we propose a way to address the safe learning task without prior knowledge of the structure of the constraints. }
	    
    We summarize the discussion of the algorithms from the past work as well as the best known lower bounds 
    in \Cref{tab:comparison}. All works in  \Cref{tab:comparison} consider one-point feedback, except for \citet{balasubramanian2018zeroth} who consider two-point feedback. \emph{Two-point} feedback allows access to the function measurements with the same noise disturbance in at least two different points, whereas \emph{one-point} feedback cannot guarantee this, and the noise can change at each single measurement. In our current work we also consider one-point feedback due to its generality.  
    Next, we provide a detailed discussion of the past work.
\begin{table*}[h!]
\scriptsize
\centering
    \begin{tabularx}{\textwidth}{Xcccc}
        \toprule
          \textbf{Algorithm} & \textbf{Sample complexity} & \textbf{Computational complexity}  & \textbf{Constraints} & \textbf{Convexity}
         \\
         \midrule
         \citet{bach2016highly} & $O\left(\frac{d^2}{\e^3}\right)$ & projections & known & yes\\
         \midrule
         \citet{bach2016highly} & $O\left(\frac{d^2}{\mu \e^2}\right)$ & projections & known & $\mu$-strongly-convex\\
         \midrule
         \citet{bubeck2017kernel}  
         & $O\left(\frac{d^{3}}{\e^2}\right)$ & samplings from $p_t$ distribution & known & yes \\
	     \midrule \citet{balasubramanian2018zeroth} &  $O\left(\frac{d}{\e^4
			}\right)$
		 & 
		 $O\left(\frac{1}{\e^2}\right)$ LPs & known & no
		 \\
		 \midrule
         \citet{garber2020improved}
         & $O\left(\frac{d^{4}}{\e^4}\right)$ & $O\left(\frac{1}{\e^2}\right)$ LPs & known & yes 
         \\
         \midrule
          \citet{usmanova2019safe}  
         &  $O\left(\frac{d^{3}}{\e^3}\right)$ & $O\left(\frac{1}{\e^2
			}\right)$ LPs & unknown, linear & yes 
         \\
         \midrule
          \citet{fereydounian2020safe}   
         & $\tilde O\left(\frac{d^{2}}{\e^4}\right)$ & $O\left(\frac{1}{\e^2
			}\right)$ LPs & unknown, linear & yes/no 
         \\
         \midrule
          {\citet{berkenkamp2017safe}}
         & $\tilde O\left(\frac{\gamma(d)}{\e^2}\right)$ & $O\left(\frac{\gamma(d)}{\e^2}\right)$ NLPs & unknown & no 
         \\ 
         \midrule
          {\bf This work }
         & $ O\left(\frac{d^{2}}{\e^7}\right)$ & $ O\left(\frac{1}{\e^3}\right)$ gradient steps & unknown & no 
         \\ 
          \midrule
          {\bf This work }
         & $\tilde O\left(\frac{d^{2}}{\e^6}\right)$ & $\tilde O\left(\frac{1}{\e^2}\right)$ gradient steps & unknown & yes 
         \\ 
         \midrule
          {\bf This work }
         & $\tilde O\left(\frac{d^{2}}{\mu\e^5}\right)$ & $\tilde O\left(\frac{1}{\e^2}\right)$ gradient steps & unknown & $\mu$-strongly-convex 
         \\ 
         \midrule
        {\em Lower bound} \citep{shamir2013complexity} 
      & $O\left(\frac{d^{2}}{\e^2}\right)$ & - &  known & yes 
        \\
        \bottomrule
    \end{tabularx}
    \caption{Zeroth-order safe smooth optimization algorithms. 
    Here $\e$ is the target accuracy, and $d$ is the dimension of the decision variable. 
    Many of the cited works provide the bounds in terms of regret, which can be converted to stochastic optimization accuracy. In  SafeOpt, $\gamma(d)$ depends on the kernel, and might be exponential in $d$. All the above works consider one-point feedback, except for \citet{balasubramanian2018zeroth} who consider two-point feedback. 
    In \citet{bubeck2017kernel}, at each iteration the sampling from a a specifically updated distribution $p_t$ can be done in $\text{poly}(d,\log(T))T$-time. Under $\tilde O(\cdot)$, we hide a multiplicative logarithmic factor. }
    \label{tab:comparison}
\end{table*}
	
	\paragraph{Known constraints} 
	\looseness=-1 
	We start with smooth 
    zeroth-order optimization with {\em known} constraints. 
    
    For \emph{convex} problems with known constraints, several approaches address zeroth-order optimization with and without projections. \citet{flaxman2005online} propose an algorithm achieving a sample complexity of $ O(\frac{d^2}{\e^4})$ using projections, where $\e$ is the target accuracy, and $d$ is the dimensionality of the problem.  \citet{bach2016highly} achieve $O\left(\frac{d^2}{\e^3}\right)$ sample complexity for smooth convex problems, and $O\left(\frac{d^2}{\mu\e^2}\right)$ for smooth $\mu$-strongly-convex problems. 
    Since the projections might be computationally expensive, in the projection-free setting, \cite{chen2019projection} propose an algorithm achieving 
    a sample complexity of $O(\frac{d^5}{\e^5})$ for stochastic optimization. 
    \citet{garber2020improved} improve the 
    bound for projection-free methods to 
    $ O(\frac{d^4}{\e^4})$ sample complexity. Instead of projections, both of the above works require solving linear programming (LP) sub-problems at each iteration. 
    \citet{bubeck2017kernel} propose a kernel-based method for adversarial learning achieving  $O(d^{9.5}T^{1/2})$ regret, 
    and conjecture that a modified version of their algorithm can achieve $O(\frac{d^3}{\e^2})$ sample complexity for stochastic black-box convex optimization. This method uses a specific annealing schedule for exponential weights, and is quite complex; at each iteration $t>0$ it requires sampling from a specific distribution $p_t$, which can be done in $\text{poly}(d,\log(T))T$-time. 
	For the smooth and strongly-convex case, \citet{hazan2016variance} propose a method that achieves $O(\frac{d^3}{\e^2}).$ The general lower bound for the convex black-box stochastic optimization $O\left(\frac{d^{2}}{\e^2}\right)$ is proposed by \citet{shamir2013complexity}. 
	 To the best of our knowledge, there is no proposed lower bound for the safe convex black-box optimization with unknown constraints.

	For \emph{non-convex} optimization, \citet{balasubramanian2018zeroth} provide a comprehensive analysis of the performance of several zeroth-order algorithms allowing two-point zeroth-order feedback.\footnote{The difference between one-point and two-point feedback is that two-point feedback allows access to the function with the same noise in multiple points, which is a significantly stronger assumption than the one-point feedback \citep{Duchi2015}.} 
	
	There exist also other classical derivative-free optimization methods addressing non-convex optimization based on various heuristics. One example is the Nelder-Mead approach, also known as simplex downhill \citep{10.1093/comjnl/7.4.308}. To handle constraints it uses penalty functions \citep{luersen2004}, or barrier functions \citep{NM-barriers}. Another example are various evolutionary algorithms \citep{journals/jgo/StornP97, kennedy95particle, 10.1007/978-3-642-83814-9_6, journals/ec/HansenO01}.
	Nevertheless, all of these approaches are based on heuristics and thus do not provide theoretical convergence rate guarantees, at best establishing asymptotic convergence. 
   
    \paragraph{Unknown constraints} 
    \looseness -1 
    There are much fewer works on safe learning for problems with a non-convex objective and \emph{unknown} constraints. A significant line of work covers objectives and constraints with bounded reproducing kernel Hilbert space (RKHS) norm \citep{sui2015safe, berkenkamp2016bayesian}, based on Bayesian Optimization (BO). 
    Also, for the \emph{linear} bandits problem, \citet{amani2019linear} design a Bayesian algorithm handling safety constraints. 
    These works build Bayesian models of the constraints and the objective using Gaussian processes \citep[GP]{10.5555/1162254} and crucially require a suitable GP prior. 
    In contrast, in our work, we {\em do not} use GP models and do not require a prior model for the functions.  
    Additionally, most of these approaches do not scale to high-dimensional problems.  \citet{kirschner2019adaptive} proposes an adaptation to higher dimensions using line search called LineBO, which demonstrates strong performance in safe and non-safe learning in practical applications.
    However, they derive the convergence rate  only for the unconstrained case, whereas for the constrained case, they only prove  safety without convergence. 
    We empirically compare our approach with their method in high dimensions  and demonstrate that our approach can solve problems where LineBO struggles.
    
    From the optimization side, in the case of \emph{unknown} constraints, projection-based optimization techniques or Frank-Wolfe-based ones cannot be directly applied. 
    Such approaches require solving subproblems with respect to the constraint set, and thus the learner requires at least an approximate model of it. 
    One can build such a model in the special case of polytopic constraints. For example, \citet{usmanova2019safe} propose a safe algorithm for convex learning with smooth objective and \textit{linear} constraints based on the Frank-Wolfe algorithm. Building on the above, \citet{fereydounian2020safe} propose an algorithm for both convex and non-convex objective and \textit{linear} constraints. Both these methods consider first-order noisy objective oracle and zeroth-order noisy constraints oracle. 
    
	For the more general case of non-linear programming, there are recent safe optimization approaches based on the interior point method (IPM).  \citet{usmanova2020safe} propose using the log barrier gradient-based algorithm for the non-convex non-smooth problem with zeroth-order information. They show the sample complexity to be $\tilde O\left(\frac{d^{3}}{\e^9}\right)$. Here, we hide a multiplicative logarithmic factor under $\tilde O(\cdot)$. 
	The work above is built on the idea of \citet{hinder2019poly} who propose the analysis of the gradient-based approach to solving the log barrier optimization (in the deterministic case). \textcolor{black}{The first-order approach of \citet{hinder2019poly} has sample complexity of $O(\frac{1}{\e^3})$, compared to which \citet{usmanova2020safe} is much slower due to harder non-smoothness and zeroth-order conditions.}
	In the current paper, we extend the above works to smooth non-convex ($ \tilde O(\frac{d^2}{\e^7})$), convex ($\tilde O(\frac{d^2}{\e^6})$) and strongly-convex ($\tilde O(\frac{d^2}{\e^5})$) problems for both first-order and zeroth-order stochastic information. 
	Note that IPM is a feasible optimization approach by definition. 
    By using self-concordance properties of specifically chosen barriers and second-order information, IPM is highly efficient in solving LPs, QPs, and conic optimization problems.
	However, constructing barriers with self-concordance properties is not possible for unknown constraints. Therefore, we focus on  logarithmic barriers.
	
	\paragraph{Price of safety} 
     To finalize  \Cref{tab:comparison}, compared to the state-of-the-art works with tractable algorithms and \emph{known} constraints \citep{bach2016highly}, we pay a price of an order $\tilde O(\e^{-3})$ in zeroth-order optimization just for the \emph{safety} with respect to unknown constraints both in convex and strongly-convex cases. In the non-convex  case, we pay $ O(d\e^{-3})$ both for safety and having one-point feedback compared to \citet{balasubramanian2018zeroth} considering $2$-point feedback. As for the computational complexity, our method is projection-free and does not require solving any subproblems compared to the above methods. 

\paragraph{Paper organization} We organize our paper as follows. In Section \ref{section:problem}, we formalize the problem, and define the assumptions for the first-order stochastic setting. 
In Section \ref{section:general_approach}, we describe our main approach for solving this problem, describe our main theoretical results about it, and establish its safety.
In Section \ref{section:method_specifications}, we specialize our approach for the non-convex, convex and strongly-convex cases. 
For each of these cases, we also provide the suitable optimality criterion and the convergence rate analysis. 
Then, in Section \ref{section:zeroth-order} we specifically analyse the setting of the zeroth-order information, and derive the sample complexity of all variants of our method in this setting. 
In Section \ref{section:experiments} we compare our approach empirically with other existing safe learning approaches and additionally demonstrate its performance on a high-dimensional constrained reinforcement learning (RL) problem.

\section{Problem Statement}\label{section:problem}
    We consider a general constrained optimization problem:
    \begin{align} \label{problem}
        &\min f^0(x) \tag{P}\\
        &\text{s.t. } f^i(x)\leq 0, i=1,\ldots,m, \nonumber
    \end{align} 
    \looseness -1 where the objective function $f^0: \R^d \rightarrow \R$ and the constraints $f^i: \R^d \rightarrow \R$ are  \emph{unknown}, possibly non-convex functions.

    We denote by $\mathcal X$ the feasible set $\mathcal X := \{x\in\R^d: f^i(x)\leq  0, i \in[m]\},$ where $[m]:=\{1,\ldots,m\}.$ By $Int(\mathcal X)$ we denote the interior of the set $\mathcal X$. 
    By $\|\cdot\|$ we denote the Euclidean $\ell_2$-norm. A function $f: \R^d\rightarrow \R$ is called \textit{$L$-Lipschitz continuous} on $\mathcal X$ if
    $  |f(x) - f(y)|\leq L\|x - y\| ~ \forall x,y\in\mathcal X. $
    It is called \textit{$M$-smooth} on $\mathcal X$ if 
    $  f(x) \leq f(y) + \la \nabla f(y), x-y\ra + \frac{M}{2}\|x-y\|^2 ~ \forall x,y\in\mathcal X. $    
    It is called \textit{$\mu$-strongly convex} on $\mathcal X$ if  
    $ f(x) \geq f(y) + \la \nabla f(y), x-y\ra + \frac{\mu}{2}\|x-y\|^2 ~ \forall x,y\in\mathcal X. $ 
    By $\mathcal L(x,\lambda) := f^0(x) + \sum_{i=1}^m\lambda^i f^i(x)$ we denote the Lagrangian function of a problem $\min_{x\in\R^d} f^0(x)\text{ s.t. } f^i(x)\leq 0~ \forall i \in [m]$, where $\lambda\in\R^m$ is the dual vector.  
    
    Our goal is to solve the \emph{safe learning} problem. That is, we need to find the solution to the constrained problem (\ref{problem})  while keeping all the iterates $x_t$ of the optimization procedure feasible $x_t \in \mathcal X$ with high probability during the learning process.
    Throughout this paper we make the following assumptions: 
\begin{assumption}\label{assumption:2_diameter}
    Let $\mathcal X$ have a bounded diameter, that is, $\exists R > 0$ 
    such that for any $x,y\in \mathcal X$  we have $\|x-y\|\leq R.$
\end{assumption}
\begin{assumption}\label{assumption:1} 
The objective and the constraint functions $f^i(x) $ for $  i \in\{0,\ldots,m\}$ are  $M_i$-smooth and $L_i$-Lipschitz continuous on $\mathcal X$ with constants $L_i,M_i>0$. We denote by $L := \max_{i\in\{0,\ldots,m\} } \{L_i\}$ and  $M := \max_{i\in\{0,\ldots,m\}} \{M_i\}$.
\end{assumption}
The above two assumptions are standard in the optimization literature. Without any assumptions on the constraints, guaranteeing safety is impossible. 
We assume the upper bounds on the Lipschitz and smoothness constants to be known.
\begin{assumption}\label{assumption:3}
    There exists a known starting point $x_0\in \mathcal X$ at which $\max_{i\in[m]} f^i(x_0)\leq -\beta$, for  $\beta>0$.
\end{assumption}
 The third assumption ensures that we have a safe starting point, away from the boundary.  In the absence of such an assumption, even the first iterate might be unsafe.  
\begin{assumption}\label{assumption:mfcq}
   \textcolor{black}{Let $\mathcal I_{\rho}(x) := \{ i\in [m]| -f^i(x)\leq \rho\}$ be the set of $\rho$-approximately active constraints at $x$ with $\rho>0$. For some $\rho \in (0,\frac{\beta}{2}]$ and for any point $x\in \mathcal X$ there exists a direction $s_x\in\R^d: \|s_x\| = 1$, such that $\la s_x, \nabla f^i(x)\ra > l$ with $l>0$, for all $i\in\mathcal I_{\rho}(x)$.}
\end{assumption}
The last assumption is the extended Mangasarian-Fromovitz constraint qualification (MFCQ). The classic MFCQ \citep{MANGASARIAN196737} is the regularity assumption on the constraints, guaranteeing that they have a uniform descent direction for all constraints at a local optimum. Our extended MFCQ guarantees this regularity condition at all points $\rho$-close to the boundary. For the classic MFCQ and further details on our extension to it, please refer to Appendix \ref{A:MFCQ}. This assumption holds for example for convex problems with the constraint set having a non-empty interior, as shown in Section \ref{sec:convex}.

\subsection{Oracle}
\label{ssec:oracle}
Typically in the applications we consider, the information available to the lear                                                      ner is noisy.
For example, one can only observe perturbed gradients and values of $f^i,\forall i = 0,\ldots,m$ at the requested points $x_t$. Therefore, formally we consider access to the first-order stochastic oracle for every $f^i(x)$, providing the pair of value and gradient stochastic measurements:
\begin{align}
    \mathcal O (f^i, x, \xi) = (F^i(x,\xi), G^i(x,\xi)).
\end{align}
Note that the formulation allows (but does not require) that $F^i(x,\xi)$ and $ G^i(x,\xi)$ are correlated. In particular, this formulation allows to define the vector of $\xi = \{(\xi^i_0, \xi^i_1)\}_{i=0,\ldots,m} $ such that each $F^i(x,\xi) =  F^i(x,\xi^i_0)$ and $G^i(x,\xi) =  \textcolor{black}{G}^i(x,\xi^i_1)$. In this formulation, $\{(\xi^i_0, \xi^i_1)\}_{i=0,\ldots,m}$ can be either correlated or independent of each-other. 
The parts of the oracle are given as follows:
\begin{itemize}
    \item[1)] 
    \textbf{Stochastic value $F^i(x,\xi)$.} 
    We assume $F^i(x,\xi)$ is unbiased $$\E [F^i(x,\xi)] = f^i(x),$$  and sub-Gaussian with variance bounded by $\sigma_i^2$, 
    that is,  
    $$\Prob\left\{ |F^i(x,\xi) - f^i(x)|\leq \sigma_i\sqrt{\ln\frac{1}{\de}}\right\} \geq 1-\de,~i\in \{0,\ldots,m\}.$$ 
    \item[2)] 
    \textbf{Stochastic gradient $G^i(x, \xi)$.}  We assume that its bias is bounded by
    $$\|\E G^i(x,\xi) - \nabla f^i(x)\| \leq \hat b_i,$$ where $\hat b_i \geq 0$, and it is sub-Gaussian with the variance such that $\E[ \|G^i(x,\xi) - \E G^i(x,\xi)\|^2] \leq \hat \sigma_i^2$. 
\end{itemize}
    Note that the variances of $F^i(x,\xi)$ and $G^i(x,\xi)$ are fixed and given by the nature of the problem. However, we can decrease these variances by taking several measurements per iteration and replacing $(F^i(x,\xi), G^i(x,\xi))$ with \begin{align} \label{oracles:avg}
    F^i_n(x,\xi):= \frac{\sum_{j=1}^n F^i(x,\xi_j)}{n} \text{ and } G^i_n(x,\xi):= \frac{\sum_{j=1}^n G^i(x,\xi_j)}{n}.
    \end{align}
    In the above, we abuse the notation and replace the dependence $F^i_n(x,\xi_1,\ldots,\xi_n)$ by $F^i_n(x,\xi)$ for simplicity. 
    Then, their variances become respectively such that 
    \begin{align} \label{eq:sigma_n}
    \E[ \|F^i_n(x,\xi) - \E F^i_n(x,\xi)\|^2]
    & \leq \sigma_i^2(n) :=  \frac{\sigma_i^2}{n},\\
    \label{eq:hat_sigma_n}\E[ \|G^i_n(x,\xi) - \E G^i_n(x,\xi)\|^2] 
    & \leq \hat \sigma_i^2(n) : = \frac{\hat \sigma_i^2}{n}.
    \end{align}
    
    Our goal is given the provided first-order stochastic information, to find an approximate solution of problem (\ref{problem}) while not making value and gradient queries outside the feasibility set $\mathcal X$ with high probability. To do so, we introduce the log barrier optimization approach.
\section{General Approach}\label{section:general_approach}
In this section we describe our main approach and provide the main theoretical results of our paper.
\subsection{Safe learning with log barriers}
The main idea of the approach is to replace the original constrained problem (\ref{problem}) by its unconstrained \textit{log barrier} surrogate $\min_{x\in \R^d} B_{\eta}(x)$, where  $B_{\eta}(x)$ and its gradient $\nabla B_{\eta}(x)$ are defined as follows
    \begin{align}\label{eq:LB}
        B_{\eta}(x) & = f^0(x) + \eta\sum_{i=1}^m-\log (-f^i(x)),
        \\
         \nabla B_{\eta}(x) & = \nabla f^0(x) + \eta\sum_{i=1}^m\frac{\nabla f^i(x)}{-f^i(x)}.
     \end{align}
     This surrogate $B_{\eta}(x)$  grows to infinity as the argument converges to the boundary of the set $\mathcal X$, and is defined only in the interior of the set $Int(\mathcal X).$ Therefore, under \Cref{assumption:1,assumption:2_diameter,assumption:3,assumption:mfcq}, a major advantage of this method for the problems we consider is that by carefully choosing the optimization step-size, \emph{the feasibility of all iterates is maintained automatically}.
    We run Stochastic Gradient Descent (SGD) with the specifically chosen step size applied to the log barrier surrogate $\min_{x\in \R^d} B_{\eta}(x)$. 

    The main intuition is that the descent direction of the log barrier pushes the iterates away from the boundary, at the same time converging to an approximate KKT point for the non-convex case, and to an approximate minimizer for the convex case. 
     To measure the approximation in the non-convex case, we define the \emph{$\e$-approximate KKT point} \textbf{($\e$-KKT)}. Specifically,  for  $\e>0$ and a pair $(x,\lambda)$, such point satisfies the following conditions:
    \begin{align}
        &\lambda^i,-f^i(x) \geq 0,~ \forall i \in [m] \tag{$\e$-KKT.1}\label{KKT.1}
        \\
        &\lambda^i(-f^i(x)) \leq \e ,~ \forall i \in [m] \tag{$\e$-KKT.2}\label{KKT.2}
        \\
        &\|\nabla_x \mathcal L(x, \lambda)\| \leq \e. \tag{$\e$-KKT.3}\label{KKT.3}
    \end{align}
    Hereby, $\lambda$ is the vector of dual variables and  $\mathcal L(x,\lambda) := f^0(x) + \sum_{i=1}^m\lambda^i f^i(x)$ is the Lagrangian function of (\ref{problem}).
    Later, we show that SGD on the $\eta$-log barrier surrogate converges to an $\e$-approximate KKT point with $\e = \eta$. We show it by demonstrating that the small barrier gradient norm $\|\nabla B_{\eta}(\hat x)\| \leq \eta$ corresponds to \textcolor{black}{a small} gradient of the Lagrangian $\|\nabla_x \mathcal L(x, \lambda)\|$ with specifically chosen vector of dual variables $\lambda \in\R^m$  \citep{hinder2019poly, usmanova2020safe}. 
    In the convex case, 
    the approximate optimality in the value $ B_{\eta}(\hat x) - B_{\eta}(x_{\eta}^*) \leq \eta$ itself implies that $\hat x$ is an $\e$-approximate solution of the original problem: $f^0(\hat x) - \min_{x\in\mathcal X} f^0(x) \leq \e $ with $\e>0$ linearly dependent  on $\eta$ up to a logarithmic factor. 

\subsection{Main results}
    We propose to apply SGD with an adaptive step-size to minimize the unconstrained log barrier objective $B_{\eta}$. We name our approach \LBM.
    We show that \LBM\, (with confidence parameter $ \de = \hat \de/Tm$) achieves the following convergence results for the target probability $1-\hat\de$:
   \begin{enumerate}
   \item For the non-convex case,  
    after at most $T = O(\frac{1}{\e^3})$ iterations, and with  $\sigma_i(n) = O(\e^{2})$ and $\hat \sigma_i(n) = O(\e),$ \LBM ~ outputs  $x_{\hat t}$ which is an $\e$-KKT point with probability $1-\hat \de$.
    \textcolor{black}{In total,} we require $N = Tn = O(\frac{1}{\e^7})$ oracle queries $\mathcal O(f^i,x,\xi)$ for all $i\in\{0,\ldots,m\}$. (Theorem \ref{thm:non-convex})
       \item For the convex case, after at most $T = \tilde O(\frac{\|x_0 - x^*\|^2}{\e^2}) $ iterations of \LBM, and with  $\sigma_i(n) = \tilde O(\e^{2})$ and $\hat \sigma_i(n) = \tilde O(\e),$
     we obtain output $\bar x_T$ such that with probability $1-\hat\de$: $f^0(\bar x_T) - \min_{x\in\mathcal X}f^0(x) \leq \e $. 
     \textcolor{black}{In total,} we require $N = Tn = \tilde O(\frac{1}{\textcolor{black}{\varepsilon^6}})$ calls of the oracle $\mathcal O(f^i,x,\xi)$ for all $i\in\{0,\ldots,m\}$. (Theorem \ref{thm:convex})
\item 
    For the $\mu$-strongly-convex case,  
    after at most $T = \tilde O(\frac{1}{\mu\e} \log \frac{1}{\e})$ iterations of \LBM\, with decreasing $\eta$, and with $\sigma_i(n) = \tilde O(\eta^{2})$ and $\hat \sigma_i(n) = \tilde O(\eta),$ for the output $\hat x_{K}$ we have with probability $1-\hat \de$: $f^0(\hat x_K) - \min_{x\in\mathcal X}f^0(x) \leq \e $.
     \textcolor{black}{In total,} we require $N = \tilde O\left(\frac{1}{\e^5}\right)$ calls of the oracle $\mathcal O(f^i,x,\xi)$ for all $i\in\{0,\ldots,m\}$. (Theorem \ref{thm:str-convex})
\item For the zeroth-order information case, 
    estimating the function gradients using finite difference, we obtain the following bounds on the number of measurements (Corollary \ref{corollary:zeroth-order}):
    \begin{itemize}
        \item $N = O(\frac{d^2}{\e^7})$ to get an $\e$-approximate KKT point in the non-convex case;
        \item $N = \tilde O(\frac{d^2}{\e^6})$ to get an $\e$-approximate minimizer in the convex case;
        \item $N = \tilde O(\frac{d^2}{\e^5})$ to get an $\e$-approximate minimizer in the strongly-convex case;
    \end{itemize}
    \item In all of the above cases the safety is guaranteed with probability $1-\de$ for all the measurements. (Theorem \ref{thm:safety}, Corollary \ref{corollary:zeroth-order})
   \end{enumerate} 
    In the above, $\tilde O(\cdot)$ denotes $O(\cdot)$ dependence up to a multiplicative logarithmic factor. Note that for zeroth-order information case we only pay the price of a multiplicative factor $d^2$.
\subsection{Our approach} 
    To minimize the log barrier function, we employ SGD using the stochastic first-order oracle providing 
    $(F^i(x,\xi), G^i(x, \xi))$ with an adaptive step size, and derive convergence rate of our methods dependent on the noise level of this oracle.
    At iteration $t$ we make the step in the form:
    \begin{align}\label{eq:procedure}
        x_{t+1} \leftarrow x_t - \gamma_t  g_t,
    \end{align}
    where $\gamma_t$ is a safe adaptive step size, $g_t$ being the log barrier gradient estimator. In \Cref{sec:lb_grad_estimator}, we show how to build the estimator $g_t$ of the log barrier gradient. Following that, in \Cref{sec:ada_gamma}, we explain how to choose $\gamma_t$. 


As mentioned before, the log barrier function is not a smooth function due to the fact that close to the boundaries of $\mathcal X$ it converges to infinity. 
To address non-smooth stochastic problems, optimization schemes in the literature typically require bounded sub-gradients. 
For the log barrier function even this condition does not hold in general. Hence, we cannot expect the classical analysis with the standard predefined step size to hold when applying  SGD to the log barrier problem. 
Contrary to that, by making the step size \emph{adaptive}, we can guarantee \emph{local-smoothness} of the log barrier. Intuitively, this is done by restricting the growth of the constraints. We leverage this property in our analysis. 
In particular, let $\gamma_t$ be such that $f^i(x_{t+1}) \leq \frac{f^i(x_t)}{2}$ for every constraint. Then, the log barrier is locally-smooth at point $x_t$ with constant $M_2(x_t)$
\begin{align} \label{eq:M2}
    M_2(x_t) \leq M_0 + \textcolor{black}{10}\eta \sum_{i=1}^m \frac{ M_i}{\alpha^i_{t}} + \textcolor{black}{8} \eta\sum_{i=1}^m  \frac{(\theta^i_t)^2}{(\alpha^i_{t})^2},
\end{align} 
    where $\theta^i_t = \la \nabla f^i(x_t), \frac{g_t}{\| g_t\|}\ra$, and $\alpha_t^i = -f^i(x_t)$ for all $i\in[m]$. The growth on the constraints can be bounded by any constant in $(0,1)$ -- we pick $\frac{1}{2}$ for simplicity, similarly to \citet{hinder2018one}. In more details, this adaptivity property and the $M_2(x_t)$-local smoothness are analyzed in \Cref{sec:ada_gamma}. Importantly,  our local smoothness $M_2(x)$ bound is more accurate since it is constructed by exploiting the smoothness of the constraints and takes into account the gradient measurements. In contrast, the bound of \citet{hinder2019poly} relies on Lipschitz continuity without considering the gradient measurements.

\subsubsection{The log barrier gradient estimator}\label{sec:lb_grad_estimator}
The key ingredient of the log barrier method together with the safe step size is estimating the log barrier gradient. 
\paragraph{Estimating the gradient} 
     Recall that the log barrier gradient by definition is:
    $$ \nabla B_{\eta}(x_t) = \nabla f^0(x_t) + \eta\sum_{i=1}^m\frac{\nabla f^i(x_t)}{\alpha_t^i}.$$     
Since we only have the stochastic information, we estimate the log barrier gradient as follows:
\begin{algorithm}[H]
    \caption{Log Barrier Gradient estimator $\mathcal O_{\eta}(x_t,  n)$ }
    \label{alg:oracle}
	\small
	\begin{algorithmic}[1]
		\STATE {
		\emph{Input:}} 
		Oracles $F^i(\cdot,\xi),G^i(\cdot,\xi),~\forall i \in \{0,\ldots,m\}$, $x_t \in \mathcal X$, \textcolor{black}{$\eta>0$}, number of measurements $n$ 
		\STATE  
	    $ g_t\leftarrow G^0_n(x_t,\xi_t) + \eta \sum_{i = 1}^{m} \frac{G^i_n(x_t,\xi_t)}{-F^i_n(x_t,\xi_t)};$ 
       \STATE {
      \emph{Output:} $g_t$}
    \end{algorithmic}
\end{algorithm}
In the above, we allow to take a batch of measurements per call and average them as defined in (\ref{oracles:avg}) in order to reduce the variances $ \sigma_i^2(n) :=  \frac{\sigma_i^2}{n},\hat \sigma_i^2(n) : = \frac{\hat \sigma_i^2}{n}.$

\paragraph{Properties of the estimator}
The log barrier gradient estimator defined above is biased and can be heavy tailed, since a part of it is a ratio of two sub-Gaussian random variables. Therefore, in the following lemma we provide a general  upper confidence bound on the deviation.
    We denote $\bar \alpha_t^i := - F^i_n(x_t, \xi_t).$
\begin{lemma}\label{lemma:LBGradEstimatorVarBias}
    The deviation of the log barrier gradient estimator $\Delta_t := g_t - \nabla B_{\eta}(x_t)$  satisfies: 
    \begin{align}\label{eq:delta_bound}
        \Prob\left\{\|\Delta_t\| 
        \leq  \hat b_0 + \hat \sigma_0(n) \sqrt{\ln\frac{1}{\de}} 
        + \sum_{i=1}^m \frac{\eta}{\bar\alpha_t^i}
        \left(\hat b_i + \hat \sigma_i(n) \sqrt{\ln\frac{1}{\de}}\right) + 
        \sum_{i=1}^m L_i \frac{\eta \sigma_i(n) }{\alpha_t^i\bar\alpha_t^i} 
        \sqrt{\ln\frac{1}{\de}}
        \right\}
        \geq 1-\de.
    \end{align}
\end{lemma}
From the above bound, we can see that the closer we are to the boundary, the smaller  $\alpha_t^i$ becomes, and the smaller variance $\sigma_i$ we require to keep the same level of disturbance of the barrier gradient estimator. That is, the closer to the boundary, the more measurements we require to stay safe despite the disturbance, which is quite natural.
For the proof see \Cref{A:var}.

The above deviation consists of the variance part and the bias part. Note that the bias is non-zero even if the biases of the gradient estimators are zero. It can be bounded as follows (see Appendix \ref{appendix:bias}): 
\begin{align}
    \|\E\Delta_t\| \leq \sum_{i=1}^m  \frac{\eta L_i \sigma_i(n)}{(\alpha_t^i)^2}  + \hat b_0 + \sum_{i=1}^m \frac{\eta}{\alpha_t^i}\hat b_i.
\end{align}
In the above,
the expectation is taken 
given fixed $x_t$. 
This bias comes from the fact that we are estimating the ratio of two sub-Gaussian distributions, which is often heavy-tailed and even for Gaussian variables might behave very badly if the mean of the denominator is smaller than its variance \citep{RePEc:spr:stpapr:v:54:y:2013:i:2:p:309-323}. 
This fact influences the SGD analysis, and does not allow getting convergence guarantees with larger noise. Therefore, our algorithm is very sensitive to the noise $\sigma_i(n)$ and might require many samples per iteration to reduce this noise.

\subsubsection{Adaptive step-size $\gamma_t$} \label{sec:ada_gamma}
First of all, recall that the log barrier is  \textcolor{black}{non-smooth}
on $\mathcal X$, since it grows to infinity on the boundary. 
However, we can use the notion of the $M_2(x_t)$-local smoothness, that guarantees smoothness in a bounded region around the current point $x_t$ \textcolor{black}{such that: 
\begin{align}
\label{eq:constr-growth}\{x_{t+1}\in \mathcal X|f^i(x_{t+1})\leq\frac{f^i(x_t)}{2}, x_t\in \mathcal X\} \tag{$\star$}.
\end{align} In particular, assume that $\gamma_t$ is such that it ensures the above condition (\ref{eq:constr-growth}). Let $Y_t$ be a local set around $x_t$ defined as follows: $Y_t := \{y\in \R^d: y = x_t - u g_t, \forall u\in [0,\gamma_t]\}.$ Then, the function $B_{\eta}$ we call $M_2(x_t)$-\emph{locally smooth} around $x_t$ if for any $x,y\in Y_t$ we have
$$\|\nabla B_{\eta}(x) -  \nabla B_{\eta}(y)\|\leq M_2(x_t)\|x-y\|.$$}
\textit{The local smoothness of the Log Barrier $B_{\eta}(x)$ is required for our convergence analysis  of the SGD}.
\paragraph{$M_2(x_t)$-local smoothness constant for the log barrier} 
    We derive our local smoothness constant based on the $M_i$-smoothness of the objective and constraints $f^i$ for $i = 0,\ldots,m$. 
    Compared to the Lipschitz constant-based approach (used in \citet{hinder2019poly, usmanova2020safe}), our way to bound the local smoothness constant $M_2(x_t)$ allows to estimate it more tightly since we use the quadratic upper bound instead of the linear bound. 
\begin{lemma}\label{lemm:M2}
    On the bounded area around $x_t$ within the radius $\gamma_t$ such that the next iterate is restricted by $f^i(x_{t+1})\leq\frac{f^i(x_t)}{2},$  along the step direction $g_t$ 
    the log barrier $B_{\eta}(x_t)$ is locally-smooth with
    \textcolor{black}{
    \begin{align}\label{eq:M2_0}
    M_2(x_t):= M_0 + 2\eta \sum_{i=1}^m \frac{ M_i}{\alpha^i_{t}} + 4 \eta\sum_{i=1}^m  \frac{\la \nabla f^i(x_{t+1}), \frac{g_t}{\|g_t\|} \ra^2}{(\alpha^i_{t})^2}.
    \end{align}
    Moreover, if $\gamma_t\leq \frac{\alpha_t^i}{2|\theta_t^i| + \sqrt{\alpha_t^i M_i}}\frac{1}{\|g_t\|}$,                                                       then
     \begin{align}\label{eq:M2}
    M_2(x_t):= M_0 + 10\eta \sum_{i=1}^m \frac{ M_i}{\alpha^i_{t}} + 8 \eta\sum_{i=1}^m  \frac{(\theta^i_t)^2}{(\alpha^i_{t})^2},
    \end{align}
    where $\theta_t^i = \la \nabla f^i(x_t), \frac{g_t}{\|g_t\|} \ra$.}
\end{lemma}
    For the proof of Lemma \ref{lemm:M2} see \Cref{P2}. \textcolor{black}{As we can see from the next paragraph, the condition on $\gamma_t$ defined in the second part of the lemma is sufficient to ensure a bounded constraints growth (\ref{eq:constr-growth}).}
    In the case with inexact measurements, we have to use lower bounds on $\alpha_t^i$ and upper bounds on $\theta_t^i$. We denote by $\underline{\alpha}_t^i:= \bar{\alpha}_t^i - \sigma_i(n)\sqrt{\ln\frac{1}{\de}} $ a lower bound on $\alpha^i_t:$ $\Prob\{\alpha^i_t\geq \underline{\alpha}^i_t\}\geq 1-\de.$  
    We denote an upper bound on $\theta_t^i$ by $\hat \theta_t^i := |\la G^i_n(x, \xi), \frac{ g_t}{\| g_t\|}\ra| + \hat b_i +\hat \sigma_i(n) \sqrt{\log\frac{1}{\de}}
		    , ~ \forall i\in[m] $ such that $\Prob\{\theta^i_t\leq \hat{\theta}^i_t\}\geq 1-\de.$ 
    Then, an upper bound on $M_2(x_t)$ can be computed as follows
\begin{align}\label{eq:M2_estimate}
    \hat M_2(x_t) = M_0 + \textcolor{black}{10}\eta \sum_{i=1}^m \frac{ M_i}{\underline{\alpha}^i_{t}} + \textcolor{black}{8} \eta\sum_{i=1}^m  \frac{(\hat \theta^i_t)^2}{(\underline{\alpha}^i_{t})^2}.
\end{align}
   
\paragraph{Adaptivity of the step-size}
    In the above, to bound the local smoothness of the log barrier  
    at the next iterate $x_{t+1} = x_t - \gamma_t g_t$, 
    we require 
    to bound the step size $\gamma_t$ \textcolor{black}{in a way that it ensures that the next iterate remains in the set (\ref{eq:constr-growth}), i.e., such that $f^i(x_{t+1})\leq\frac{f^i(x_t)}{2}$.} 
    Automatically, such condition guarantees the feasibility of $x_{t+1}$ given the feasibility of $x_t$.
    
    One way to get the adaptive step size $\gamma_t$ is to use the Lipschitz constants $L_i$ of $f^i$ to bound $\gamma_t$ (see \citet{hinder2019poly, usmanova2020safe}):
    $$\gamma_t \leq \min_{i\in[m]}\frac{- f^{i}(x_t)}{2L_i} \frac{1}{\|g_t\|}.$$ 
    In practice, $L_i$ are typically unknown or overestimated. For example, even in the quadratic case $f^i(x) = \|x\|^2$, $L_i$ depends on the diameter of the set $\mathcal X$, and thus might be very conservative in the middle of the set. Again, we propose to use the smoothness constants $M_i$ for safety instead.\footnote{
    \emph{Firstly,} the Lipschitz constant, even if it is tight, provides the first-order linear upper bound on the constraint growth, whereas using the smoothness constant we can exploit more reliable and tight second order upper bound on the constraint growth. \emph{Secondly,} Lipschitz constant is often much harder to estimate since it might strongly depend on the size of the set. By the same reason, in practice, even for the hard functions modeled by a neural network with smooth activation functions, we can estimate the smoothness parameters, but it is much less clear how to estimate the Lipschitz constants properly.} Of course, smoothness parameter also can be overestimated, and a promising direction for future work is to incorporate problem-adaptive techniques \citep{vaswani:hal-03456663} and approaches to efficiently estimate such constants \citep{fazlyab2019efficient}.

\begin{lemma}\label{lemma:adaptive_gamma} 
    The adaptive safe step size $\gamma_t$ bounded by
    $$\gamma_t \leq  \min_{i\in[m]}\left\{
         \frac{\alpha^i_t}{2|\theta_t^i| + \sqrt{\alpha_t^i M_i}} \right\}
    \frac{1}{\|g_t\|},$$ 
     guarantees $f^{i}(x_{t+1})\leq \frac{f^{i}(x_{t})}{2}.$
\end{lemma}
    The proof is based on the smoothness bound on the constraint growth:
    \begin{align*}
        f^i(x_{t+1})
        & \leq f^i(x_{t}) - \gamma_t \la \nabla f^i(x_t),g_t \ra + \gamma_t^2\frac{M_i}{2}\|g_t\|^2.
    \end{align*}
For the full proof see \Cref{P1}. We illustrate the principle of choosing this adaptive bound on Figure \ref{fig:step-size-illustration}.
\begin{figure}
    \centering
    \includegraphics[width = 0.7\textwidth]{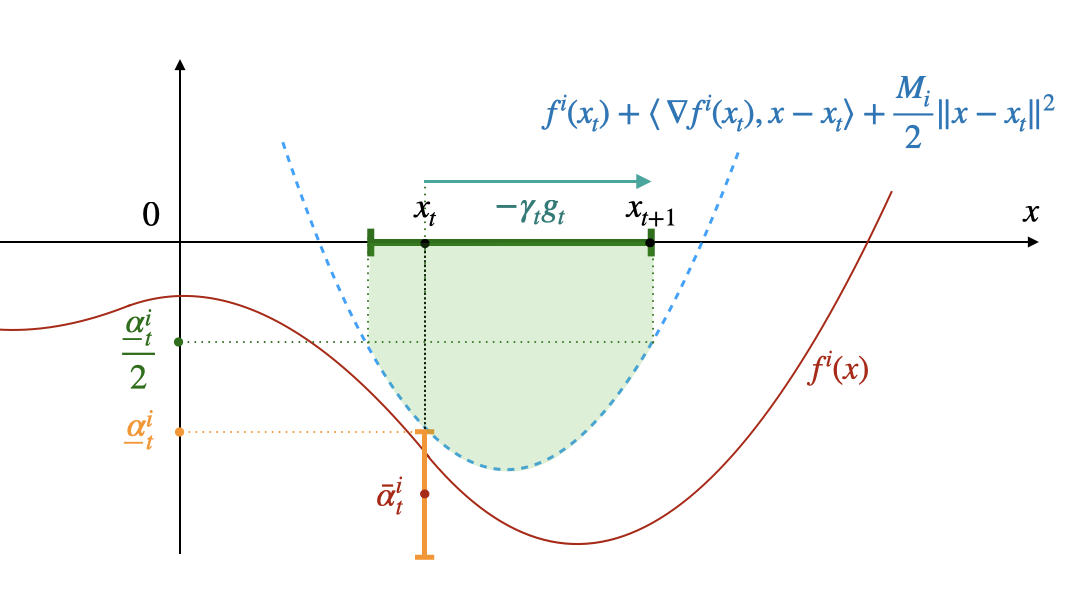}
    \caption{Illustration of the adaptivity. Step size $\gamma_t$ is chosen such that the quadratic smoothness upper bound (\emph{black}) on the constraint guarantees $f^i(x_{t+1}) \leq f^i(x_t)/2$. $\underline \alpha_t^i$ is the lower bound on $\alpha_t^i = -f^i(x_t)$, constructed based on the mean estimator $\bar\alpha_t^i$. By the \emph{orange} interval we denote the confidence interval for $\alpha_t^i$. By the \emph{green} interval we denote the adaptive region for $x_{t+1}$ based on the requirement $f^i(x_{t+1}) \leq f^i(x_t)/2$. }
    \label{fig:step-size-illustration}
\end{figure}
Then, finally, we set the step-size to:
\begin{align}\label{eq:gamma_long}
    \gamma_t =  \min\left\{\min_{i\in[m]}\left\{\frac{\underline{\alpha}^i_t}{2|\hat \theta^i_t| + \sqrt{\underline{\alpha}_t^i M_i}}\right\}
    \frac{1}{\|g_t\|}, \frac{1}{\hat M_2(x_t)}\right\}.
\end{align}

\subsubsection{Basic algorithm}
To sum up, below we propose our basic algorithm, but emphasize that it can be instantiated differently for different problem classes. We showcase possible instantiations of \LBM\, in following  \Cref{section:method_specifications}.
\\ 
\begin{minipage}{\textwidth}
\begin{algorithm}[H]
	\caption{
	 \LBM ($x_0, \textcolor{black}{\eta},  T\textcolor{black}{, n}$)}
	\label{alg:lb_sgd}
	\small
	\begin{algorithmic}[1]
		\STATE \emph{Input:} $ M_i, \sigma_i, \hat \sigma_i, \hat b_i \in\R_+ ~ \forall i \in \{0,\ldots,m\},$ $R \in\R_+ $, \textcolor{black}{$\eta \in\R_+ $,} $n\in \mathbb N$, $T \in \mathbb N$, $\de\in [0,1]$;
	     \FOR {$t = 1,\ldots,T$ } 
		    \STATE  Set $g_t$  $\leftarrow \mathcal O_{\eta}(x_t,  n)$ by taking a batch of measurements of size $n$ at $x_t$;
		    \STATE  Compute lower bounds $ \underline{\alpha}_t^i := \bar{\alpha}_t^i - \sigma_i(n)\sqrt{\ln\frac{1}{\de}}$
		  $, ~ \forall i\in[m]$;
		    \STATE Compute upper bounds $\hat \theta_t^i = |\la G^i_n(x, \xi), \frac{ g_t}{\| g_t\|}\ra| + \hat b_i +\hat \sigma_i(n) \sqrt{\log\frac{1}{\de}}
		    , ~ \forall i\in[m] $;
		    \STATE Compute $\hat M_2(x_t)$ using (\ref{eq:M2});
		    \STATE    $\gamma_t\leftarrow \min\left\{\min_{i\in[m]}\left\{\frac{\underline{\alpha}^i_t}{2|\hat \theta^i_t| + \sqrt{\underline{\alpha}_t^i M_i}}\right\}
    \frac{1}{\|g_t\|},
    \frac{1}{\hat M_2(x_t)}\right\}$;
            \STATE 
		    $x_{t+1} \leftarrow x_{t} - \gamma_t  g_t $;
        \ENDFOR
	\STATE \emph{Output:} $\{x_1,\ldots,x_T\}$.
	\end{algorithmic}
\end{algorithm}
\end{minipage}
\textcolor{black}{In the above, the input parameters are: the smoothness constant $ M_i$ of each function $f^i$ for $ i \in \{0,\ldots,m\}$, the bound on the variance of its value measurements $\sigma^2_i$,  the bound on the variance of its gradient measurements $\hat \sigma^2_i$ and the upper bound on the bias of its gradient measurements $\hat b_i$, the bound on the diameter $R$ of the set $\mathcal X$, the log barrier parameter $\eta$, the number of measurements per iteration $n$, the number of iterations $T$, and the confidence parameter  $\de$.} 
\subsection{Safety} 
From the safety side, the adaptive step-size $\gamma_t$ automatically guarantees the safety of all the iterates due to construction, for any procedure generating the iterations in the form $x_{t+1} = x_t - \gamma_t g_t$ where $\gamma_t$ is bounded by (\ref{eq:gamma_long}).  The feasibility of the optimization trajectory we guarantee with probability at least $1 - \hat\de$ with $\hat \de := mT\de.$
\begin{theorem}\label{thm:safety}
    Let $T>0$ denote the total number of iterations of the form  (\ref{eq:procedure}), and $\hat \de\in(0,1)$ denote the target confidence level. Then, for \LBM~ with parameter $ \de \leq \hat \de/mT$, all the query points $x_t$ 
    are feasible with probability greater than $1-\de$.
\end{theorem}
\begin{proof}
Due to the adaptive step size $\gamma_t$, we have $x_{t}\in \mathcal X$ implies  $x_{t+1} \in \mathcal X$
(see Lemma \ref{lemma:adaptive_gamma}) since $f^{i}(x_{t+1})\leq \frac{f^{i}(x_{t})}{2}$ with probability $1-\de$. 
Then, using  $x_0\in \mathcal X$ and Boole's inequality, we conclude that the whole optimization trajectory $\{x_t\}_{t\in[T]}$ is
feasible with probability at least $1 - mT\de \geq 1- \hat \de$.
\end{proof}
 
\section{Method Variants and Convergence Analysis}\label{section:method_specifications}

First of all, let us show the following general property of the log barrier method, important for the further convergence analysis of any problem type that we discuss. 

\subsection{Keeping a distance away from the boundary}
 Imagine that $\underline{\alpha}_t^i$ becomes $0$ for some iteration $t$ during the learning. That would lead to $\gamma_t = 0$, and the algorithm will stop without converging, since there is no safe non-zero step-size. Moreover, the log barrier gradient at that point simply blows up. However, we can lower bound the step sizes $\gamma_t$ if we can provide a lower bound on $\underline{\alpha}_t^i$ for all $t>0$ during the learning with high probability:
 \begin{lemma}\label{lemma:gamma}
If $\underline{\alpha}_t^i \geq c\eta$ for $c>0$, then we have $\Prob\{\gamma_t \geq C\eta \}\geq 1-\de$ with $C$ defined by
\begin{align}\label{eq:C}
    C := \frac{c}{2L^2(1+\frac{m}{c})}\frac{1}{\max \left\{
         \textcolor{black}{4} + \frac{\textcolor{black}{5}Mc\eta}{L^2}, 1 +  \sqrt{\frac{Mc\eta}{4L^2}} \right\}}.    
    \end{align}
\end{lemma}
The proof is shown in Appendix \ref{Appendix:5}.
 
 Therefore, for convergence, we need to show that our algorithm's iterates $x_t$ do not only stay inside the feasible set, but moreover keep a distance away from the boundary.
 Keeping distance is the key property, guaranteeing the regularity of the log barrier function in the sense of a bounded gradient norm, bounded local smoothness and bounded variance.
 For the exact information case without noise, the adaptive gradient descent on the log barrier is shown to converge without stating this property explicitly \citep{hinder2018one}. 
 However, in the stochastic case, this property becomes crucial for establishing  stable convergence.
 It guarantees that the method pushes the iterates $x_t$ away from the boundary of the set $\mathcal X$ as soon as they come too close to the boundary. We formulate it below.
\begin{lemma}\label{lemma:keeping_distance}
    Let Assumptions 
    \ref{assumption:1} and \ref{assumption:3} hold,  Assumption \ref{assumption:mfcq} hold 
    with $\rho \geq \eta$, 
    and let $\hat \sigma_i(n) \leq \frac{\alpha_t^i L}{\eta\sqrt{\ln \frac{1}{\de}}}$,
    $\hat b_i \leq \frac{\alpha_t^i L}{2\eta}$, and
    $\sigma_i(n) \leq \frac{(\alpha_t^i)^2}{2\eta\sqrt{\ln \frac{1}{\de}}}$. 
    Then, we can show that 
    for all $x_t$  for all iterations $t\in [T]$
    generated by the optimization process $x_{t+1} = x_t - \gamma_t g_t$ the following holds:
    $\Prob\{\forall t \in[T]~ \min_{i\in[m]}\underline{\alpha}_t^i \geq c\eta\}\geq 1-\hat \de$ with  
    $$c := \textcolor{black}{\frac{1}{2}}\left(\frac{l}{4L(2m+1)}\right)^m,$$
    where $l>0$ is defined as in \Cref{assumption:mfcq}.
\end{lemma}
\begin{proof}
First, let us note the following fact demonstrating that the product of the smallest absolute constraint values is not decreasing if $x_t$ is close enough to the boundary.
 \begin{fact}\label{fact:alpha_prod}
    Let Assumptions 
    \ref{assumption:1}, \ref{assumption:3} hold, Assumption \ref{assumption:mfcq} hold with $\rho \geq  \eta$, and let 
     $\hat \sigma_i(n) \leq \frac{\alpha_t^i L}{2\eta\sqrt{\ln \frac{1}{\de}}}$, $\hat b_i \leq \frac{\alpha_t^i L}{2\eta}$, and
    $\sigma_i(n) \leq \frac{(\alpha_t^i)^2}{2\eta\sqrt{\ln \frac{1}{\de}}}$.   If at iteration $t$ we have $\min_{i\in[m]} \alpha_t^i \leq \bar c\eta$ 
    with  $\bar c := \frac{l}{L} \frac{1}{2m+1}$, then, for the next iteration $t+1$ we get 
    $\prod_{i\in \mathcal \I }\alpha_{t+1}^i\geq \prod_{i\in \mathcal \I }\alpha_t^i$ for any $\I : \I_t \subseteq \I$ with $\I_t :=\{i\in[m]:\alpha_t^i\leq \eta\}$ with probability $1-\de$.
\end{fact} 
For the proof see Appendix \ref{Appentix:Proof_fact:alpha_prod}.

Note that if $\min_{i\in[m]} \alpha_t^i \geq \bar c \eta  $ for all $t\geq 0$, then the statement of the Lemma holds automatically. Now, consider a consecutive set of steps $t = \{t_0,\ldots,t_k\}$ on whose $\min_{i\in[m]} \alpha_t^i \leq \bar c\eta$. 
By definition, and using  Fact \ref{fact:alpha_prod}, for any $t\in \{t_0,\ldots,t_k\}$ we have with probability $1-\de$
\begin{align*}
    \prod_{i\in \mathcal I_{t+1}} \alpha_{t+1}^i 
    = \frac{\prod_{i\in \mathcal I_{t} \cup \I_{t+1}} \alpha_{t+1}^i}{ \prod_{i\in \mathcal I_{t} \setminus \I_{t+1}} \alpha_{t+1}^i}
    \geq \frac{\prod_{i\in \mathcal I_{t} \cup \I_{t+1}} \alpha_{t}^i}{ \prod_{i\in \mathcal I_{t} \setminus \I_{t+1}} \alpha_{t+1}^i}
    .
\end{align*}
By induction, applying the above sequentially for all $t \in \{t_0,\ldots,t_k\}$ we can get
\begin{align*}
    \prod_{i\in \mathcal I_{t_k}} \alpha_{t_k}^i 
    \geq 
    \frac{\prod_{i\in \mathcal I_{t_k} \cup \I_{t_{k - 1}} \cup \ldots \cup \I_{t_0}} \alpha_{t_0}^i}{ \prod_{i\in \mathcal I_{t_{k - 1}} \setminus \I_{t_k}} \alpha_{t_k}^i \prod_{i\in \mathcal I_{t_{k-2}} \setminus (\I_{t_k}\cup \I_{t_{k - 1}})}\alpha_{t_{k - 1}}^i\ldots  \prod_{i\in \mathcal I_{t_0} \setminus (\I_{t_k}\cup \ldots \cup \I_{t_1})} \alpha_{t_1}^i
    }
\end{align*}
 with probability $1 - \hat \de$ (using Boole's inequality).

Note that by definition of $\I_t$: $\alpha_t^i \leq \eta$ for $ i \in \I_t$. 
At the same time, due to the step size choice, we have  $\alpha_{t+1}^i\leq 2\alpha_t^i \leq 2 \eta.$ Also, note that the sum of the set cardinalities in the denominator equals to the cardinality of the set $|\I_{t_0}\cup\ldots\cup\I_{t_k}\setminus \I_{t_k}|$. Hence, with probability $1-\hat \de$
\begin{align*}
    \prod_{i\in \mathcal I_{t_k}} \alpha_{t_k}^i  
    \geq \frac
    {\prod_{i\in \mathcal I_{t_k} \cup \I_{t_{k-1}} \cup \ldots \cup \I_{t_0}} \alpha_{t_0}^i}
    { (2 \eta)^{|\I_{t_0}\cup\ldots\cup\I_{t_k}\setminus \I_{t_k}|}
    }.
\end{align*}
Thus, for any $j\in\I_{t_k}$ we get the bound:
\begin{align*}
   \alpha_{t_k}^j  
    \geq \frac
    {\prod_{i\in \mathcal I_{t_k} \cup \I_{t_{k-1}} \cup \ldots \cup \I_{t_0}} \alpha_{t_0}^i}
    { (2\hat c \eta)^{|\I_{t_0}\cup\ldots\cup\I_{t_k}\setminus j|}
    }\geq \frac{(\bar c\eta/2)^{|\I_{t_0}\cup\ldots\cup\I_{t_k}|}}{2^{|\I_0\cup\ldots\cup\I_{t_k}|-1}} \geq \left(\frac{\bar c}{4}\right)^{m}\eta.
\end{align*}
\textcolor{black}{Since $\sigma_i(n)\leq \frac{\alpha_t^i}{2\sqrt{\ln \frac{1}{\de}}}$, we get that $\underline{\alpha}_t^i \geq \frac{\alpha_t^i}{2}\geq  \frac{1}{2}\left(\frac{\bar c}{4}\right)^{m}\eta$ with probability $1-\de$.  }
Using the definition 
$\bar c := \frac{l}{ L} \frac{1}{ 2m + 1},$ we obtain the statement of the lemma.
\end{proof}

\subsection{Stochastic non-convex problems}
For the non-convex problem we analyse \LBM$(x_0\textcolor{black}{, \eta}, T\textcolor{black}{, n})$ with the fixed parameter $\eta$, that uses the stopping criterion \begin{align}\label{stopping}
\|g_t\|\leq 3\eta/4
\end{align} and outputs $x_{\hat t}$ with $\hat t$ corresponding to $\arg\min_{t\in T}\|g_t\|$.
\subsubsection{Stationarity criterion in the non-convex case}
    Similarly to \citet{usmanova2020safe}, we can state that in general case small gradient of the log barrier with parameter $\eta$ leads to an $\eta$-approximate KKT point of the constrained problem. Let us set the pair of primal and dual variables to $(x, \lambda) := \left(x, \left[\frac{\eta}{- f^1(x)},\ldots,\frac{\eta}{- f^m(x)}\right]^T\right)$. Then, it satisfies:
    \begin{align*} 
    &1)~\|\nabla_x \mathcal L(x,\lambda)\| = \|\nabla B_{\eta}(x)\|;\\
    &2)~\lambda^i(-f^i(x)) = \frac{\eta}{- f^i(x)}(-f^i(x)) = \eta;\\
    &3)~\lambda^i \geq 0, -f^i(x) \geq 0, ~ i \in [m].
    \end{align*}
    This insight immediately implies the following Lemma.
    \begin{lemma}\label{crit:1}
        Consider problem (\ref{problem}) under Assumptions 
        \ref{assumption:1}, and  \ref{assumption:3}. Let $\hat x$ be an $\eta$-approximate solution to $\min_{x\in \R^d} B_{\eta}(x)$, the  $\eta$-log barrier approximation  of (\ref{problem}), such that $\|\nabla B_{\eta}(\hat x)\|\leq \eta$. Then, $\hat x$ is an $\eta$-approximate KKT point to the original problem (\ref{problem}).
    \end{lemma}
    Thus, \textcolor{black}{we can use a small log barrier norm to guarantee stationarity for problem 
    (\ref{problem}).}
\subsubsection{Convergence for the non-convex problem}    
Then, we get the following convergence result:
\begin{theorem}\label{thm:non-convex}
    After at most $T$ iterations of \LBM ~with
    $T\leq \textcolor{black}{11}\frac{B_{\eta}(x_{0}) - \min_x B_{\eta}(x)}{C \eta^3},$ 
    and with $\sigma_i(n) = O(\textcolor{black}{\eta^{2}})$, $\hat \sigma_i(n) = O(\textcolor{black}{\eta}),$ and $\hat b_i = O(\textcolor{black}{\eta}),$
    for the output $x_{\hat t}$ with $\hat t = \arg\min_{t\in T}\|g_t\|$  we have
    \begin{align}
    \Prob\left\{\|\nabla B_{\eta}(x_{\hat t})\|\leq \eta\right\}\geq 1-\hat \de
    \end{align}
    Therefore, given $\sigma_i(n) = \frac{\sigma_i}{\sqrt{n}}$ (\ref{eq:sigma_n}) and $\hat \sigma_i(n) = \frac{\hat \sigma_i}{\sqrt{n}}$ ( \ref{eq:hat_sigma_n}), for constant $\hat \sigma_i, \sigma_i$, we require $n = O(\frac{1}{\eta^4})$ oracle calls per iteration, and $N = O(\frac{1}{\eta^7})$ calls of the first-order stochastic oracle in total. 
    Using Lemma \ref{crit:1}, we get that $x_{\hat t}$ is an $\e$-approximate KKT point to the original problem (\ref{problem}) with $\e = \eta$.
\end{theorem}
\paragraph{Remark}\textit{Lower bound in the unconstrained non-safe case.} 
In a well-known model where algorithms access smooth,  non-convex functions through queries to an unbiased stochastic gradient oracle with bounded variance, \citet{arjevani2019lower} prove that in the worst case any algorithm requires at least $\e^{-4}$ queries to find an $\e$ stationary point. Although, they allow $d$ to depend on $\e$. 
Therefore, we "pay" extra $\e^{-3}$ measurements for safety. From the methodology point of view, this happens due to the non-smoothness of the log-barrier on the boundary and the fact that the noise of the barrier gradient estimator is very sensitive to how close the iterates $x_t$ are to the boundary.
\begin{proof}\\
\textcolor{black}{\textbf{\textit{Step 1. Bounding number of iterations $T$.}}} First, let us denote 
$
    \hat\gamma_t := \gamma_t \|g_t\|.
    $
At each iteration of \textcolor{black}{\Cref{alg:lb_sgd}} with the fixed $\eta$ the value of the logarithmic barrier decreases at least by the following value:
    \begin{align}\label{eq:0.0}
      B_{\eta}(x_{t}) - B_{\eta}(x_{t+1}) 
      &
      \overset{\textcircled{\scriptsize 1}}{\geq}  \gamma_t \left\la \nabla B_{\eta}(x_t),  g_t \right\ra - \frac{1}{2} M_{2}(x_t) \gamma_t^2 \|g_t\|^2 \nonumber 
        \\
        &
        = \hat\gamma_t \la \Delta_t , \frac{g_t}{\|g_t\|}\ra + \hat\gamma_t \left(1-\frac{ M_{2}(x_t) \gamma_t}{2} \right)\|g_t\|\nonumber
        \\
        &
      \overset{\textcircled{\scriptsize 2}}{\geq} \frac{1}{2}\hat\gamma_t\|g_t\| - \hat\gamma_t\|\Delta_t\|.
    \end{align}
In the above, $M_2(x)$ is the local smoothness constant that we bound by (\ref{eq:M2}). The first inequality $\textcircled{\scriptsize 1}$ is due to the local smoothness of the barrier. $\textcircled{\scriptsize 2}$ is due to the fact that $\Prob\{\gamma_t\leq \frac{1}{M_2(x_t)}\}\geq 1-\de,$ given $x_t\in Int(\mathcal X)$. 
Summing up the above inequalities (\ref{eq:0.0}) for $t\in[T]$, \textcolor{black}{we obtain the second inequality below}:
    \begin{align*}
       \textcolor{black}{ T \min_{t\in [T]} \hat\gamma_t\left(\frac{1}{2}\|g_t\| - \|\Delta_t\|\right)}
        \leq \sum_{t\in[T]} \hat\gamma_t\left(\frac{1}{2}\|g_t\| - \|\Delta_t\|\right)
        &
        \leq  B_{\eta}( x_{0}) - \min_{x\in\mathcal \mathcal X} B_{\eta}(x).
    \end{align*}
    \textcolor{black}{In the above, the first inequality is due to the fact that the minimum of summands is smaller than any of the summands. Recall that we stop the algorithm as soon as  $\|g_t\|\leq 3\eta/4$, as stated in the beginning of the section in (\ref{stopping}).}
 Hence, for all $T$ iterations $t\in [T]$ with $\|g_t\| \geq 3\eta/4$ we have $\hat\gamma_t\geq 0.75 \eta\gamma_t $. Therefore, we get: 
    \begin{align}\label{eq:T_k}
        \textcolor{black}{T} 
        \textcolor{black}{\leq \frac{B^{\eta}( x_{0}) - \min_x B^{\eta}(x)}{\textcolor{black}{\min_{t\in[T]}}\{ \hat \gamma_t\left(0.5\|g_t\| - \|\Delta_t\|\right)\}}} \leq \frac{B^{\eta}( x_{0}) - \min_x B^{\eta}(x)}{0.75\eta\textcolor{black}{\min_{t\in[T]}}\{ \gamma_t\left(0.5\|g_t\| - \|\Delta_t\|\right)\}}.
    \end{align} 
   We have to obtain the lower bound on the denominator.  
     Using the result of Lemmas \ref{lemma:gamma} and \ref{lemma:keeping_distance},  for all $t\in\textcolor{black}{[T]}$ we have $\gamma_t \geq C\eta$. \\
     \textcolor{black}{\textbf{\textit{Step 2. Bounding $\|\Delta_t\|$.}}} Next, we have to upper bound $\|\Delta_t\|$ with high probability.  \textcolor{black}{Recall from Lemma \ref{lemma:LBGradEstimatorVarBias} (\ref{eq:delta_bound}):
 \begin{align*}
        \Prob\left\{\|\Delta_t\| 
        \leq  b_0 + \hat \sigma_0(n) \sqrt{\ln\frac{1}{\de}} 
        + \sum_{i=1}^m \frac{\eta}{\bar\alpha_t^i}
        \left(\hat b_i + \hat \sigma_i(n) \sqrt{\ln\frac{1}{\de}}\right) + 
        \sum_{i=1}^m L_i \frac{\eta \sigma_i(n) }{\alpha_t^i\bar\alpha_t^i} 
        \sqrt{\ln\frac{1}{\de}}
        \right\}
        \geq 1-\de.
    \end{align*} }
Hence, we can guarantee $\Prob\{ \|\Delta_t\| \leq \frac{\eta}{4} \forall t\}\geq 1-\hat\de$ 
if for all $i\in[m]$ 
\begin{align}\label{eq:sigma_bounds1}
  \hat \sigma_0(n) &\leq \frac{\eta}{4\textcolor{black}{(3m + 2)}\sqrt{\ln\frac{1}{\de}}}  , 
  \hat \sigma_i(n) \leq \frac{\underline{\alpha}_t^i}{4\textcolor{black}{(3m + 2)}\sqrt{\ln\frac{1}{\de}}} , \\
  \hat b_0 & \leq \frac{\eta}{4\textcolor{black}{(3m + 2)}}, 
  \hat b_i \leq \frac{\underline{\alpha}_t^i}{4\textcolor{black}{(3m + 2)}}, 
  \sigma_i(n) \leq \frac{ (\underline{\alpha}_t^i)^2}{4\textcolor{black}{(3m + 2)}L\sqrt{\ln\frac{1}{\de}}},
\end{align}
 (using the Boolean inequality).  
 \textcolor{black}{  Using the lower bound on $\underline{\alpha}_t^i$ by Lemma \ref{lemma:keeping_distance}, we get that for all $i\in[m]$ we require $\hat \sigma_i(n) = O(\eta)$, $\sigma_i(n) = O(\eta^2) $, $\hat b_i = O(\eta).$\\
\textcolor{black}{\textbf{\textit{Step 3. Finalising the bounds.}}} Then, using $\|g_t\|\leq 3\eta/4$ for all $t\in[T]$, we can claim that with high probability for any $t\in[T]$  the following holds $$0.5\|g_t\| - \|\Delta_t\| \geq \frac{3}{8} \eta - \frac{1}{4} \eta = \frac{\eta}{8}.$$}  Combining it with inequality (\ref{eq:T_k}), the algorithm stops after at most $T$ iterations with
    $$T\leq 8\frac{B_{\eta}( x_{0}) - \min_x B_{\eta}(x)}{\textcolor{black}{0.75} C \eta^3} \textcolor{black}{\leq 11\frac{B_{\eta}( x_{0}) - \min_x B_{\eta}(x)}{ C \eta^3}}.$$
    Finally, using $\Prob\{ \|\Delta_t\| \leq \frac{\eta}{4} \forall t\}\geq 1-\hat\de$ and \textcolor{black}{the stopping criterion $\|g_t\| \leq 3\eta/4$} (\ref{stopping}), we obtain 
    \begin{align*}
    \Prob\left\{\|\nabla B_{\eta}( x_{\hat t})\|\textcolor{black}{\leq \|g_{\hat t}\| + \|\Delta_{\hat t}\|} \leq \eta\right\}\geq 1-\hat\de.
    \end{align*}
\end{proof}

\subsection{Stochastic convex problems}\label{sec:convex}
 For the convex case, we propose to use \LBM ($ x_0,\textcolor{black}{\eta},  T\textcolor{black}{, n}$) with the \emph{output:} $\bar x_T:= \frac{\sum_{t=1}^T\gamma_t x_{t}}{\sum_{t=1}^T\gamma_t}$. Next, we discuss the optimality criterion for convex problems.

\subsubsection{Optimality criterion in the convex case}\label{sec:optimality_crit_convex}
    In the convex case, we can relate an approximate solution of the log barrier problem to an $\e$-approximate solution of the original problem in terms of the objective value. 
\begin{assumption}\label{assumption:convexity} 
    The objective and the constraint functions $f^i(x) $ for all $i \in\{0,\ldots,m\}$ are convex.
\end{assumption}
 Note that Assumption \ref{assumption:3} implies non-emptiness on $Int (\mathcal X)$ which is called Slater Constraint Qualification. In the convex setting, it in turn implies the extended MFCQ:
\begin{fact}\label{fact:1}
    Let Assumptions \ref{assumption:2_diameter}, \ref{assumption:3}, and \ref{assumption:convexity} hold. Then, Assumption \ref{assumption:mfcq} holds with $s_x := \frac{x - x_0}{\|x-x_0\|}$, such that 
    $
        \la \nabla f^i(x), s_x \ra \geq \frac{\beta-\rho}{R}
    $ for all $i\in \mathcal I_{\rho}(x)$ for any $0 <\rho < \beta$ .
\end{fact}
\begin{proof}
Indeed, for any point $x\in \mathcal X$ 
 and for any convex constraint $f^i$ such that $f^i(x)\geq - \rho$, due to convexity we  have  
    $
         f^i(x) - f^i(x_0) \leq  \la\nabla f^i(x),   x - x_0\ra.$ Given the bounded diameter of the set   $\|x_0 - x\|\leq R$, we get $ \la\nabla f^i(x),  
    s_x \ra \geq \frac{\beta - \rho}{R}$.
\end{proof}
 Then, we can relate an $\eta$-approximate solution by the log barrier value with an $\e$-approximate solution for the original problem, where $\e$ depends on $\eta$ linearly up to a logarithmic factor. We formulate that in the following lemma:
\begin{lemma}\label{lemma:connection}
    Consider problem (\ref{problem}) under Assumptions \ref{assumption:2_diameter}, \ref{assumption:1}, \ref{assumption:3}, and the convexity Assumption \ref{assumption:convexity}.
    Assume that $\hat x$ is an $\eta$-approximate solution to the $\eta$-log barrier approximation, that is,  
    $$ B_{\eta}(\hat x) - B_{\eta}(x^*_{\eta}) \leq \eta,$$
    where $x^*_{\eta}$ is a solution of the $\min B_{\eta}$ minimization problem, with $\eta \leq \beta/2$. Then, $\hat x$ is an $\e$-approximate solution to the original problem (\ref{problem}) with $\e = \eta(m + 1) + \eta m\log\left(\frac{2mL R\hat\beta}{\eta \beta}\right) $, that is,
    $
        f^0(\hat x) - \min_{x\in \mathcal X} f^0(x) \leq \e,
    $ where $\hat \beta>0$ is such that $\forall i\in[m]\forall x\in\mathcal X~ |f^i(x)|\leq \hat \beta.$ Since the constraints are smooth and the set $\mathcal X$ is bounded, such $\hat \beta$ exists.
    \end{lemma}
    \paragraph{Proof sketch}
      Let $\hat x$ be an approximately optimal point for the log barrier: 
    $B_{\eta}(\hat x) - B_{\eta}(x^*_{\eta}) \leq \eta,$ and $x^*_{\eta}$ be an optimal point for the log barrier. Then, using the definition, we can bound:    
    $
        f^0(\hat x) - f^0( x^*_{\eta})  \leq  \eta + \eta\sum_{i=1}^m -\log\frac{- f^i(x^*_{\eta})}{ - f^i(\hat x)}.
    $
    Combining Fact \ref{fact:1} with the first order stationarity criterion, we can derive:
    $\min_{i\in[m]}\{-f^i(x^*_{\eta})\}  \geq \frac{\eta \beta}{2mLR}.$
    Hence, combining the above two inequalities, we get the following relation of point $\hat x$ and point $x^*_{\eta}$:
        $f^0(\hat x) -f^0( x^*_{\eta})   \leq \eta \left(1 + m\log\left(\frac{2m LR\hat \beta}{\eta \beta}\right)\right)$ using $-f^i(\hat x)\leq\hat\beta $.
    Using the Lagrangian definition for stationarity of the optimal point of the initial problem $x^*$, we get the following relation between $x^*$ and $x^*_{\eta}$: 
       $ f^0(x^*_{\eta}) - f^0(x^*) \leq m \eta.$
   Combining  it with the above, we get the statement of the Lemma
        $$ f^0(\hat x) - \min_{x\in \mathcal X}f^0( x)   \leq \eta + \eta m\log\left(\frac{2mLR \hat\beta}{\eta  \beta}\right) + m\eta.$$
        For the full proof see Appendix \ref{Appendix:proof_lemma:connection}.
    

\subsubsection{Convergence in the convex case}
As already discussed in the optimality criterion \Cref{sec:optimality_crit_convex}, for the convex problem we only require the convergence in terms of the value of the log barrier. 
Thus, we get the following convergence result for this method.
\begin{theorem}\label{thm:convex}
    \textcolor{black}{Let Assumptions \ref{assumption:2_diameter}, \ref{assumption:1}, \ref{assumption:3},  \ref{assumption:convexity} hold,} $B_{\eta}(x):= f^0(x) -\eta \sum_{i=1}^m\log(-f^i(x))$ be a  log barrier function with parameter $\eta>0$, and $x_0\in\R^d$ be the starting point. 
    Let $x^*_{\eta}$ be a minimizer of $B_{\eta}(x)$. Then, after \textcolor{black}{$T \geq \frac{\|x_0 - x_{\eta}^*\|^2}{C\eta^2} $} iterations of \LBM, and with  $\sigma_i(n) = O(\frac{\eta^{2}}{R})$, $\hat \sigma_i(n) =  O(\frac{\eta}{R}),$ and $\hat b_i  =O(\frac{\eta}{R}) ,$ for the point  $\bar x_T:= \frac{\sum_{t=1}^T\gamma_t x_{t}}{\sum_{t=1}^T\gamma_t}$
    we obtain:
     $$\Prob\left\{B_{\eta}(\bar x_{T}) - B_{\eta}(x^*_{\eta}) \leq \eta 
     \right\}\geq 1-\hat\de.
     $$ 
     For the noise with constant variances $\hat \sigma_i, \sigma_i$, given $\sigma_i(n) = \frac{\sigma_i}{\sqrt{n}}$ (\ref{eq:sigma_n}) and $\hat \sigma_i(n) = \frac{\hat \sigma_i}{\sqrt{n}}$ ( \ref{eq:hat_sigma_n}), we require $n = O(\frac{1}{\eta^4})$ oracle calls per iteration, and $O(\frac{1}{\eta^6})$ measurements of the first-order oracle in total. Using Lemma \ref{lemma:connection} we get  $\bar x_T$ is an $\e$-approximate solution to the original problem (\ref{problem}) with $\e = \eta(m + 1) + \eta m\log\left(\frac{2mL R\hat\beta}{\eta \beta}\right) $, that is,
    $
        f^0(\hat x) - \min_{x\in \mathcal X} f^0(x) \leq \e.
    $
\end{theorem}
\begin{proof}
Note the following
\begin{align*} 
    B_{\eta}(x_{t+1}) - B_{\eta}(x^*_{\eta})
    & \overset{\textcircled{\scriptsize 1}}{\leq} B_{\eta}(x_{t}) + \la \nabla B_{\eta}(x_t), x_{t+1} - x_t \ra + \frac{M_2(x_t)}{2}
    \|x_t - x_{t+1}\|^2 - B_{\eta}(x^*_{\eta})
    \\
    & \overset{\textcircled{\scriptsize 2}}{\leq}
    \la \nabla B_{\eta}(x_t), x_{t+1} - x_t \ra + \frac{M_2(x_t)}{2}
    \|x_t - x_{t+1}\|^2 +
    \la \nabla B_{\eta}(x_t), x_t - x^*_{\eta}\ra 
    \\
    & = 
   \frac{M_2(x_t)}{2}
    \|x_t - x_{t+1}\|^2 + \la g_t, x_{t+1} - x^*_{\eta}\ra 
    - \la \Delta_t, x_{t+1} - x^*_{\eta}\ra  
    \\
    & \overset{\textcircled{\scriptsize 3}}{\leq} 
    \frac{\|x_{t} - x^*_{\eta}\|^2}{2\gamma_t} - \frac{\|x_{t+1} - x^*_{\eta}\|^2 }{2\gamma_t}
    - \left(\frac{1}{2\gamma_t} - \frac{M_2(x_t)}{2}\right)\|x_{t+1} - x_t \|^2
     - \la \Delta_t, x_{t+1} - x^*_{\eta}\ra 
          \\
    & \overset{\textcircled{\scriptsize 4}}{\leq} 
     \frac{\|x_{t} - x^*_{\eta}\|^2}{2\gamma_t} - \frac{\|x_{t+1} - x^*_{\eta}\|^2 }{2\gamma_t}
     - \la  \Delta_t, x_{t+1} - x^*_{\eta}\ra.
\end{align*}
The first inequality $\textcircled{\scriptsize 1}$ is due to the $M_2(x_t)$-local smoothness of the log barrier, the second one $\textcircled{\scriptsize 2}$ is due to convexity. The third inequality $\textcircled{\scriptsize 3}$ uses the fact that:
$\forall u \in \R^d : ~ \la g_t, x_{t+1} - u \ra = \frac{\|x_{t} - u\|^2}{2\gamma_t} - \frac{\|x_{t+1} - u\|^2}{2\gamma_t} - \frac{\|x_{t+1} - x_{t}\|^2}{2\gamma_t}.$
And the last one $\textcircled{\scriptsize 4}$ is due to $\gamma_t \leq \frac{1}{M_2(x_t)}$.
By multiplying both sides by $\gamma_t$, we get:
\begin{align}\label{eq1}
        2 \gamma_t (B_{\eta}(x_{t+1}) - B_{\eta}(x^*_{\eta})) 
   & \leq 
         \|x_{t} - x^*_{\eta}\|^2 - \|x_{t+1} - x^*_{\eta}\|^2 
         -  2\gamma_t\la \Delta_t, x_{t+1} - x^*_{\eta}\ra.
\end{align}
Then, by summing up the above for all $t \in [T]$ we get, and using the Jensen's inequality:
\begin{align}\label{eq1}
        B_{\eta}\left(\frac{\sum_{t=1}^T\gamma_t x_{t}}{\sum\gamma_t}\right) - B_{\eta}(x^*_{\eta}) &\leq \frac{1}{\sum_{t=1}^T\gamma_t}\sum_{t=1}^T \gamma_t(B_{\eta}(x_{t}) - B_{\eta}(x^*_{\eta}))\nonumber\\ 
   & \leq 
         \frac{1}{2\sum_{t=1}^T\gamma_t}\sum_{t=1}^T\left(\|x_{t} - x^*_{\eta}\|^2- \|x_{t+1} - x^*_{\eta}\|^2 
         -  2\gamma_t \la \Delta_t, x_{t+1} - x^*_{\eta}\ra\right)\nonumber\\
    & \leq 
         \frac{\|x_{0} - x^*_{\eta}\|^2}{2\sum_{t=1}^T\gamma_t} 
         - \frac{\sum_{t=1}^T \gamma_t \la \Delta_t, x_{t+1} - x^*_{\eta}\ra}{2\sum_{t=1}^T\gamma_t}.
\end{align}
That is, we can bound the accuracy by
\begin{align}
        B_{\eta}\left(\bar x_T\right) - B_{\eta}(x^*_{\eta}) 
        & \leq 
         \frac{\|x_{0} - x^*_{\eta}\|^2}{2\sum_{t=1}^T\gamma_t}  +
         \frac{\max_t\la \Delta_t, x_{t+1} - x^*_{\eta}\ra}{2}.
\end{align}
Using Lemma \ref{lemma:gamma} we can prove for $ \hat \sigma_i(n) \leq  \frac{L\alpha_t^i}{3\eta\sqrt{\ln\frac{1}{\de}}} $  that  $\gamma_t \geq C\eta.$
Recall from Lemma \ref{lemma:LBGradEstimatorVarBias} (\ref{eq:delta_bound}):
 \begin{align*}
        \Prob\left\{\|\Delta_t\| 
        \leq  b_0 + \hat \sigma_0(n) \sqrt{\ln\frac{1}{\de}} 
        + \sum_{i=1}^m \frac{\eta}{\bar\alpha_t^i}
        \left(\hat b_i + \hat \sigma_i(n) \sqrt{\ln\frac{1}{\de}}\right) + 
        \sum_{i=1}^m L_i \frac{\eta \sigma_i(n) }{\alpha_t^i\bar\alpha_t^i} 
        \sqrt{\ln\frac{1}{\de}}
        \right\}
        \geq 1-\de.
    \end{align*} 
Hence, we can guarantee $\|\Delta_t\| \leq \frac{\eta}{R} $ for all $t\in[T]$ with probability $1-\hat\de$ if for all $i\in[m]$ 
\begin{align}\label{eq:var_bounds}
  \hat \sigma_0(n) &\leq \frac{\eta}{(3m + 2)R\sqrt{\ln\frac{1}{\de}}}  , 
  \hat \sigma_i(n) \leq \frac{\underline{\alpha}_t^i}{(3m + 2)R\sqrt{\ln\frac{1}{\de}}} , \\
  \hat b_0 &\leq \frac{\eta}{(3m + 2)R}, 
  \hat b_i \leq \frac{\underline{\alpha}_t^i}{(3m + 2)R}, 
  \sigma_i(n) \leq \frac{ (\underline{\alpha}_t^i)^2}{(3m + 2)LR\sqrt{\ln\frac{1}{\de}}}.\nonumber
\end{align}
 Using the lower bound on $\underline{\alpha}_t^i$ by Lemma \ref{lemma:keeping_distance}, we get that for all $i\in[m]$ we require $\hat \sigma_i(n) = O(\eta)$, $\sigma_i(n) = O(\eta^2) $, $\hat b_i = O(\eta).$\\
Therefore, we get for the $\bar x_T$ the following bound on the accuracy:
\begin{align}\label{eq1}
        B_{\eta}\left(\bar x_T \right) - B_{\eta}(x^*_{\eta}) & \leq 
         \frac{\|x_{0} - x^*_{\eta}\|^2 }{TC\eta} + \frac{\max_t \|\Delta_t\| R }{2}\leq  \frac{\|x_{0} - x^*_{\eta}\|^2}{TC\eta} + \frac{\eta}{2} .
\end{align}
Thus, for $T\geq \frac{\|x_{0} - x^*_{\eta}\|^2}{2C\eta^2}$ we obtain $\Prob\{ B_{\eta}\left(\bar x_T \right) - B_{\eta}(x^*_{\eta}) \leq \eta\}\geq 1-\hat\de$. 
In order to satisfy conditions on the variance (\ref{eq:var_bounds}), we require at each iteration $n = O\left(\frac{1}{\eta^4}\right)$ measurements, and therefore $N = Tn = O(\frac{1}{\eta^6})$ measurements in total.
\end{proof}
\subsection{Strongly-convex problems}
For the strongly convex case, we make use of restarts with iteratively decreasing parameter \textcolor{black}{$\eta_k$}:\\
\begin{minipage}{\textwidth}
\begin{algorithm}[H]
	\caption{
	 \LBM \, with decreasing $\eta_k$ ($\eta\in\R_+, \eta_0\in\R_+ , x_0 \in \R^d, \omega \in (0,1), \{T_k\}, \{n_k\}$)}
	\label{alg:lb_sgd_str_conv}
	\small
	\begin{algorithmic}[1]
		\STATE \emph{Input: }$M_i > 0, i \in \{0,\ldots,m\}, \eta_0\in\R_+ $, $\hat x_0 \leftarrow x_0$, $\textcolor{black}{K \leftarrow [\log_{\omega^{-1}}\frac{\eta_0}{\eta}]}$;
	    \FOR {$k = 1,\ldots,K$ } 
		    \STATE  $\hat x_{k} \leftarrow$ \LBM$(\hat x_{k-1},\textcolor{black}{\eta_{k-1}},  T_{k-1}\textcolor{black}{, n_{k-1}})$;
		    \STATE $\eta_{k}\leftarrow \omega\eta_{k-1},$ with $\omega \in (0,1)$;
        \ENDFOR
	\STATE \emph{Output:} $\hat x_K$.
	\end{algorithmic}
\end{algorithm}
\end{minipage}\\

\textcolor{black}{In the above, $\eta_0$ is the initial log barrier parameter, $\omega\in(0,1)$ is the parameter reduction rate, $K$ is the number of rounds, $ \eta = \eta_K$ is a barrier parameter at the last round,  at every round $k$ we run LB-SGD with $T_k$ iterations and $n_k$ oracle calls per iteration.}
\subsubsection{Convergence}
\begin{theorem}\label{thm:str-convex}
    \textcolor{black}{Let Assumptions \ref{assumption:2_diameter}, \ref{assumption:1}, \ref{assumption:3}, \ref{assumption:convexity} hold,} and the log barrier function with parameter $\eta$: $B_{\eta}(x):= f^0(x) -\eta \sum_{i=1}^m\log(-f^i(x))$ be $\mu$-strongly-convex. 
     Then, after at most $T = \frac{\|x_0 - x^*_{\eta}\|^2}{C\eta_0^2} + \sum_{k=1}^K O(\frac{1}{C\mu\eta_k}) = O\left(\frac{\log_{\textcolor{black}{\omega^{-1}}} \frac{\eta_0}{\eta}}{\mu\eta}\right)$ iterations of \LBM  \, with decreasing $\eta_k$ (\Cref{alg:lb_sgd_str_conv}), and with  
     $\sigma_i(n_k) = O(\textcolor{black}{\eta_k^{2}})
    $, $\hat b_i = O(\textcolor{black}{\eta}),$
      $\hat \sigma_i(n_k) = O(\textcolor{black}{\eta_k}),$
    we obtain:
     $$\Prob\left\{B_{\eta}(\hat x_{K}) - \min_{x\in\mathcal X}B_{\eta}(x) \leq \eta 
     \right\}\geq 1-\hat\de 
     $$ 
    Hence, we require  $n_k = O(\frac{1}{\eta_k^4})$ measurements per iteration at round $k$, and $N = \tilde O(\frac{1}{\eta^5})$ measurements of the first-order oracle in total. Using Lemma \ref{lemma:connection}, we obtain that $\hat x_K$ is an $\e$-approximate solution to the original problem (\ref{problem}) with $\e = \eta(m + 1) + \eta m\log\left(\frac{2mL R\hat\beta}{\eta \beta}\right).$
\end{theorem}
\begin{proof}\\
\textcolor{black}{\textbf{\emph{Step 1. Important relations.}}} Let $x^*_{\eta_k}$ be the unique minimizer of $B_{\eta_k}(x)$.
We do the restarts with decreasing $\eta_{k} = \omega\eta_{k-1} $. 
From strong convexity, for any round $k>0$ we have: 
    \begin{align}\label{eq:2}
    \|\hat x_{k-1} - x^*_{\eta_k}\|^2\leq  R_k^2 := \frac{B_{\eta_k}(\hat x_{k-1}) - B_{\eta_{k}}(x^*_{\eta_k})}{\mu}  .\end{align}
 Note that for all
$ x\in \mathcal X$ we have $B_{\eta_k}(x) \geq B_{\eta_{k-1}}(x)$ (without loss of generality assuming $-f^i(x) \leq \hat \beta \leq 1$). Consequently, $-B_{\eta_k}(x^*_{\eta_k}) \leq -B_{\eta_{k-1}}(x^*_{\eta_{k-1}}).$
Therefore, using the definition of the log barrier, we can get:
\begin{align}\label{eq:3}
    B_{\eta_k}(\hat x_{k-1}) - B_{\eta_k}(x^*_{\eta_k}) & \leq  B_{\eta_{k-1}}(\hat x_{k-1}) - B_{\eta_{k-1}}(x^*_{\eta_{k-1}}) + (\eta_{k-1} - \eta_k) \sum_{i=1}^m -\log -f^i(\hat x_{k-1})\nonumber \\
    & \leq B_{\eta_{k-1}}(\hat x_{{k-1}}) - B_{\eta_{k-1}}(x^*_{\eta_{k-1}}) + m \eta_{k-1}( 1 - \omega)  \log \frac{1}{c\eta_k}.
\end{align}
\textcolor{black}{\textbf{\emph{Step 2. Induction over the rounds.}}}\\
\textcolor{black}{\emph{As an induction base}, we can use Theorem \ref{thm:convex} at the first round $k = 1$, which claims that after $T_0 \geq \frac{R^2}{C\eta_0}$ and $n_0 = O\left(\frac{1}{\eta_0^4}\right)$ 
we get 
    $B_{\eta_{0}}(\hat x_{0}) - B_{\eta_{0}}(x^*_{\eta_0}) \leq \eta_{0}.$ }
\\
\emph{As an induction step}, we assume that for some $k>0$, we have 
$$B_{\eta_{k-1}}(\hat x_{k-1}) - B_{\eta_{k-1}}(x^*_{\eta_{k-1}}) \leq \eta_{k-1}.$$ This in turn, by \cref{eq:3}, leads to 
$B_{\eta_{k-1}}(\hat x_{k-1}) - B_{\eta_{k-1}}(x^*_{\eta_k}) \leq \eta_{k-1}(1 + m(1-\omega)\log \frac{1}{c\eta_k}).$
Combining it with inequality (\ref{eq:2}), we get:
    \begin{align}\label{eq:4}
    R_k^2 \leq \frac{\eta_{k-1}(1 + m(1-\omega)\log \frac{1}{c\eta_k})}{\mu} 
    = \frac{\omega^{-1}\eta_{k}(1 + (1 - \omega)m\log \frac{1}{c\eta_{k}})}{\mu}.
    \end{align}
Then, from Theorem \ref{thm:convex}, by the end of round $k$ we get the induction statement for the next step:
    $$B_{\eta_{k}}(\hat x_k) - B_{\eta_{k}}(x^*_{\eta_{k}}) \leq \eta_{k},$$
for  
$\sigma_{i}(n_k) \leq \frac{c^2\eta^2_k}{(3m + 2)LR\sqrt{\ln\frac{1}{\de}}} $,
\textcolor{black}{$\hat b^i \leq \frac{ c\eta_k }{(3m + 2)R}$}, 
$\hat\sigma_{i}(n_k) \leq \frac{c\eta_k}{(3m + 2)R}\frac{1}{\sqrt{\ln 1/\de}},$ and $T_k = [\frac{R_k^2}{ C\eta_k^2}]$. 
Inserting the result of (\ref{eq:4}) into the bound on $T_k$, we require for $k>0$: $T_k = [\frac{\omega^{-1}+2m(\omega^{-1} - 1)\log \frac{1}{c\eta_k}}{\mu C\eta_k}] $.\\
\textcolor{black}{Thus, under the above conditions, the induction statement holds for all $k>0$, including $k = K$: $$B_{\eta}(\hat x_K) - B_{\eta}(x^*_{\eta}) \leq \eta.$$}
\\
\textcolor{black}{\textbf{\emph{Step 3. Finalising the bounds.}}
Recall that $\sigma^i(n_k) = \frac{\sigma^i}{\sqrt{n_k}}, \hat \sigma^i(n_k) = \frac{\hat \sigma^i}{\sqrt{n_k}}.$
Then, we get that the required amount of measurements per iteration in order to get the required bounds on $\sigma^i(n_k), \hat \sigma^i(n_k)$ is 
$$n_k = \frac{4R^2(3m+2)^2}{\ln 1/\de}\max\left\{\frac{(\hat\sigma^i)^2}{c^2\eta_k^2}, \frac{(\sigma^i)^2L^2}{c^4\eta_k^4}\right\} = O\left(\frac{1}{\eta_k^4}\right) .$$
We require the bias to be \textcolor{black}{$\hat b^i \leq \frac{c\eta}{(3m + 2)R}$} since $\eta \leq \eta_k$ for all $k\geq 0$.} \\
Then, we need the following total number of iterations 
\begin{align*}
    T = T_0 + \sum_{k=1}^K T_k &= \frac{R^2}{C\eta_0^2} + \sum_{k=1}^K \frac{\omega^{-1}+m(\omega^{-1} - 1)\log \frac{1}{c\eta_k}}{\mu C\eta_k} \\
    & 
    \leq \frac{R^2}{C\eta_0^2} + \frac{\omega^{-1}+m(\omega^{-1} - 1)\log \frac{1}{c\eta}}{\mu C\eta} \log\textcolor{black}{_{\omega^{-1}}} \frac{\eta_0}{\eta} 
    = \tilde O\left(\frac{1}{\mu\eta}\right).
\end{align*}
In total, we require $N$ measurements bounded as follows:
\begin{align*}
    \textcolor{black}{N} & = T_0 n_0 + \sum_{k=1}^K T_k n_k \leq \frac{R^2}{C\eta_0^6} +  \sum_{k=1}^K \frac{\omega^{-1}+m(\omega^{-1} - 1)\log \frac{1}{c\eta_k}}{\mu C\eta_k} O\left(\frac{1}{\eta_k^4}\right) \\
    & \leq \frac{R^2}{ C\eta_0^6} + \frac{\omega^{-1}+m(\omega^{-1} - 1)\log \frac{1}{c\eta}}{\mu C\eta^5} \log\textcolor{black}{_{\omega^{-1}}}  \frac{\eta_0}{\eta} = \tilde O\left(\frac{1}{\mu \eta^5}\right).
\end{align*}
\end{proof}

\subsection{Black-box optimization}\label{section:zeroth-order}
A special case of stochastic optimization is zeroth-order optimization, in which one can access only the value measurements of $f^i$. In many applications, for example in physical systems with measurements collected by noisy sensors, we only have access to noisy evaluations of the functions. 
\subsubsection{Stochastic zeroth-order oracle
}
 Formally we assume access to a \emph{one-point stochastic zeroth-order oracle}, defined as follows. For any $i \in \{0,\ldots,m\}$ this oracle provides noisy function evaluations at the requested point $x_j$: 
    $
        F^i(x_j, \xi^i_j) = f^i(x_j) + \xi^i_j, 
    $
where $\xi^i_j$ is a zero-mean $\sigma_i$-sub-Gaussian noise.
We assume that noise values $\xi^i_j$ may differ over iterations $j$ and indices $i$ even for the close points, i.e., we cannot access the evaluations of $f^i$ with the same noise by two different queries: $\xi^i_j \neq \xi^i_{j+1}$ for any $F^i(x_j, \xi^i_j)$ and $ F^i(x_{j+1}, \xi^i_{j+1})$ even if $x_j = x_{j+1}.$ Also, we assume that the noise
vectors $\xi_j$ of the
measurements taken around the same point 
are i.i.d.~random variables.

\subsubsection{Zeroth-order gradient estimator
}
\label{sec:oracle} 
One way to tackle  zeroth-order optimization is to sample a random point $x_t + \nu s_t$ around $x_t$ at iteration $t$, and approximate the stochastic gradient $G^i(x,\xi)$ using finite differences. A classical choice of the sampling distribution is the Gaussian distribution, referred to as Gaussian sampling. However, since the Gaussian distribution has infinite support, one has an additional risk of sampling a point in the unsafe region arbitrarily far from the point, which is inappropriate for \emph{safe learning.} 
Therefore, we propose to use the uniform distribution $\mathcal U(\mathbb S^d)$ on the unit sphere for sampling. In particular, in the case where we only have access to a noisy zeroth-order oracle, we estimate the gradient in the following way. 

We need to estimate the descent directions of $f^i$ using the zeroth-order information. 
For any point $x$, we can estimate the gradient of the function $\nabla f^i$ by sampling  directions $s_{j}$ uniformly at random on the unit sphere $ s_{j} \sim \mathcal U(\mathbb S^d)$, and using the finite difference as follows:
    \begin{align}\label{eq:gradient_estimator}
	    &G^i_{\nu,n}(x,\xi) := \frac{d}{n}\sum_{j = 1}
	    ^{n}
	    \frac{F^i(x+ \nu s_{j}, \xi_{j}^{i+}) - F^i(x, \xi_{j}^{i-})}{\nu} s_{j},
	\end{align}
where 
$\xi^{i\pm}_{j}$ are sampled from $\sigma_i$-sub-Gaussian distribution. Note that $s_j$ also satisfy the sub-Gaussian condition. 
\footnote{There is also an option of using the one-point estimator $G^i_{\nu,n}(x, \xi) := \frac{d}{n}\sum_{j = 1}^{n}\frac{F^i(x+ \nu s_{j}, \xi_{j}^{i+}) }{\nu} s_{j}$, but the variance of this estimator might be much higher. 
Note that even with zero-noise its variance grows to infinity while $\nu\rightarrow 0$.  
Its variance would depend on $\frac{\max_{x\in{D}} |f^i(x)|}{\nu}$, while the two-point estimator's variance depends on the Lipschitz constant $L_i$, which might be significantly smaller. 
Also, in the case of differentiable $f^i$ with small noise $\xi$ the two-point estimator becomes a finite difference directional derivative estimator with the accuracy dependent on $\nu$ only, in contrast to the one-point estimator.  
} 

There are also several other ways to sample directions to estimate the gradient from finite-differences.
\citet{Berahas_2021} compared various zeroth-order gradient approximation methods and showed that their sample complexity has a similar dependence on the dimensionality $d$ required for a precise gradient approximation. 
Deterministic coordinate sampling requires fewer samples due to smaller constants. 
However, we stick with sampling on the sphere because deterministic coordinate sampling requires the number of samples to be divisible by $d$. We want to keep flexibility on how many samples we can take per iteration; this number might be provided by the application. However, we note that any other sampling procedure can also be used. 

Then, the estimator $G^i_{\nu,n}(x,\xi)$ defined above is a biased estimator of the gradient $\nabla f^i(x)$  and an unbiased estimator of the smoothed function gradient $ \nabla f^i_{\nu}(x)$. The smoothed approximation $ f^i_{\nu}$ of each function $f^i$ is defined 
 as follows:
\begin{definition}\label{def:smoothed}
    The $\nu$-smoothed approximation of the function $f(x)$ is defined by
    $f_{\nu}(x):= \E_{b}f(x + \nu b),$ 
    where  $b$ is uniformly distributed in the unit ball $\mathbb B^d$, and  $\nu\geq 0$ is the sampling radius.
\end{definition}
    \begin{algorithm}[H]
        \caption{Zeroth-order gradient-value estimator $(F^i_n(x, \xi), G^i_{\nu,n}(x, \xi))$}
        \label{alg:oracle}
    	\small
    	\begin{algorithmic}[1]
    		\STATE {
    		\emph{Input:}} 
    	 $F^i(\cdot,\xi),i \in\{0,\ldots,m\}$, $x \in \mathcal X$, 
    		$\nu>0$,
    		$n\in\mathbb N$;
            \STATE Sample $n$ directions $s_j \sim \mathcal U(\mathbb S^d)$, sample $F^i(x
            + \nu s_{j}, \xi_{j}^{i+})$ and $ F^i(x, \xi_{j}^{i-})$,  $j \in [n]$;
            \STATE \emph{Output:}   
            \begin{align*}
             F^i_n(x,\xi) &:= \frac{\sum_{j=1}^n F^i(x,\xi_{j}^{i-})}{n} \\ 
	        G^i_{\nu,n}(x, \xi) &:= \frac{d}{n}\sum_{j = 1}^{n}
	        \frac{F^i(x+ \nu s_{j}, \xi_{j}^{i+}) - F^i(x, \xi_{j}^{i-})}{\nu} s_{j}
	    \end{align*}
        \end{algorithmic}
    \end{algorithm}

\begin{lemma} \label{lemma:G}
    Let $f^i_{\nu}(x)$ be the $\nu$-smoothed approximation of $f^i(x).$ Then $\E G^i_{\nu,n}(x,\xi) = \nabla f^i_{\nu}(x)$, where the expectation is taken over both $s_{j}$ and $\xi^{i\pm}_{j}$ for all $j\in [n]$. 
\end{lemma} 
\begin{proof}
 First note that $\E G^i_{\nu,n}(x,\xi) = \underbrace{\E \frac{d}{n}\sum_{j = 1}^{n}
	    \frac{f^i(x+ \nu s_{j})  - f^i(x)}{\nu} s_{j}}_{(1)} +  \underbrace{\E \frac{d}{n}\sum_{j = 1}^{n}\frac{  \xi_{j}^{i+} -\xi_{j}^{i-}}{\nu}s_j}_{(2)}$.
	    Recall that $\xi_{j}^{i\pm}$ are independent on $s_j$ and zero-mean, hence $(2) = 0$. The proof that $(1) = \nabla f^i_{\nu}(x)$ is classical \citep{flaxman2005online}
 and is based on  Stokes' theorem.
\end{proof}

The following lemma shows  important properties of the above zeroth-order gradient-value estimators.
    \begin{lemma}\label{lemma:zero-order-bias-var}
        Let $F^i(x,\xi)$ have variance $\sigma_i>0$ and let the estimator $G^i_{\nu,n}(x,\xi)$ be defined as in (\ref{eq:gradient_estimator}) by sampling $s_j$ uniformly from the unit sphere $\mathcal U(\mathbb S^d)$, then $f^i_{\nu}(x)$, and $G^i_{\nu,n}(x,\xi)$ are biased approximations of $f^i(x)$ and $\nabla f^i(x)$ respectively, such that 
        $$ |f^i(x) - f^i_{\nu}(x)| \leq  \nu^2 M_i,$$ 
        the variance of $F^i_n(x,\xi)$ is 
        $\sigma_i(n) = \frac{\sigma_i}{\sqrt{n}}$
        and the bias of $G^i_{\nu,n}(x,\xi)$ is bounded by: 
        \begin{align}\label{eq:bias}
        \hat b_i := \|\nabla f^i(x) - \nabla f^i_{\nu}(x)\| \leq  \nu M_i, ~\forall i\in \{0,\ldots,m\}.
        \end{align}
        The variance of $G_{\nu,n}^i(x,\xi)$ is bounded as follows:
        \begin{align}\label{eq:variance}
        \hat \sigma_i^2(n) := \E \|G^i_{\nu,n}(x,\xi) - \nabla f^i_{\nu}(x)\|^2 
         \leq   
         \frac{3}{n}\left(d\|\nabla f^i(x)\|^2 +  \frac{d^2 M_i^2 \nu^2}{4}\right) + 4\frac{ d^2}{n} \frac{\sigma^2_i}{\nu^2 }
        ~\forall i\in \{0,\ldots,m\}.
        \end{align}
    \end{lemma}
    \begin{proof}
        These properties are corollaries from \citet{Berahas_2021}.
        For the bias (\ref{eq:bias}) we use the result of
         Equation (2.35) \citep{Berahas_2021}, and for the variance (\ref{eq:variance}) the result of Lemma 2.10 of the same paper, in both cases  by setting the disturbance $\epsilon_f = 0$ in \citet{Berahas_2021}. The last term of the variance is coming from the additive noise. 
         We set the disturbance $\epsilon_f$ to zero for their formulation and analyze the noise separately since they consider the disturbance without any assumptions on it. In contrast, we consider the zero-mean and sub-Gaussian noise which we can use explicitly. 
         For further discussions and proof, see \Cref{Appendix:zero-order-estimator}.
    \end{proof}

\subsubsection{Setting the sample radius $\nu$ and bounding the sample complexity} 
The parameters of the estimator defined in \Cref{alg:oracle} that we can control are $\nu$ and $n$. We want to set them in such a way that the biases $b_i$ and variances $\sigma_i,\hat\sigma_i$ satisfy requirements of Theorems \ref{thm:non-convex}, \ref{thm:convex}, \ref{thm:str-convex}. Based on them, we can bound the sample complexity of our approach for zeroth-order setting.

According to Theorems \ref{thm:non-convex}, \ref{thm:convex}, \ref{thm:str-convex}, we require the bias to be bounded by $\hat b_i \leq \frac{\underline{\alpha}_t^i}{(3m+2)R}, \; \hat b_0 \leq \frac{\eta}{(3m+2)R}.$ Therefore, since $\hat b_i \leq  \nu M_i$, we need to set the sampling radius small enough $\nu \leq \min \left\{\frac{\underline{\alpha}_t^i}{2m M_i R}, \frac{\eta}{2 m M_0 R}\right\}$. 
Moreover, in order to guarantee \emph{safety} of all the measurements within the sample radius $\nu$ around the current point $f^i(x_{t} + \nu  s_t)\leq 0$ using the smoothness of each constraint
   \begin{align*}
        f^i(x_{t} + \nu s_t)
        & \leq f^i(x_{t}) - \nu \la \nabla f^i(x_t),s_t \ra + \nu^2\frac{M_i}{2}\|s_t\|^2,
    \end{align*}
we require the sample radius to be $\nu\leq \frac{\underline{\alpha}_t^i}{2 \|\nabla f^i(x_t)\| + \sqrt{\underline{\alpha}_t^i M_i}}$. This bound can be obtained using the same derivations as for the adaptive step size $\gamma_t$ (Lemma \ref{lemma:adaptive_gamma}). Hence, we set
$$\nu = \min\left\{ \frac{\underline{\alpha}_t^i}{2 \|\nabla f^i(x_t)\| + \sqrt{\underline{\alpha}_t^i M_i}}, \frac{\underline{\alpha}_t^i}{2mM_iR}, \frac{\eta}{2mM_0} \right\} = O(\eta)=\Omega(\eta).$$ 
Thus, from the above Lemma \ref{lemma:zero-order-bias-var} Equation \ref{eq:variance}, the variance of the estimated gradient with $\nu = O(\e) = \Omega(\e)$ is 
\begin{align}\label{eq:hatsigma}
\hat \sigma_i^2(n) =
\frac{1}{n}O\left(\max\left\{\frac{d^2 \sigma_i^2}{\e^2 }, L_i^2, d^2 M_i^2\e^2 \right\}\right).
\end{align}

Additionally, according to the previous Theorems \ref{thm:non-convex}, \ref{thm:convex}, \ref{thm:str-convex}, we require the variances to be $\hat \sigma_i(n) = O(\e)$ and  $\sigma_i(n) = O(\e^2)$. 
From the above Equation \ref{eq:hatsigma}, in order to have $\hat \sigma_i(n) = O(\e)$  we require $n = O\left(\max\{\frac{d^2\sigma_i^2}{\e^4}, \frac{L^2}{\e^2}, d M^2 \}\right).$ 
From the properties of the zero-mean noise, to have $\sigma_i(n) = O(\e^2)$ we require $n = O(\frac{\sigma_i^2}{\e^4}).$
Thus, we can prove the following corollary of the previously proven Theorems \ref{thm:non-convex}, \ref{thm:convex}, \ref{thm:str-convex} particularly for the zeroth-order information case:
\begin{corollary}\label{corollary:zeroth-order}
We get the following sample complexities for the zeroth-order information case, using $\nu = \min\left\{ \frac{\underline{\alpha}_t^i}{2 \|\nabla f^i(x_t)\| + \sqrt{\underline{\alpha}_t^i M_i}}, \frac{\underline{\alpha}_t^i}{2mM_iR}, \frac{\eta}{2mM_0} \right\}$ :
\begin{itemize}
    \item For the non-convex problem, \LBM \, returns $x_t$ such that is $\e$-approximate KKT point after at most $N = O(\frac{d^2 \sigma_i^2}{\e^7})$ measurements with probability $1-\hat \de$.
    \item For the convex problem, \LBM \, returns $x_t$ such that $\Prob\{f^0(x_t) - \min_{x\in\mathcal X}f^0(x) \leq \e\}\geq 1-\hat \de$ 
    after at most $N = \tilde O(\frac{d^2\sigma_i^2}{\e^6})$ measurements.
    \item For the strongly-convex problem, \LBM\, returns $x_t$ such that $\Prob\{f^0(x_t) - \min_{x\in\mathcal X}f^0(x) \leq \e\}\geq 1-\hat \de$
     after at most $N = \tilde O(\frac{d^2\sigma_i^2}{\e^5})$ measurements.
    \item Moreover, all the query points of \LBM\,  are feasible for (\ref{problem}) with probability at least $1-\hat \de$.
\end{itemize}
\end{corollary}

\section{Experiments}\label{section:experiments}
In this section, we demonstrate the empirical performance of our method when optimizing synthetic functions, as well as on a complex case study in constrained reinforcement learning. 

\emph{Numerical stability.}
    First, we note that to improve numerical stability, we slightly modify the steps of our method for practical applications. 
    Recall that the log barrier gradient estimator is $ g_t\leftarrow G^0_n(x_t,\xi_t) + \eta \sum_{i = 1}^{m} \frac{G^i_n(x_t,\xi_t)}{-F^i_n(x_t,\xi_t)}.$  Due to noise, the value of $-F^i_n(x_t,\xi_t)$ might become infinitely close to zero or negative, which leads to $g_t$ blowing up or being unreliable. Therefore, 
    we denote by $\bar \alpha_t^i$ the truncated value measurements $-F^i_n(x_t,\xi_t)$ with small truncation parameter $a>0$, that is $\bar \alpha_t^i = [- F^i_n(x_t,\xi_t)]_a := \begin{cases}
    -F^i_n(x_t,\xi_t) &,\, -F^i_n(x_t,\xi_t) > a\\
    a & ,\, -F^i_n(x_t,\xi_t) \leq a
    \end{cases}$. 
    Based on the above, we use the following estimator for the first-order stochastic optimization at point $x_t$:
    $ g_t = G^0_n(x_t,\xi_t) +\eta\sum_{i=1}^m \frac{G^i_n(x_t,\xi_t)}{\bar \alpha_t^i}.$
\subsection{Safe black-box learning}
In this section, we demonstrate the performance of our method on simulations and compare it to other existing non-linear safe learning approaches.  All the experiments in this subsection were carried out on a Mac Book Pro 13 with 2.3 GHz Quad-Core Intel Core i5 CPU and with 8 GB RAM. The code corresponding to the experiments in this subsection can be found under the following link: \url{https://github.com/Ilnura/LB_SGD}.
    \subsubsection{Convex objective and constraints}
     We first compare our safe method \LBM\, with SafeOpt \citep{sui2015safe,berkenkamp2016bayesian} and LineBO \citep{kirschner2019adaptive}, on a simple synthetic example. 
     
     We consider the quadratic problem with linear constraints $\min_{x\in\R^d} \|x - x_0\|^2/4d \text{, s.t. } Ax \leq b$, where $x_0 = [2,\ldots,2]$ and 
     $A = \begin{bmatrix} I_d \\ 
                         -I_d
           \end{bmatrix}$, 
    $b = \textbf{1}/\sqrt{d}$. The optimum of this problem is on the boundary. We assume that the linearity of the constraints is unknown, hence for SafeOpt we use the Gaussian kernel. For dimensions $d=2,3,4$ we carry out the simulations with standard deviation $\sigma = 0.001$ of an additive noise  \Cref{fig:LBM_Safeopt_convex} averaged over $10$ different experiments. For $d=2$ we run SafeOpt, and for $d=3,4$ we run SafeOptSwarm, which is a heuristic making SafeOpt updates more tractable for slightly higher dimensions \citep{berkenkamp2016bayesian}.
    For SafeOpt and LineBO methods, instead of plotting  the accuracy and constraints corresponding to $x_t$, we plot the smallest accuracy and biggest constraint seen up to the step $t$ (for sake of interpretability of the plots).
     \begin{figure}
         \centering
         \includegraphics[trim=5 5 5 5,clip,width=1.\textwidth]{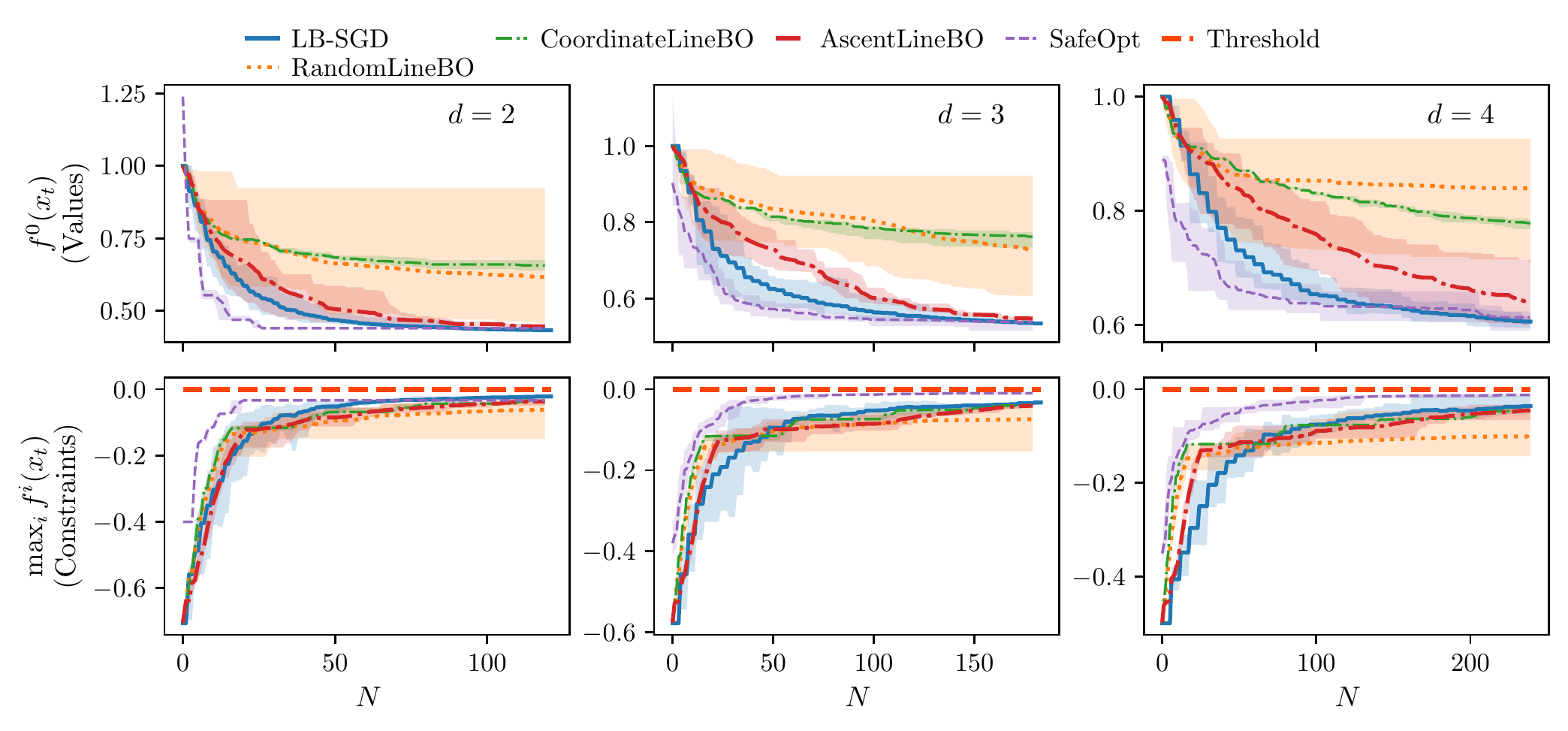}
         \caption{Accuracy (upper plots) and constraints (lower plots) of \LBM\, and SafeOpt for $d = 2,3,4$, averaged over $10$ samples. $t$ here is the amount of zeroth-order oracle calls. In these experiments, for \LBM\, we decrease $\eta_{k+1} = 0.7 \eta_k$ gradually every $T_k = 7$ steps with $n_k = [\frac{d}{2}]$ value measurements at each step. Already for $d = 4$ \LBM\, starts outperforming all BO-based methods on this problem in terms of the sample complexity. Best viewed in color.}
         \label{fig:LBM_Safeopt_convex}
     \end{figure}
     Even for $d = 4$, \LBM\, is already notably more sample efficient compared to both SafeOpt and LineBO. Moreover, \LBM\, significantly outperforms SafeOpt over computational cost and memory usage. 
      It is well known that 
      SafeOpt's sample complexity and computational cost can exponentially depend on the dimensionality. In contrast, the complexity of LB-SGD depends on $d$  polynomially. 
      The runtimes of the above experiments, in seconds, are shown in 
      Table  \ref{tab:runtimes_convex} and 
      Figure \ref{fig:runtimes_non-convex}.
      \begin{table}[H]
          \centering
        \begin{tabular}{|c|c|c|c|}
          \hline
          $d$ & 2 & 3 & 4\\
          \hline
              SafeOpt (SafeOptSwarm) &  4.289 & 114.406 & 212.514\\
           \hline
               LineBO & 8.180 & 17.837 & 40.8\\
            \hline
                LB-SGD & 0.429 &  0.895 &  0.781 \\
            \hline
          \end{tabular}
          \caption{Average runtime dependence on dimensionality $d$ (in seconds). Importantly, the wall-clock time of \LBM\ remains roughly constant as we increase $d$, noting that we only increase the number of measurements per iteration, but not the number of iterations.}
          \label{tab:runtimes_convex}
      \end{table}

\subsubsection{Non-convex objective and constraints}
    As a non-convex example, we consider the Rosenbrock function, a common benchmark for black-box optimization, with quadratic constraints. In particular, we consider the following problem 
    \begin{align*}
    & \min_{x\in\R^d} \sum_{i=1}^{d-1} 100 \|x_i - x_{i+1}\|^2 - \|1 - x_i\|^2,  \\
    & \text{s.t. } \|x\|^2 \leq r_1^2,\, \|x - \hat x\|^2 \leq r_2^2.
    \end{align*}
    We set $r_1 = 0.1$, $r_2 = 0.2$, $\hat x = [-0.05,\ldots,-0.05].$ The optimum of this problem is on the boundary of the constraint set. We show the comparison of \LBM\, and SafeOpt on
    \Cref{fig:LBM_Safeopt_non-convex}. 
        \begin{figure}[t]
         \centering 
        \includegraphics[width = 1.0\textwidth]{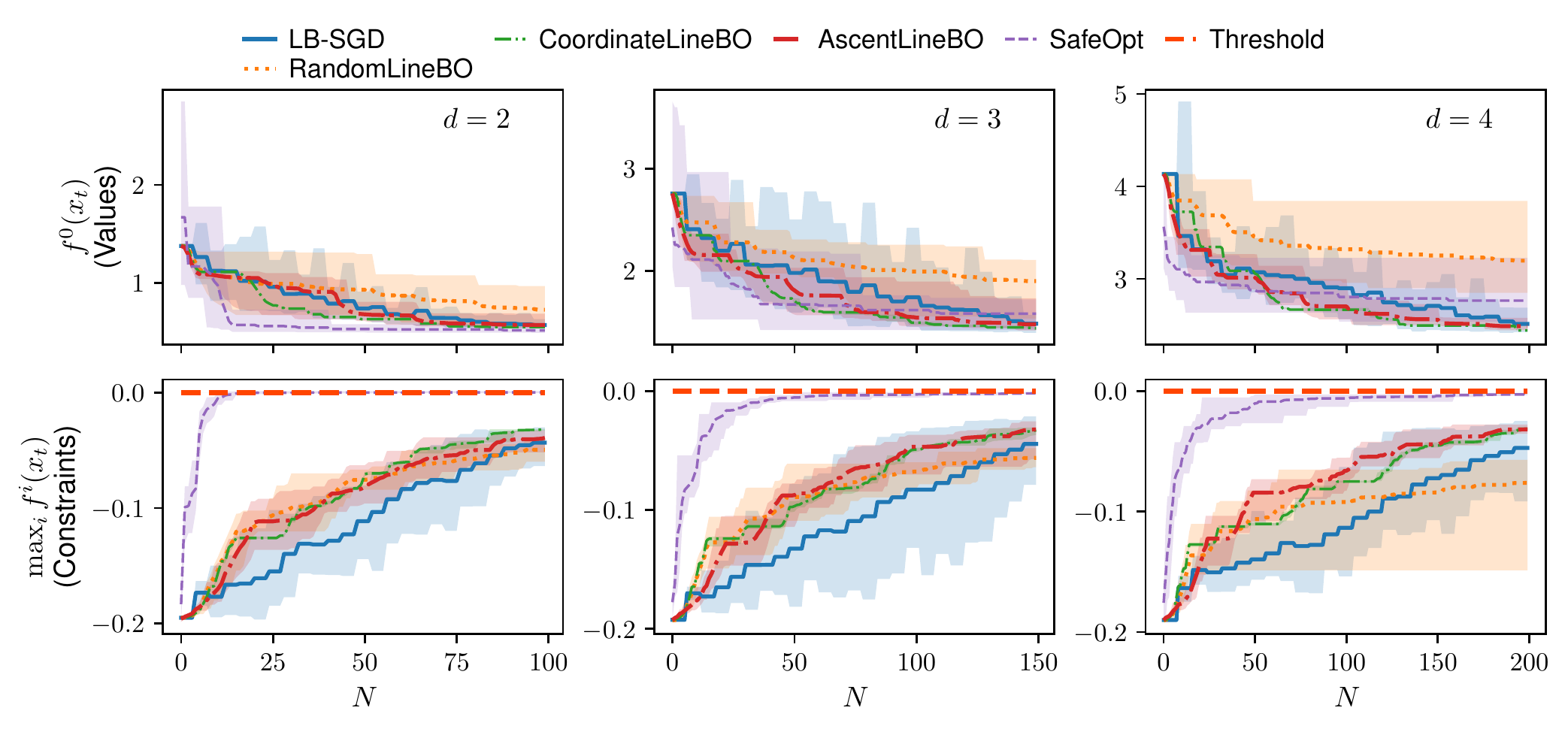}
         \caption{Accuracy  and constraints  of \LBM\, and SafeOpt for $d = 2,3,4$, averaged over $10$ samples. $t$ here is the amount of zeroth-order oracle calls. In these experiments, for \LBM\, we decrease $\eta_{k+1} = 0.7 \eta_k$ gradually every $T_k = 5$ steps with $n_k = d-1$ value measurements at each step. On this problem, for $d=4$ we observe that \LBM\, performs better that SafeOpt, and comparable to LineBO. Best viewed in color.}
         \label{fig:LBM_Safeopt_non-convex}
     \end{figure}
    Again, for $d=2$ we run SafeOpt, and for $d=3,4$ we run SafeOptSwarm. Here, on the constraints plot of SafeOpt and LineBO we again plot the highest value of the constraints over all points explored so far.

     The run-times of LineBO and SafeOpt are demonstrated in 
     Table \ref{tab:runtimes_non-convex} 
     and 
     Figure \ref{fig:runtimes_non-convex}.
     \begin{table}[t]
          \centering
        \begin{tabular}{|c|c|c|c|}
          \hline
          $d$ & 2 & 3 & 4\\
          \hline
            SafeOpt (SafeOptSwarm) & 26.960 & 44.909 & 63.019
              \\
        \hline 
            LineBO & 7.584 & 10.593 & 13.293\\
        \hline
            LB-SGD & 0.294 & 0.332 & 0.324\\
        \hline
          \end{tabular}
          \caption{Run-time (in seconds) as dependent on dimensionality $d$ (Rosenbrock benchmark). }
        \label{tab:runtimes_non-convex}
      \end{table}
    \begin{figure}[t]
         \centering 
        \includegraphics[width = 0.7\textwidth]{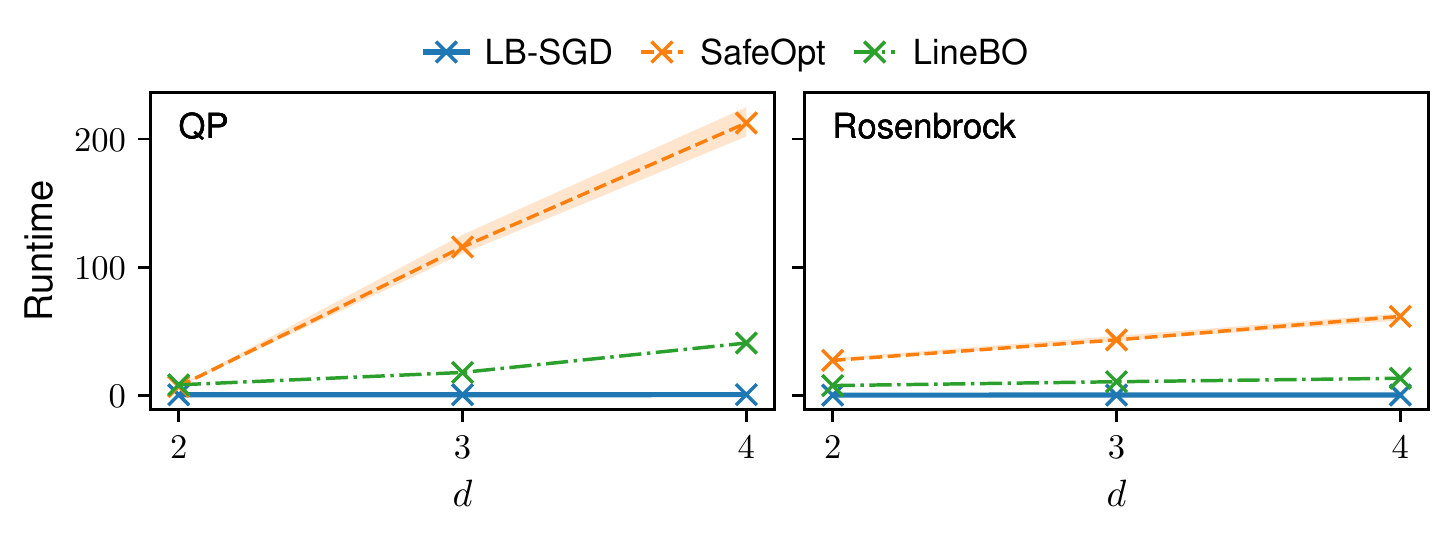}
         \caption{Run-times of \LBM\, and SafeOpt for $d = 2,3,4$, averaged over $10$ samples, in seconds. $t$ here is the amount of zeroth-order oracle calls. We can observe that \LBM\, is a significantly cheaper approach in terms of the computational cost compared to both BO-based methods with growing dimensions.}
         \label{fig:runtimes_non-convex}
     \end{figure}
     Note that the second problem is easier for BO methods than the first one. It is related to the fact that in the first problem, the number of constraints (and therefore, the number of GPs) is higher and grows with dimensionality ($m = 4,6,8$). In contrast, there are always only two constraints for the second problem.       
      As one can see, our approach is significantly cheaper in computational time than SafeOpt. This is, of course, at the price of finding only a local minimum, not the global one.
      
     \subsubsection{Comparison with LineBO in higher dimensions
     }\label{sec:exp:linebo}
     In higher dimensions, it is well known that SafeOpt is not tractable. Therefore, we compare our method only with LineBO \citep{kirschner2019adaptive}.
     This method scales significantly better with dimensionality than the classical BO approaches. 
     The method was demonstrated to be efficient in the unconstrained case and in cases where the solution lies in the interior of the constraint set. The authors proved the theoretical convergence in the unconstrained case and the safety of the iterations in the constrained case. However, in contrast to our method, this approach has a drawback that we discuss below. 
     At each iteration, LineBO  samples a direction (at random, an ascent direction of the objective, or a coordinate direction). 
     Then it solves a \emph{1-dimensional} constrained optimization along this direction, using SafeOpt. After optimizing along this direction, it samples another direction starting from the current point. 
     The drawback of this approach is that when the solution is on the boundary, LineBO might get stuck on the wrong point on the boundary. In such a case, it might be difficult for it to find a safe direction of improvement too close to the boundary. See Figure \ref{fig:illustraionLineBO}  for the illustration of this potential problem. Furthermore, the higher dimension, the harder it is to sample a suitable direction. 
     \begin{figure}[t]
         \centering
      \includegraphics[width = 0.7\textwidth]{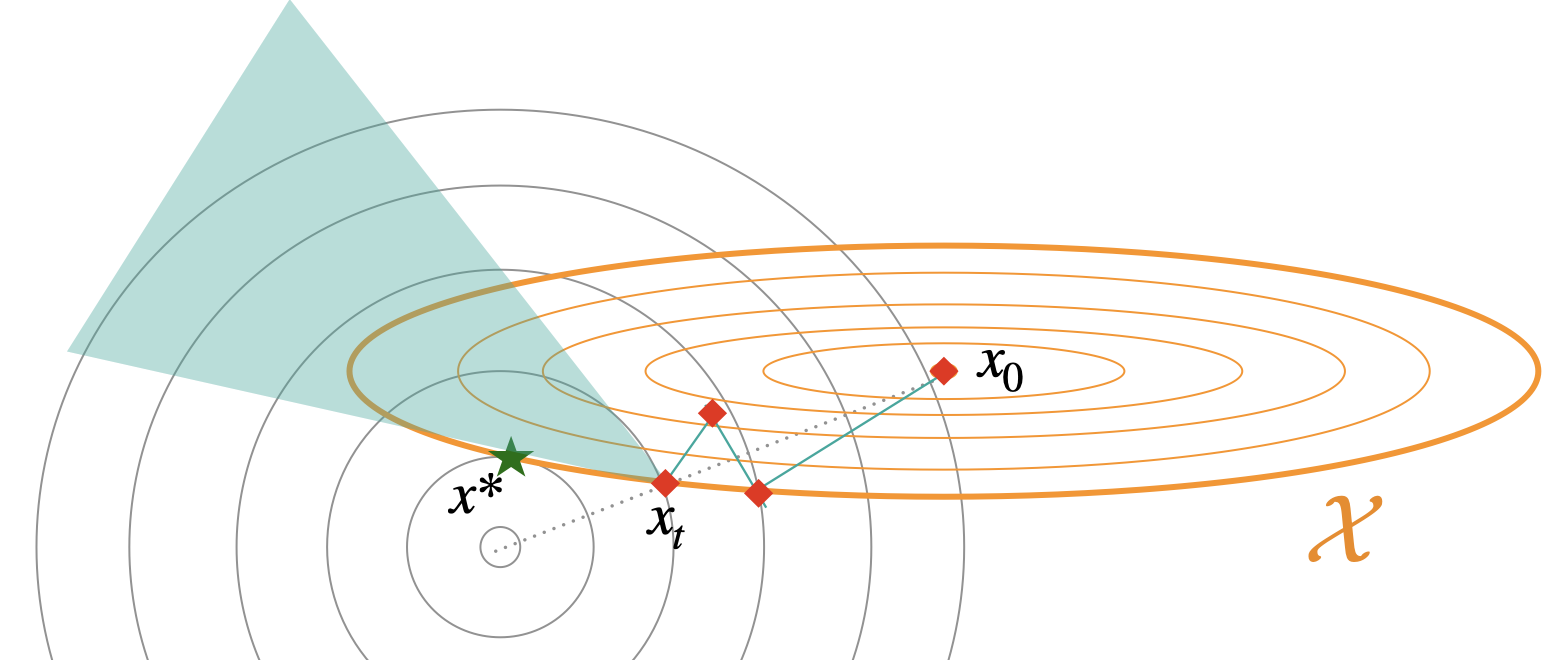}
     \caption{Illustration of the LineBO behavior. At point $x_t$ not every direction allows the safe improvement (only the directions lying in the green sector). Therefore, the LineBO method might get stuck sampling the wrong directions. On this example, the closer to the solution, the narrower is the improvement sector.}
     \label{fig:illustraionLineBO}
     \end{figure}
      We demonstrate that empirically in application to the following problem:
      \begin{align}
      &\min_{x\in\R^d} -\text{exp}^{-4\|x\|^2}, \\
      &\text{s.t. } \la x - \hat x, A (x-\hat x)\ra \leq r^2,
      \end{align}
      with $r = 0.5$ and $A = \text{diag}(3,1.2,\ldots,1.2).$ On Figure \ref{fig:LBM_LineBO} we demonstrate the comparison of LineBO and \LBM\, methods on the above problem for dimensionalities $d = 2, 10, 20$.
          \begin{figure}[t]
         \centering
        \includegraphics[width = \textwidth]{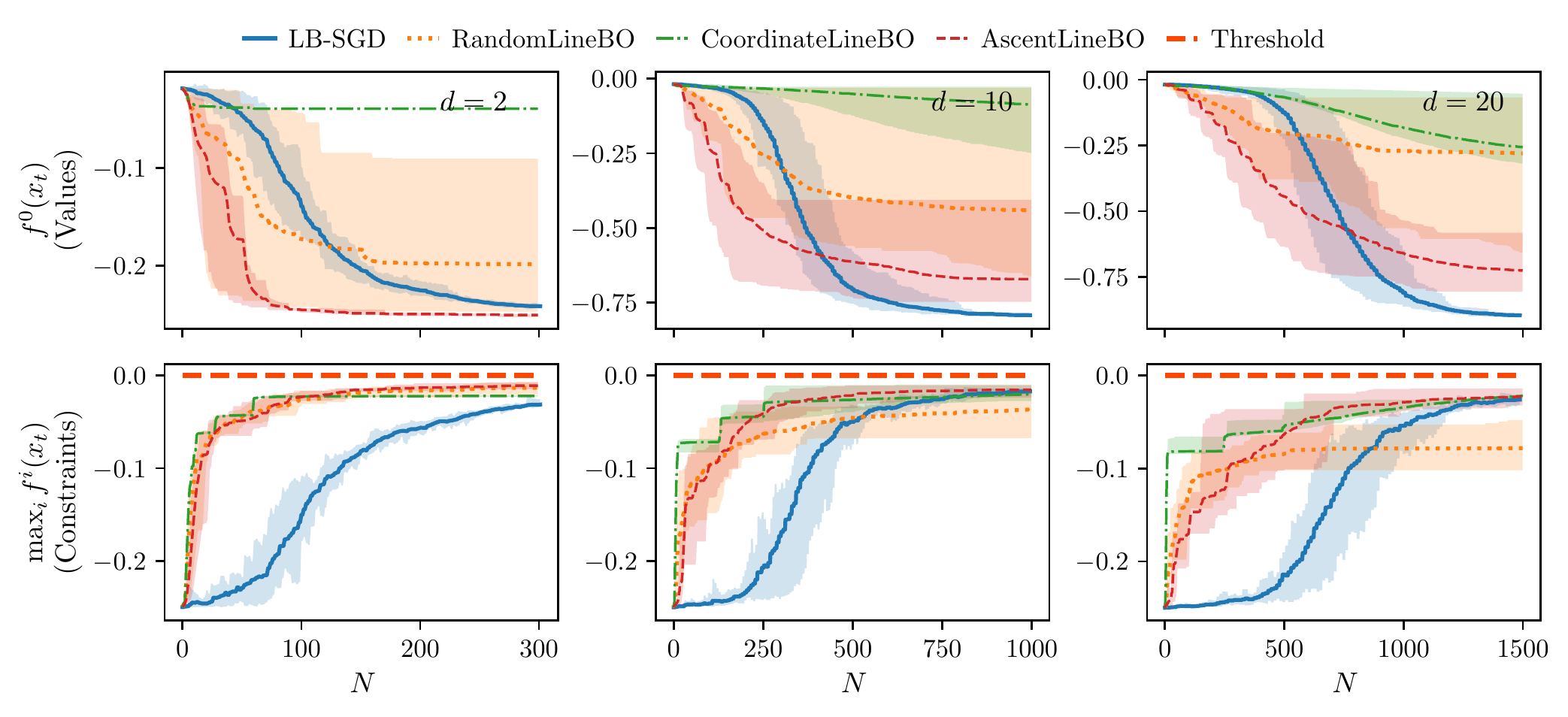}
         \caption{Accuracy  and constraints  of \LBM\, and SafeRandomLineBO for $d = 2,10,20$, averaged over $10$ samples. $t$ here is the amount of zeroth-order oracle calls. In these experiments, for \LBM\, we decrease $\eta_{k+1} = 0.85 \eta_k$ gradually every $T_k = 3$ steps with $n_k = [\frac{d+1}{2}]$ value measurements at each step.  }
         \label{fig:LBM_LineBO}
     \end{figure}
     We report the run-times  in Table \ref{tab:runtimes_gaussian}. 
     \begin{table}[t]
          \centering
        \begin{tabular}{|c|c|c|c|}
          \hline
          $d$ & 2 & 10 & 20\\
          \hline
        LB-SGD & 0.828 & 2.186 & 2.676\\
        \hline
        LineBO & 12.883 & 298.097 & 1038.459\\
        \hline
          \end{tabular}
          \caption{Runtime (in seconds) dependence on dimensionality $d$ on the negative Gaussian minimization benchmark. We can observe that LineBO is significantly more expensive in computational cost (for the same number of queried points).}
        \label{tab:runtimes_gaussian}
      \end{table}
      
     To compare, in the case when the solution is in the interior of the constraint set achieved by setting $r = 10$ (that is, if the constraints do not influence the solution), the LineBO approach does not have this issue and can still be very efficient (see Figure \ref{fig:LBM_LineBO_Un}).
    
    \begin{figure}[t]
         \centering
        \includegraphics[width = \textwidth]{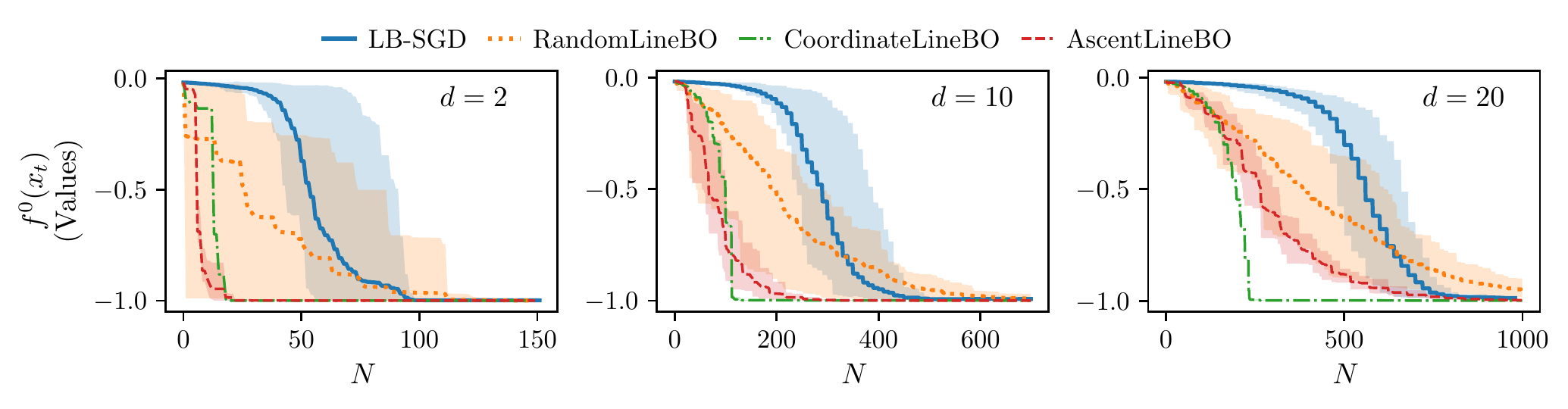}
         \caption{Accuracy  and constraints  of \LBM\, and SafeRandomLineBO for $d = 2,10,20$, averaged over $10$ samples. $t$ here is the amount of zeroth-order oracle calls. In these experiments, for \LBM\, we decrease $\eta_{k+1} = 0.85 \eta_k$ gradually every $T_k = 3$ steps with $n_k = [\frac{d+1}{2}]$ value measurements at each step. }
         \label{fig:LBM_LineBO_Un}
     \end{figure}
     
     \subsection{\LBM\ for safe reinforcement learning}
We previously showed \LBM\ performance on smaller scale, classical black box benchmark problems. In this part, we showcase how \LBM\ scales to more complex, high-dimensional domains arising in RL. 
We consider the ``Safety Gym'' \citep{Ray2019} benchmark suite designed to evaluate safe RL approaches on constrained Markov decision processes (CMDP). In Safety Gym, a robot needs to reach a goal area while avoiding obstacles on the way. Navigation is done by observing first-person-view of the robot.
\subsubsection{Problem statement}
\paragraph{Constrained Markov decision processes}
The problem of safe reinforcement learning can be viewed as finding a policy that solves a \emph{constrained Markov decision process} \citep{altman-constrainedMDP}.
Briefly, we define a discrete-time episodic CMDP as a tuple $\left(\mathcal{S}, \mathcal{A}, \rho, P, R, \gamma, \mathcal{C}\right)$. At each time step $\tau \in \{0, \dots, \mathrm{T}\}$, an agent observes a state $s_\tau \in \mathcal{S}$. The initial state $s_0$ is determined according to some unknown distribution $\rho(s_0)$ such that $s_0 \sim \rho(s_0)$. Given a state, the agent decides what action $a_\tau \in \mathcal{A}$ to take next. Then, an unknown transition density $P : \mathcal{S} \times \mathcal{A} \times \mathcal{S} \rightarrow \left[0, 1\right]$, $s_{\tau + 1} \sim P(\cdot | s_\tau, a_\tau)$ generates a new state. $R : \mathcal{S} \times \mathcal{A} \rightarrow \mathbb{R}$ is a reward function that generates an immediate reward signal observed by the agent. The discount factor $\gamma \in (0, 1]$ weighs the importance of immediate rewards compared to future ones. Lastly, $\mathcal{C} = \left\{c^i : \mathcal{S} \times \mathcal{A} \rightarrow \mathbb{R} \bigm| i \in \left[m\right]\right\}$ is a set of immediate cost signals that the agent observes alongside the reward. The goal is to find a policy $\pi : \mathcal{S} \times \mathcal{A} \rightarrow \left[0, 1\right]$ that solves the constrained problem: 
\begin{equation}
   \label{eq:rl-objective}
    \begin{aligned}
	    \max_{\pi} &  \; \underbrace{\E \left[\sum_{\tau}^{} \gamma^\tau R(s_\tau, a_\tau) \right]}_{-f^0(x)} & \text{s.t. } \; \underbrace{\E \left[\sum_{\tau}^{} \gamma^\tau c^i(s_\tau, a_\tau)\right] - d^i}_{f^i(x)} \le 0 \; \forall i \in [m].
    \end{aligned}
\end{equation}
In the above, $d^i, i \in [m]$ are predefined threshold values for the expected discounted return of costs. Note that we take the expectation with respect to all stochasticity induced by the CMDP and policy.
\paragraph{On-policy methods as black-box optimization problems}
A typical recipe for solving CMDPs at scale is to parameterize the policy with parameters $x$ and use \emph{on-policy} methods. 
On-policy methods use Monte-Carlo sampling to sample trajectories from the environment, evaluate the policy, and finally update it \citep{DBLP:journals/corr/ChowGJP15, achiam2017constrained, Ray2019}. 
By using Monte-Carlo, these methods compute unbiased estimates of the constraints, objective and their gradients  \citep[e.g., via REINFORCE, cf.][]{NIPS1999_464d828b}, equivalently to the assumptions in \Cref{ssec:oracle}. 
In particular, the process of sampling trajectories from the CMDP and averaging them to estimate the objective and constraints in \Cref{eq:rl-objective} is equivalent to querying $f^0(x)$ and $f^i(x), i \in [m]$ and gives rise to a first-order, stochastic and unbiased oracle. However, without deliberately enforcing $x_t \in \mathcal{X}\ \forall t \in \{0, \dots, T\}$, these methods may use an unsafe policy during learning\footnote{It is important to note the work of \citet{dalal2018safe}, which, under a more strict setting, treats this specific challenge, but uses an off-policy algorithm \citep[DDPG]{https://doi.org/10.48550/arxiv.1509.02971} to solve CMDPs.}.\looseness -2

\paragraph{Solving CMDPs with \LBM}
Another shortcoming of the previously mentioned algorithms is that using only Monte-Carlo sampling often leads to high variance estimates of $f^0(x)$ and $f^i(x)$ \citep{https://doi.org/10.48550/arxiv.1506.02438}. To reduce this variance, one is typically required to take an abundant number of queries of $f^0(x)$ and $f^i(x)$, making the aforementioned algorithms sample-inefficient. One way to improve sample efficiency, is to learn a model of the CMDP, and query it (instead of the real CMDP) to have \emph{approximations} of $f^0(x)$ and $f^i(x)$. 
This allows us to trade off the high variance and sample inefficiency with some bias introduced my model errors. By making this compromise, model-based methods empirically exhibit improved sample efficiency compared to the previously mentioned on-policy methods \citep{10.5555/3104482.3104541,DBLP:journals/corr/abs-1805-12114,hafner2021mastering}. Motivated by this insight, we use LAMBDA \citep{As2022Constrained}, a recent model-based approach for solving CMDPs. In short, LAMBDA learns the transition density $P$ from image observations, and uses this learned model to try to find an optimal policy. To accommodate the complexity of learning a policy from a high dimensional input such as images, LAMBDA requires $\sim$588400 parameters to parameterize the policy. This allows us to demonstrate \LBM's ability to scale to problems with large dimensionality. To solve \Cref{eq:rl-objective}, LAMBDA queries \emph{approximations} of $f^0(x), f^i(x)$ by using its model of the CMDP together with its policy to sample \emph{model-generated, on-policy trajectories}.
As described before, these model-generated trajectories are subject to model errors which in turn makes the estimation of the objective and constraints biased. As a result, the assumptions in \Cref{ssec:oracle} do not necessarily, hold as in this case, the oracle is first-order, stochastic but \emph{biased}. Nevertheless, this biasedness is subject only to LAMBDA's model inaccuracies, so \LBM\ can still produce safe policies with high utility, as we empirically show in the following section. \citet{As2022Constrained} use the ``Augmented Lagrangian'' \citep{nocedal2006numerical} to turn the constrained problem into an unconstrained one. However, by using the Augmented Lagrangian, $x_t$ are generally infeasible throughout training, even if $x_0 \in \mathcal{X}$. Therefore, we use \LBM\ instead and empirically show that by using it we get $x_t \in \mathcal{X}\; \forall t \in \{0, \dots, T\}$.\looseness -2

\subsubsection{Experiments}
\paragraph{Addressing the assumptions} Let us briefly discuss the assumptions  in \Cref{section:problem} and explicitly state which of them do not hold. 
\begin{inparadesc}
    \item[Oracle.] As mentioned before, we cannot guarantee the assumptions in \Cref{ssec:oracle}. LAMBDA uses neural networks to model the transition density and to learn an approximation of the objective and constraints\footnote{The approximation of the objective and constraint is done by learning their corresponding \emph{value functions}. Please see \citet{As2022Constrained} for further details.}. For this reason, the assumption on unbiased zeroth-order queries, and assumptions of sub-Gaussian oracles do not hold.
    \item[Smoothness.] By choosing ELU activation function \citep{elupaper} we ensure the smoothness of our approximation of the objective and constraints.
    \item[MFCQ.] In general, similarly to the assumptions on the oracle, this cannot be guaranteed. However, in our experiments, the CMDP is defined to have only one constraint ($m = 1$) so this assumption is satisfied by definition.
    \item[Safe initial policy.] This assumption exists in a large body of previous work \citep{berkenkamp2017safe,koller2018learning,wabersich2021predictive}. Yet, it is not always clear how to design such a policy a-priori. In the following, we propose an experimental protocol in which this assumption empirically holds.
\end{inparadesc}
\begin{figure}
  \begin{minipage}[c]{0.31\textwidth}
    \caption{In Safety Gym, a robot should navigate to a goal area (green circle), while avoiding all other objects (turquoise and blue). The robot on the right picture should arrive to a smaller goal region, making navigation harder.}
  \end{minipage}\hspace{0.6cm}
  \begin{minipage}[c] {0.6\textwidth}\includegraphics[height=100pt,trim=400 0 20 0,clip]{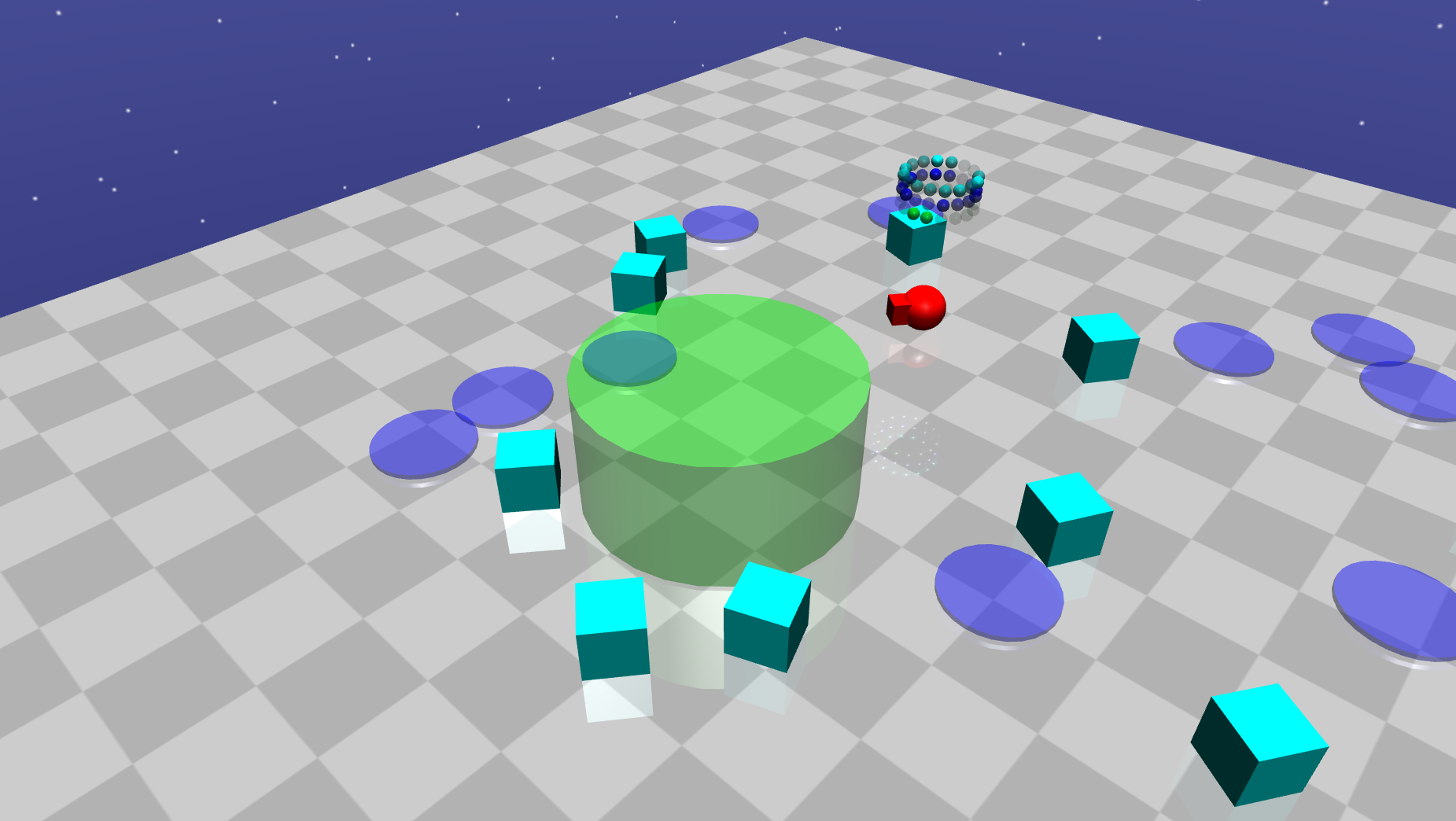}
    \hspace{0.4cm}
    \includegraphics[height=100pt,trim=200 0 200 0,clip]{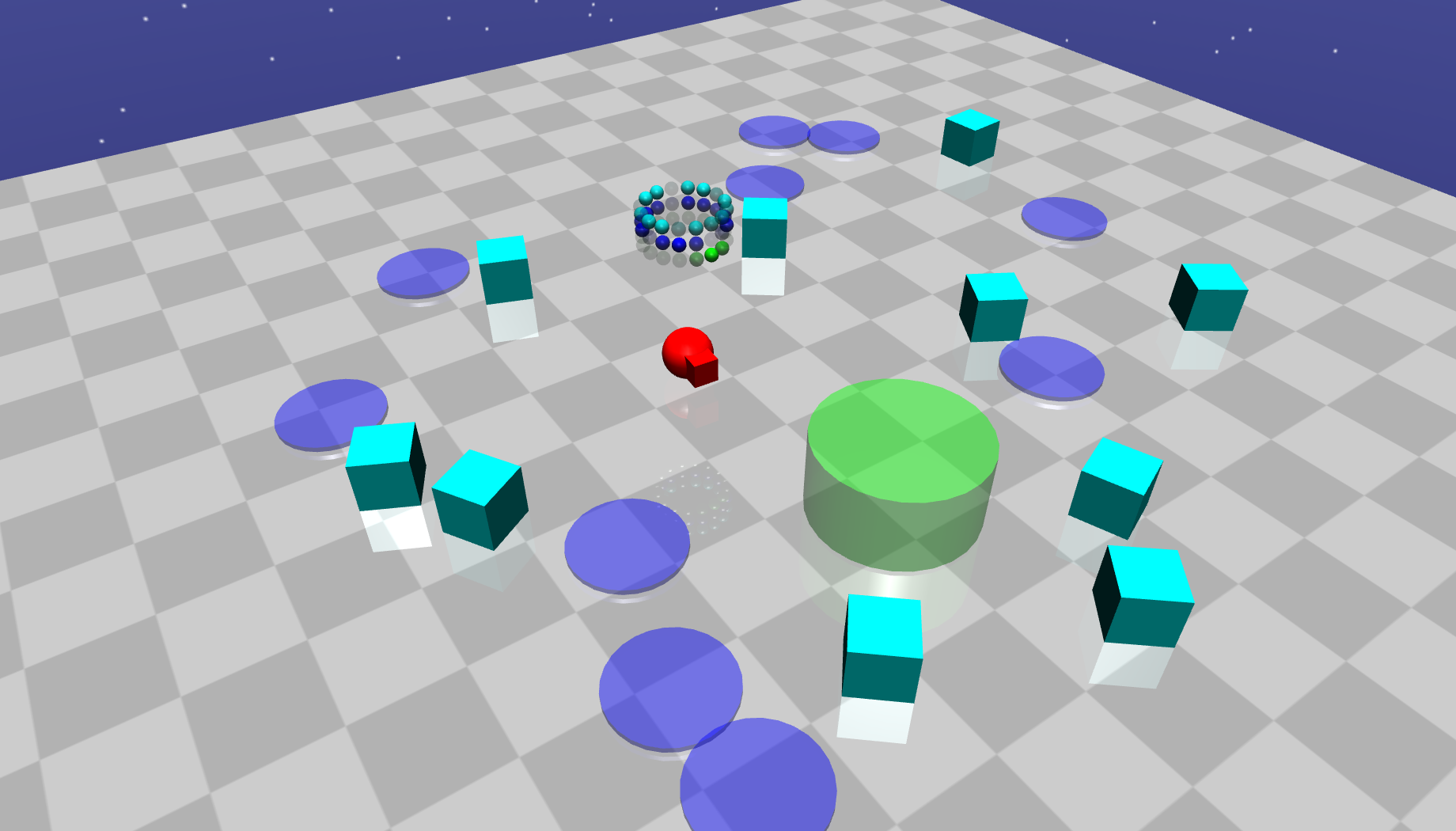}
  \end{minipage}
  \label{fig:big-normal-goal}
\end{figure}
\paragraph{Experiment protocol}
To ensure \LBM\ starts from a safe policy, we warm-start it with a policy that was trained on a similar, but easier task. Specifically, we follow a similar experimental setup as \citet{As2022Constrained} but first train the agent with LAMBDA on a task in which the \emph{goal area is larger}, as shown in \Cref{fig:big-normal-goal}. We use the policy parameters of the trained agent as a starting point for \LBM\ on a harder task, in which the goal area is smaller. As we later show, this allows the agent to start the second stage with a \emph{safe but sub-optimal policy}.
We verify this setup with all three available robots of the Safety-Gym benchmark suite \citep{Ray2019}, each run with 5 different random seeds. For our implementation, please see \url{https://github.com/lasgroup/lbsgd-rl}.

\paragraph{Results}
We first validate our experimental setup. In \Cref{fig:rl-no-update} we show that by using either \LBM\ or the Augmented Lagrangian in the first stage, and \emph{not updating} the policy in the second stage, LAMBDA's policy is \emph{safe} but \emph{sub-optimal}. With this, we empirically confirm the safe initial policy assumption on the second stage of training.
Further, given such a policy, we compare \LBM\ with the Augmented Lagrangian on the second stage. In \Cref{fig:rl-update} we demonstrate how the Augmented Lagrangian needs to ``re-learn'' a new value for the Lagrange multiplier and therefore fails to transfer safely to the harder task. However, by using \LBM, the agent is able to maintain safety after transitioning to the harder task. It is important to note that this safe transfer comes at the cost of limited exploration. As shown in \Cref{fig:rl-no-update,fig:rl-update}, \LBM\ finds slightly less performant policies compared to the Augmented Lagrangian. \textcolor{black}{In \Cref{fig:rl-update} we also compare with the classical constrained policy optimization (CPO) algorithm \citep{achiam2017constrained} following the implementation in \url{https://github.com/lasgroup/jax-cpo} as a baseline. CPO is a model-free method and hence requires much more environment interactions. Therefore, we only show the final level CPO reaches after 10 million environment steps.}
\begin{figure}
    \centering
    \includegraphics[trim=5 5 5 5,clip,width=\textwidth]{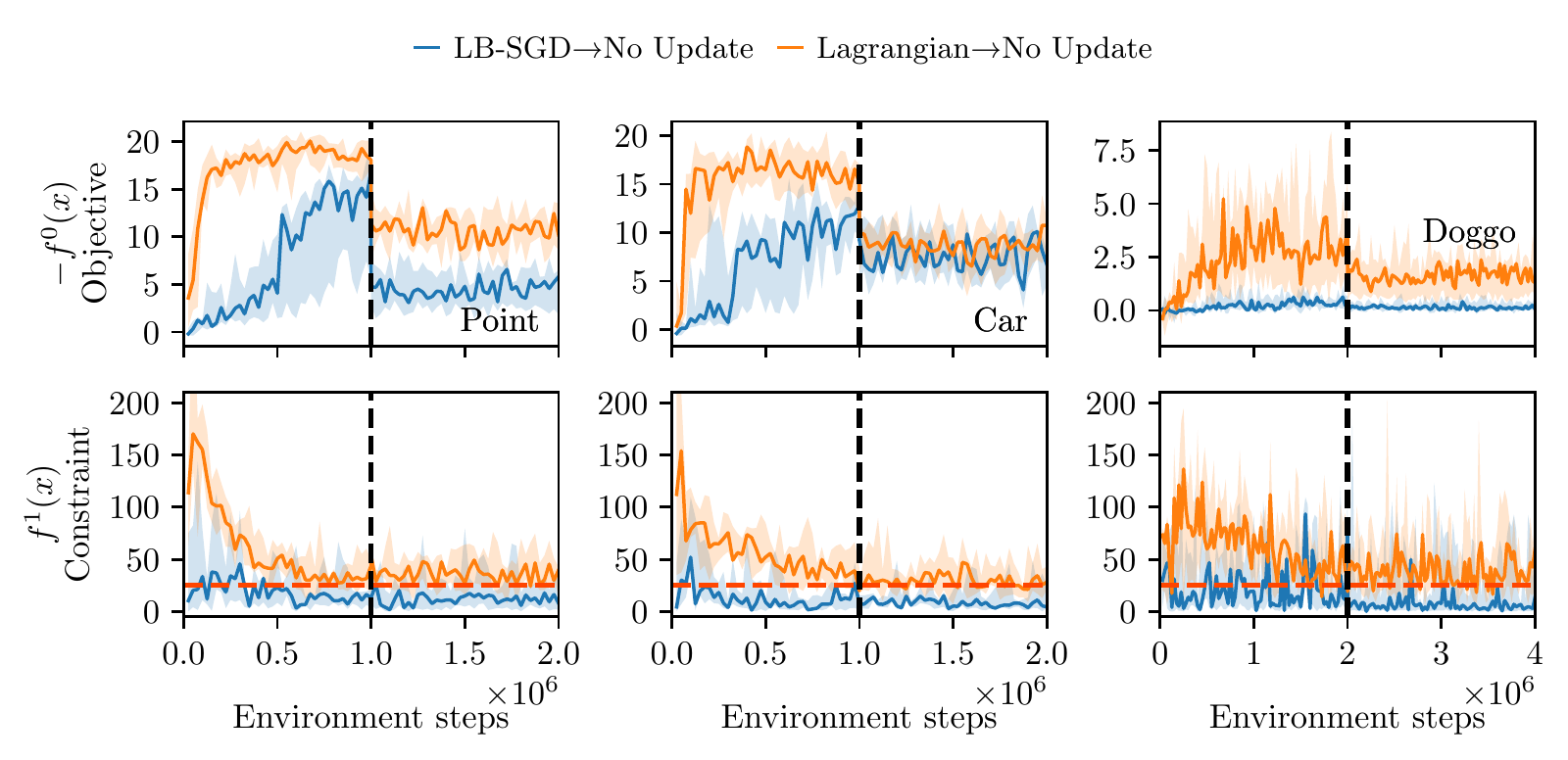}
    \caption{Across all different robots of the Safety Gym suite, LAMBDA with \LBM\ and the Augmented Lagrangian transfer well to the second stage in terms of safety. Since we do not update the policy, LAMBDA fails to reach the same task performance on the second stage, as expected. Shaded areas represent the the 5\% and 95\% percentiles of 5 different random seeds.}
    \label{fig:rl-no-update}
\end{figure}
\begin{figure}
    \centering
    \includegraphics[trim=5 5 5 5,clip,width=\textwidth]{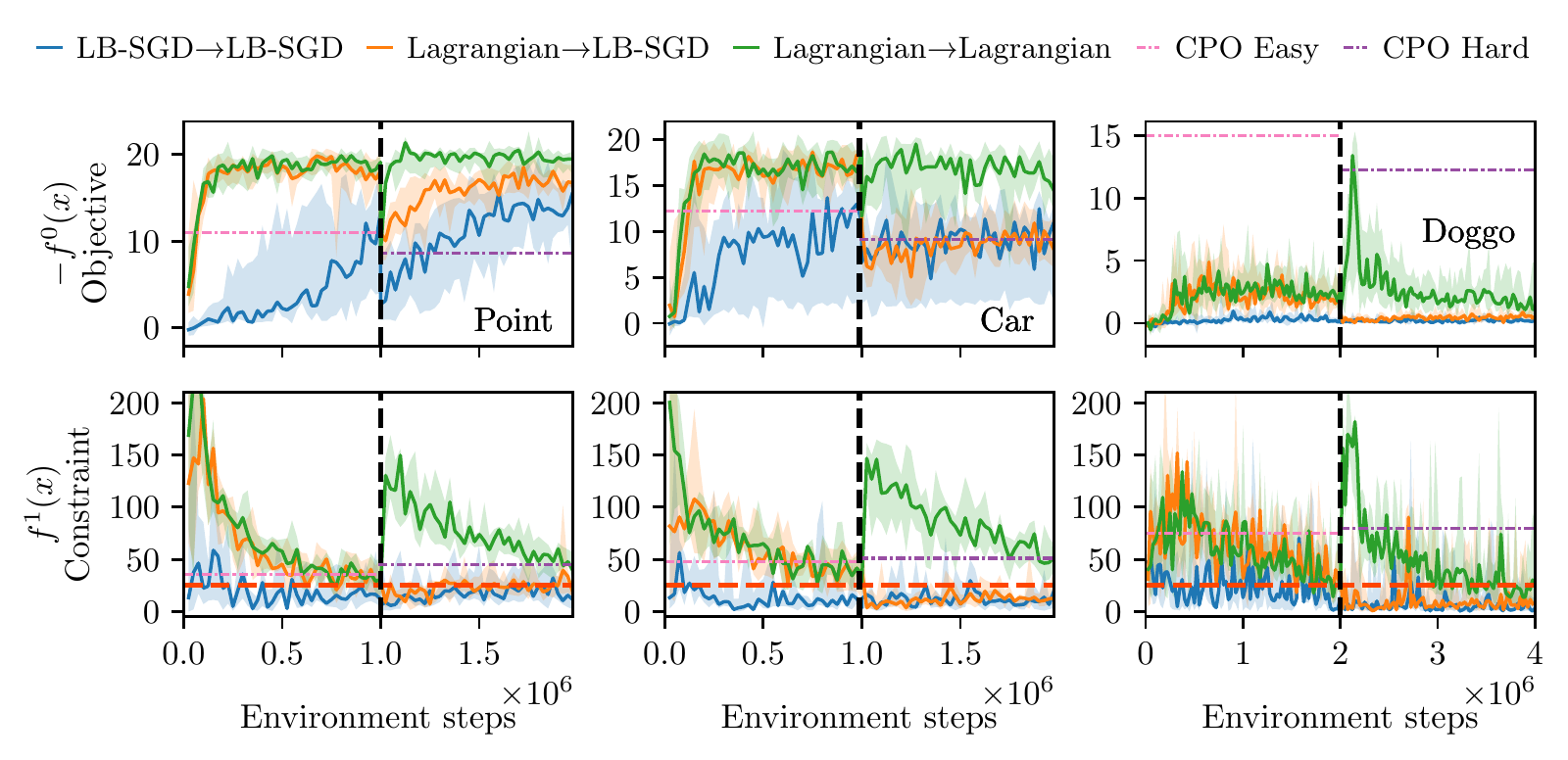}
    \caption{LAMBDA with \LBM\ transfers safely to the second task. The main trade-off, however, is the lower asymptotic performance of \LBM. Conversely, Lagrangian $\rightarrow$ Lagrangian fails transfer safely as the constraints with all robots rapidly grow when the new task is revealed (as shown by the vertical dashed black line). \textcolor{black}{CPO Easy and CPO Hard correspond to CPO's performance on the easier task (larger goal size) and the harder task (smaller goal size). We train CPO on each task separately and show its performance at convergence, i.e., after $10 \times 10^6$ interaction steps for the ``Point'' and ``Car'' robots and $100 \times 10^6$ for the ``Doggo'' robot. As shown, while CPO achieves reasonable performance, it fails to satisfy the constraints in all tasks and all robots, as previously reported by \citet{Ray2019}.}}
    \label{fig:rl-update}
\end{figure}

     \section{Conclusion}
    In this paper, we addressed the problem of sample and computationally efficient safe learning. 
    We proposed an approach based on logarithmic barriers, which we optimize using SGD with adaptive step sizes.
    We analytically proved its safety during the learning and analyzed the convergence rates for non-convex, convex, and strongly-convex problems. We empirically demonstrated the performance of our method in comparison with other existing methods. We show that 
    1) its sample and computational complexity scale efficiently to high dimensions, and;
    2) it keeps optimization iterates within the feasible set with high probability. Additionally, we demonstrate the efficiency of the log barrier approach for high-dimensional constrained reinforcement learning problems. 
    
     While not requiring to explicitly specify a prior (in the Bayesian sense, as considered in safe Bayesian optimization), our method does involve  hyper-parameters such as $\eta_0$, $\eta$-decrease rate parameter $\omega$, amount of steps per episode $T_k$, and exhibits sensitivity to the noise. Also, in the non-convex case, it can converge only to a local minimum, as any other descent optimization approach. However, it is easy to implement and has efficient computational performance due to cheap updates. Therefore, \LBM\ is better suited to problems of high scale.
     
     For future work, it would be exciting to take the best of both worlds and combine the BO approaches that allow us to build and use a global model with our simple and cheap safe descent approach based on log barriers.
     \section{Acknowledgements}
     We thank the support of Swiss National Science Foundation, under the grant SNSF 200021\_172781, the Swiss National Science Foundation under NCCR Automation, grant agreement 51NF40 180545, and European Research Council (ERC) under the European Union's Horizon 2020 research and innovation programme grant agreement No 815943.
     \newpage
\appendix
\section{Additional proofs}
First, on Figure \ref{fig:relations} we show how various theorems and lemmas relate to each other throughout our paper.
    \begin{figure}
        \centering
        \includegraphics[width = \textwidth]{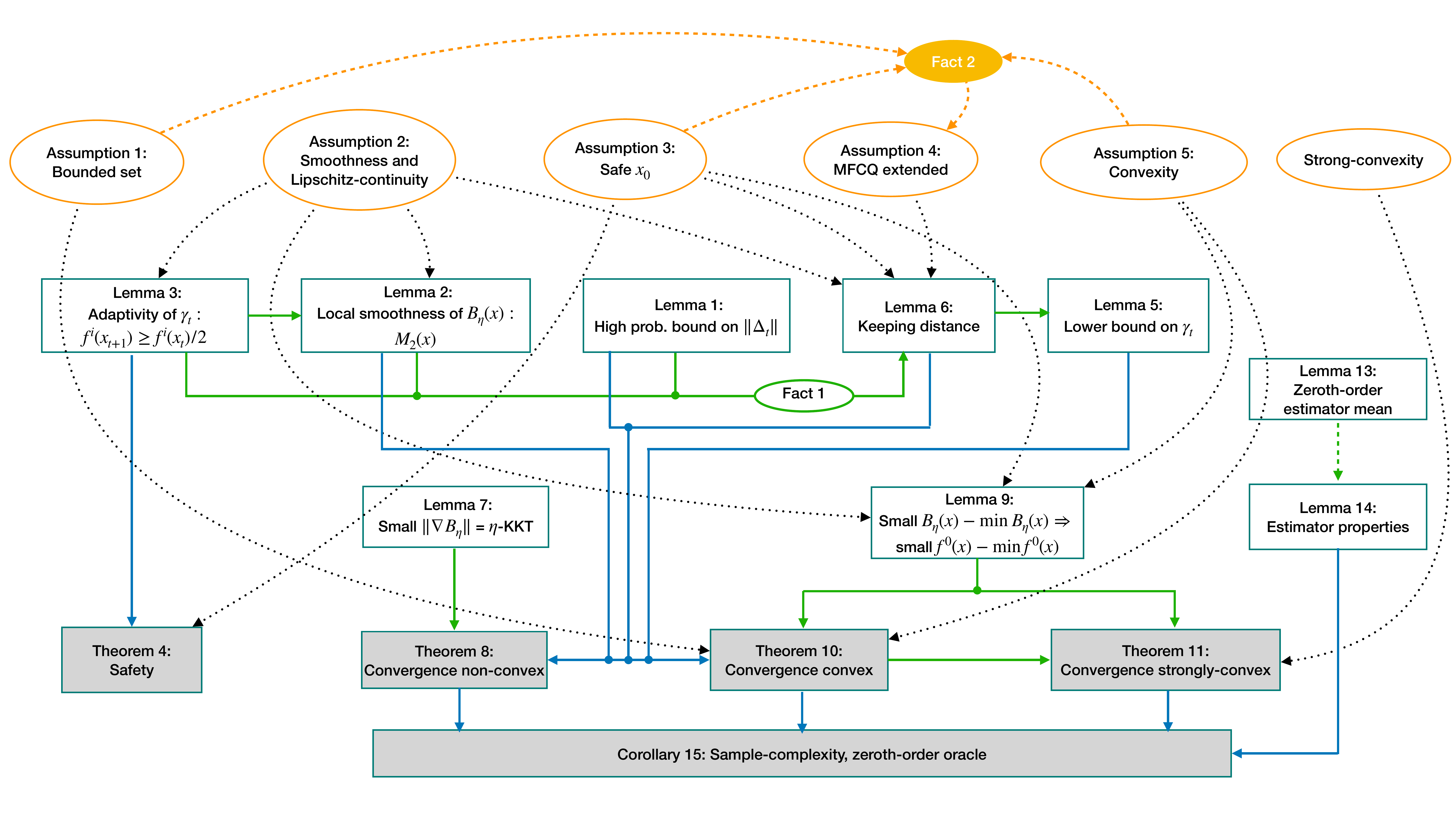}
        \caption{Relations of our theoretical results.}
        \label{fig:relations}
    \end{figure}
    
\begin{wrapfigure}[18]{r}{0.34\textwidth}
\vspace{-35pt}
    \centering
    \includegraphics[height=0.35\textwidth]{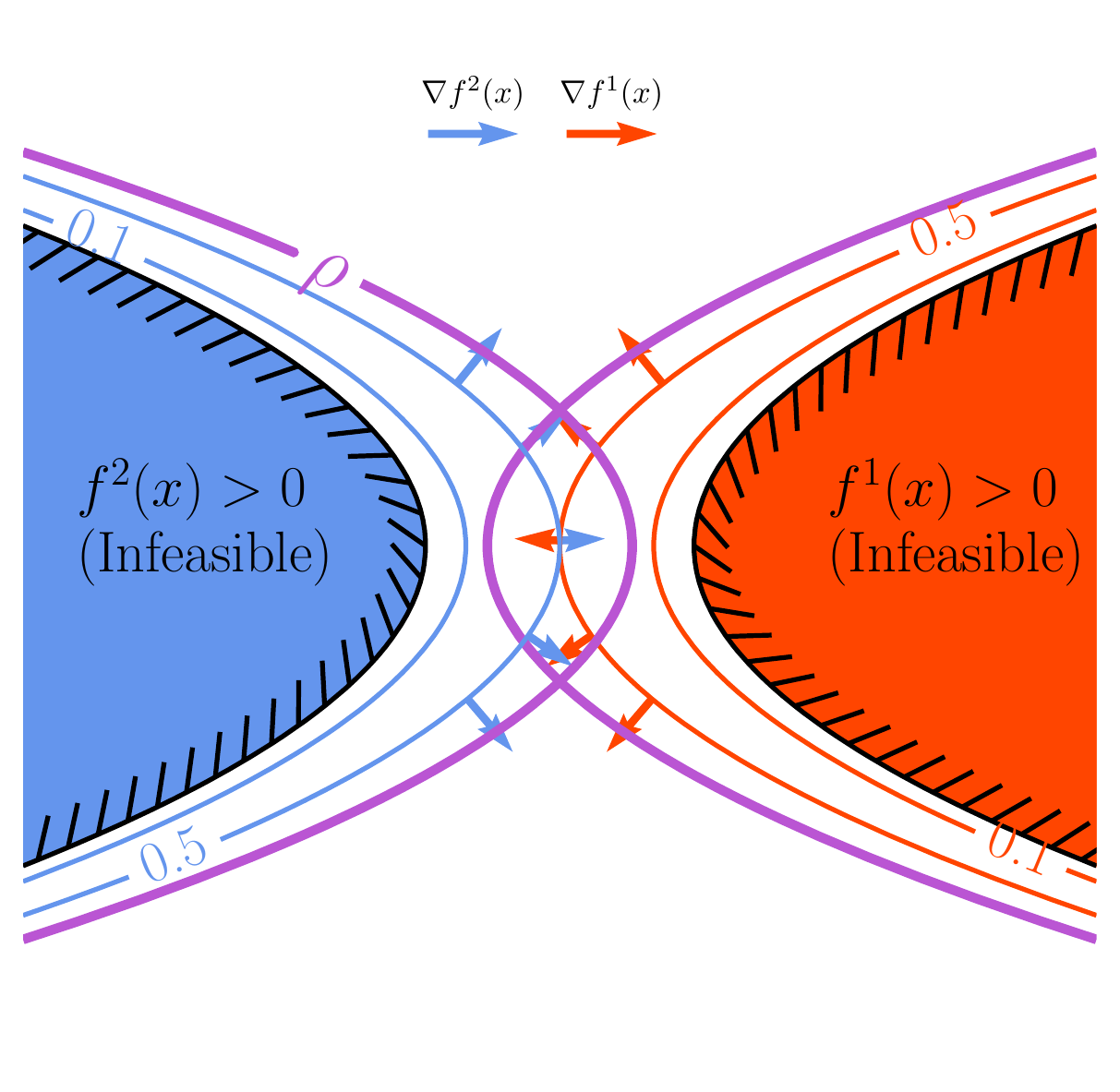}
    \caption{Illustration of extended MFCQ.}
    \label{fig:mfcq}
\end{wrapfigure}
\subsection{MFCQ}\label{A:MFCQ}
Let $x^*$ be a local minimizer of constrained problem (\ref{problem}), and let $\I(x^*) := \{i\in[m]: f^i(x^*) = 0 \}$ denote the set of active constraints at $x^*$. Then, the classic MFCQ is satisfied at $x^*$ if there exists $s\in \R^d$ such that $\la s,\nabla f^i(x^*)\ra < 0$ for all $i\in\I(x^*).$

On Figure \ref{fig:mfcq}, for the point in the middle, both constraints are $\rho$-almost active. Note that at this point, no descent direction exists for both constraints since their gradients are pointing to the opposite directions. That is, this set does not satisfy the extended MFCQ with the given  $\rho>0$.
\subsection{Proof of Lemma \ref{lemma:LBGradEstimatorVarBias}}\label{A:var}
\begin{proof}
        Using the triangle inequality, we get
    \begin{align*}
        \|\Delta_t\| & = \|g_t - \nabla B_{\eta}(x_t)\|\\
         & = \left\|G^0_n(x_t) - \nabla f^0(x_t) +  \sum_{i=1}^m \left[\eta G^i_n(x_t) \left(\frac{1}{\bar \alpha_t^i} - \frac{1}{\alpha_t^i} \right) + \eta( G^i_n(x_t) - \nabla f^i(x_t))  \frac{1}{\alpha_t^i}\right]\right\|\\
         & \leq \left\|G^0_n(x_t) - \nabla f^0(x_t) \right\|+  \sum_{i=1}^m \left[\eta\left\| G^i_n(x_t)\right\| \left(\frac{1}{\bar \alpha_t^i} - \frac{1}{\alpha_t^i} \right) + \frac{\eta}{\alpha_t^i}\left\| G^i_n(x_t) - \nabla f^i(x_t))\right\|  \right].
    \end{align*}
    With high probability, we know $\|G^i_n(x_t)\|\leq L_i$, and from the sub-Gaussian property we have:
    
    $$\Prob\left\{\left\|G^i_n(x_t) - \nabla f^i(x_t)  \right\| \leq \hat b_i + \hat \sigma_i(n) \sqrt{\ln\frac{1}{\de}}\right\}\geq 1-\de$$
    
    $$\Prob\left\{\left|\alpha_t^i - \bar \alpha_t^i  \right| \leq \sigma_i(n) \sqrt{\ln\frac{1}{\de}}\right\}\geq 1-\de,$$
    from what we conclude:
    \begin{align*}
        \Prob\left\{\|\Delta_t\| \leq  \hat b_0 + \hat \sigma_0(n) \sqrt{\ln\frac{1}{\de}} + \sum_{i=1}^m \frac{\eta}{\bar\alpha_t^i}\left( \hat b_i + \hat \sigma_i(n) \sqrt{\ln\frac{1}{\de}}\right) + \sum_{i=1}^m L_i \frac{\eta}{\alpha_t^i\bar\alpha_t^i} \sigma_i(n) \sqrt{\ln\frac{1}{\de}} \right\} \geq 1-\de.
    \end{align*}
\end{proof}

\subsubsection{Bias}\label{appendix:bias}
\begin{proof}
    Using $\E[XY]\leq \sqrt{\E [X^2] \E [Y^2]}$, we get
    \begin{align*}
        \|\E \Delta_t\|
         & = 
        \|\E[g_t - \nabla B_{\eta}(x_t)]\|\\
        & \leq \left\| \E [G^0(x_t) - \nabla f^0(x_t)] +  \sum_{i=1}^m \E \left[\eta G^i(x_t) \left(\frac{1}{\bar \alpha_t} - \frac{1}{\alpha_t} \right) + \eta( G^i(x_t) - \nabla f^i(x_t))  \frac{1}{\alpha_t}\right] \right\| \\
         & = \left\|\sum_{i=1}^m \E \left[\eta G^i(x_t) \left(\frac{1}{\bar \alpha_t} - \frac{1}{\alpha_t} \right) \right] \right\| + \hat b^0_t + \sum_{i=1}^m \frac{\eta}{\alpha_t^i}\hat b^i_t
         \\ 
         & \leq   \sum_{i=1}^m \E \left[ \left\|\eta G^i(x_t) \left(\frac{1}{\bar \alpha_t^i} - \frac{1}{\alpha_t^i} \right)\right\|\right] + \hat b^0_t + \sum_{i=1}^m \frac{\eta}{\alpha_t^i}\hat b^i_t
         \\
        & \leq  \sum_{i=1}^m  \sqrt{ \E [\eta^2\|G^i(x_t)\|^2] \E \left[\frac{1}{(\bar\alpha_t^i)^4}\|\alpha_t^i - \bar\alpha_t^i\|^2\right] } + \hat b^0_t + \sum_{i=1}^m \frac{\eta}{\alpha_t^i}\hat b^i_t\\
        & \leq \sum_{i=1}^m  \frac{\eta L_i \sigma_t}{(\alpha_t^i)^2}  + \hat b^0_t + \sum_{i=1}^m \frac{\eta}{\alpha_t^i}\hat b^i_t
      ,
    \end{align*}
\end{proof}


\subsection{Proof of the adaptivity}\label{P1}
\begin{proof}
    For adaptivity, we require
    $$f^i(x_{t+1}) \leq \frac{f^i(x_{t})}{2}.$$
    Using $M_i$-smoothness, we can bound the $i$-th constraint growth:
    \begin{align*}
        f^i(x_{t+1})
        & \leq f^i(x_{t}) + \la \nabla f^i(x_t), x_{t+1} - x_t \ra + \frac{M_i}{2}\|x_t-x_{t+1}\|^2 \\
        &  = f^i(x_{t}) - \gamma_t \la \nabla f^i(x_t),g_t \ra + \gamma_t^2\frac{M_i}{2}\|g_t\|^2
    \end{align*}
    That is, the condition on $\gamma_t$ for adaptivity (and safety) we can formulate by
    $$ -\gamma_t \la \nabla f^i(x_t),g_t \ra + \gamma_t^2\frac{M_i}{2}\|g_t\|^2  \leq \frac{-f^i(x_{t})}{2} = \frac{\alpha_t^i}{2}.$$
    By carefully rewriting the above inequality without strengthening it, we get
       \begin{align*}
        &     \gamma_t^2\frac{M_i}{2}\|g_t\|^2  - \gamma_t \la \nabla f^i(x_t),g_t \ra \pm  \frac{1}{2M_i}\frac{\la \nabla f^i(x_t),g_t\ra^2}{\|g_t\|^2}\leq \frac{\alpha_t}{2}
        \\
        & \left(\gamma_t\|g_t\| - \frac{\la \nabla f^i(x_t),g_t\ra}{M_i\|g_t\|}\right)^2 
        \leq \frac{\alpha_t}{ M_i} + \frac{\la \nabla f^i(x_t),g_t\ra^2}{M_i^2\|g_t\|^2}
        \end{align*}
        Using the quadratic inequality solution, we obtain the following sufficient bound on the adaptive $\gamma_t:$ 
        $$ 
         \gamma_t\|g_t\| 
        \leq \frac{\la \nabla f^i(x_t),g_t\ra}{M_i\|g_t\|} + \sqrt{\frac{\alpha_t}{ M_i} + \frac{\la \nabla f^i(x_t),g_t\ra^2}{M_i^2\|g_t\|^2}} = (*)$$
        Then, we can rewrite this expression of the right part as follows:
        \begin{align*} 
          (*)& = \sqrt{\frac{\la \nabla f^i(x_t),g_t\ra^2}{M_i^2\|g_t\|^2} + \frac{\alpha^i_t}{ M_i} } + \frac{\la \nabla f^i(x_t),g_t\ra}{M_i\|g_t\|}
          = 
          \sqrt{\frac{\alpha_t^i}{ M_i}}
          \left(\sqrt{\frac{\la \nabla f^i(x_t),g_t\ra^2}{M_i \alpha_t^i\|g_t\|^2} + 1} + \frac{\la \nabla f^i(x_t),g_t\ra}{\sqrt{M_i \alpha_t^i}\|g_t\|}\right)\\
          &= \sqrt{\frac{\alpha_t^i}{ M_i}}
          \frac{1}{
          \sqrt{\frac{\la \nabla f^i(x_t),g_t\ra^2}{M_i \alpha_t^i\|g_t\|^2} + 1} - \frac{\la \nabla f^i(x_t),g_t\ra}{\sqrt{M_i \alpha_t^i}\|g_t\|}}
          = \frac{\alpha_t^i}{\sqrt{\frac{\la \nabla f^i(x_t),g_t\ra^2}{\|g_t\|^2} + M_i \alpha_t^i} - \frac{\la \nabla f^i(x_t),g_t\ra}{\|g_t\|}} \\
          & =  \frac{\alpha_t^i}{\sqrt{(\theta_t^i)^2 + M_i \alpha_t^i} - \theta_t^i}
        \end{align*} 
Therefore, the condition $\gamma_t \leq  \min_{i\in[1,m]}\left\{\frac{\alpha_t^i}{\sqrt{(\theta_t^i)^2 + M_i \alpha_t^i} - \theta_t^i}\right\}
\frac{1}{\|g_t\|} $
is sufficient for $f^i(x_{t+1})\leq \frac{f^i(x_{t})}{2}.$ 
Using the Cauchy-Schwartz inequality, we can simplify this condition (but making it  more conservative):
$$(*) \geq \frac{\alpha_t^i}{
          \sqrt{(\theta_t^i)^2 + \alpha_t M_i} +|\theta_t^i|}
          \geq 
         \frac{\alpha^i_t}{2|\theta_t^i| + \sqrt{\alpha_t^i M_i}}.$$
\end{proof}

\subsection{Proof of the local smoothness}\label{P2}

\begin{proof}
Let us define  the hessian of the log-barrier $B_{\eta}(x)$ by $H_B(y)$ at the region $Y_t$ around $x_t$ such that $Y_t := \{y\in\R^d|y = x_t - u g_t, u\in[0,\gamma_t]\}.$ 
 Note that by definition of the log barrier, the hessian of it at the point $y\in Y_t$ is given by $$H_B(y) = \nabla^2B_{\eta}(y) =  \nabla^2 f^0(y) + \sum_{i=1}^m \eta\frac{\nabla^2 f^i(y)}{-f^i(y)} + \sum_{i=1}^m \eta\frac{\nabla f^i(y)\nabla f^i(y)^T}{(-f^i(y))^2}.$$
\textcolor{black}{From the above and the fact that for any $y \in Y_t$  we have $-f^i(y) \geq 0.5 \alpha_t$,} we get
\begin{align*}
    | g_t^T H_B(y) g_t| & \leq M_{\textcolor{black}{0}}\| g_t\|^2 + \eta \sum_{i=1}^m \frac{M_i}{0.5 \alpha_{t}^i}\| g_t\|^2 + \eta\sum_{i=1}^m  \frac{(\nabla f^i(y)^T g_t/\|g_t\|)^2}{(0.5 \alpha_{t}^i)^2}\|g_t\|^2 \\
    & \leq \|g_t\|^2\left(M_{\textcolor{black}{0}} + \eta \sum_{i=1}^m \frac{M_i}{0.5 \alpha_{t}^i} + \eta\sum_{i=1}^m  \frac{\la \nabla f^i(y),  g_t\ra^2/\| g_t\|^2}{(0.5 \alpha_{t}^i)^2}\right)\\
    & \leq \| g_t\|^2\left(M_{\textcolor{black}{0}} + 2\eta \sum_{i=1}^m \frac{ M_i}{\alpha_{t}^i} + 4 \eta\sum_{i=1}^m  \frac{\la \nabla f^i(y), g_t\ra^2}{(\alpha_{t}^i)^2 \| g_t\|^2}\right).
\end{align*}
Thus, $$M_2(x_t) = M_{\textcolor{black}{0}} + 2\eta \sum_{i=1}^m \frac{ M_i}{\alpha_{t}^i} + 4 \eta\sum_{i=1}^m  \frac{\la \nabla f^i(y),  g_t\ra^2}{\textcolor{black}{(\alpha_{t}^i)}^2 \| g_t\|^2}.$$
\textcolor{black}{
Note that $\nabla f^i(y)$ is unknown, therefore we have to bound it based on $\nabla f^i(x_{t}).$ 
Observe that 
\begin{align}\label{eq:p1}
&\la \nabla f^i(y),  g_t\ra^2 \\
&= \la \nabla f^i(x_{t}),  g_t\ra^2 +  \la \nabla f^i(y) - \nabla f^i(x_{t}),  g_t\ra^2 + 2\la \nabla f^i(x_{t}),  g_t\ra \la \nabla f^i(x_{t+1}) - \nabla f^i(x_{t}),  g_t\ra.\nonumber
\end{align}
Using the $M_i$-smoothness of $f^i(x)$ we get:
$$\|\nabla f^i(x_{t}) - \nabla f^i(y)\| \leq M_i\|x_t - y\| \leq M_i\gamma_t\|g_t\|.$$
Recall that we choose the step-size such that: 
$$\gamma_t \leq  \min_{i\in[m]}\left\{\frac{\alpha^i_t}{2|\theta_t| + \sqrt{\alpha_t^i M^i}}\right\}
\frac{1}{\|g_t\|_2}.$$
Hence,
$$\|\nabla f^i(x_{t}) - \nabla f^i(y)\| \leq M_i\|x_t - y\| \leq \frac{M_i \alpha^i_t}{2|\theta_t| + \sqrt{\alpha_t^i M_i}} \leq  \frac{M_i \alpha^i_t}{\sqrt{\alpha_t^i M_i}} = \sqrt{\alpha_t^i M_i}.$$
Then, using \Cref{eq:p1} and $\la \nabla f^i(x_{t}),  g_t\ra^2 = (\theta_t^i)^2\|g_t\|^2$ we get:
\begin{align}\label{eq:p2}
    \la \nabla f^i(y),  g_t\ra^2 
    &\leq |\theta_t^i|^2\|g_t^i\|^2 
    + \alpha_t^i M_i\|g_t^i\|^2 + 2|\theta_t^i|\sqrt{\alpha_t^i M^i}\|g_t^i\|^2\\
    & \leq  2|\theta_t^i|^2\|g_t^i\|^2
    + 2\alpha_t^i M^i\|g_t^i\|^2.
\end{align}
The above inequality is due to the fact that $a^2 + 2ab + b^2 \leq 2a^2 + 2b^2$.
Thus, we finally obtain 
\begin{align*}
M_2(x_t) 
& = M_0 + 2\eta \sum_{i=1}^m \frac{ M_i}{\alpha^i_{t}} + 4 \eta\sum_{i=1}^m  \frac{2|\theta_t^i|^2\|g_t^i\|^2
    + 2\alpha_t^i M^i\|g_t^i\|^2}{(\alpha^i_{t})^2 \|g_t\|^2}\\
& = M_0 + 10\eta \sum_{i=1}^m \frac{ M_i}{\alpha^i_{t}} + 8 \eta\sum_{i=1}^m  \frac{(\theta_t^i)^2}{(\alpha^i_{t})^2 }.
\end{align*}
}
\end{proof}
\subsection{Proof of Fact \ref{fact:alpha_prod}}\label{Appentix:Proof_fact:alpha_prod}
 \paragraph{Fact \ref{fact:alpha_prod}}
    \textit{Let Assumptions 
    \ref{assumption:1}, \ref{assumption:3} hold, and Assumption \ref{assumption:mfcq} hold with $\rho \geq  \eta$, and let 
    $\hat \sigma_i(n) \leq \frac{(\alpha_t^i)^2 L}{ 2\eta\sqrt{\ln \frac{1}{\de}}}$, 
    $\hat b_i \leq \frac{(\alpha_t^i)^2 L}{2\eta}$, and $\sigma_i(n) \leq \frac{\alpha_t^i}{2\sqrt{\ln \frac{1}{\de}}}$.  If at iteration $t$ we have $\min_{i\in[m]} \alpha_t^i \leq \bar c\eta$ 
    with  $\bar c := \frac{l}{L} \frac{1}{2m+1}$, then, for the next iteration $t+1$ we get 
    $\prod_{i\in \mathcal \I }\alpha_{t+1}^i\geq \prod_{i\in \mathcal \I }\alpha_t^i$ for any $\I : \I_t \subseteq \I$ with $\I_t :=\{i\in[m]:\alpha_t^i\leq \eta\}$.}
\\
\begin{proof} Using the local smoothness of the log barrier, we can see:
\begin{align}\label{eq:bound1}
    \eta \sum_{i\in \I_t} -\log \alpha_{t+1}^i 
    & \leq \eta \sum_{i\in \I_t} -\log \alpha_t^i - \gamma_t \la \eta \sum_{i\in\I_t}\frac{\nabla f^i(x_t)}{\alpha^i_t},g_t\ra +\frac{M_2(x_t)}{2}\gamma_t^2\|g_t\|^2 \nonumber\\
    & \leq \eta \sum_{i\in \I_t} -\log \alpha_t^i + \gamma_t \left(-\la \eta \sum_{i\in\I_t}\frac{\nabla f^i(x_t)}{\alpha^i_t},g_t\ra +\frac{1}{2}\|g_t\|^2\right)\nonumber\\
    & = \eta \sum_{i\in \I_t} -\log \alpha_t^i + \frac{\gamma_t \eta^2}{2} \left( 2\la A, A + B\ra +\|A + B\|^2\right)\nonumber \\
    & = \eta \sum_{i\in \I_t} -\log \alpha_t^i + \frac{\gamma_t \eta^2}{2} \left(\|B\|^2 - \|A\|^2\right), 
\end{align}
where 
$g_t = A + B$, with $A := \sum_{i\in\I_t}\frac{\nabla f^i(x_t)}{\alpha_t^i}$ and $B:= \frac{g_t}{\eta} -  \sum_{i\in\I_t}\frac{\nabla f^i(x_t)}{\alpha_t^i} $. Using Assumption \ref{assumption:mfcq} we obtain a lower bound on $\|A\|$:
\begin{align}
    \|A\| = \left\|\sum_{i\in\I_t}\frac{\nabla f^i(x_t)}{\alpha_t^i}\right\| & \geq \la\sum_{i\in\I_t}\frac{\nabla f^i(x_t)}{\alpha_t^i}, s_x\ra \geq \sum_{i\in\I_t}\frac{\la\nabla f^i(x_t), s_x\ra}{\alpha_t^i} \geq \sum_{i\in\I_t} \frac{l}{\alpha_t^i}.
\end{align}
The second part $\|B\|$ we can upper bound with high probability $1-\de$ using the definition of $\I_t$ as follows (since $\forall \alpha_t^j\notin \I_t$ we have $ \alpha_t^j\geq \eta$, therefore $\bar \alpha_t^j\geq \eta/2$ for $\sigma_i(n) \leq \frac{\alpha_t^i}{2\sqrt{\ln \frac{1}{\de}}}$):
\begin{align}
    \|B\| &= \left\|\frac{G^0_n(x_t, \xi_t)}{\eta} 
        + \sum_{j\notin\I_t}\frac{G^j_n(x_t, \xi_t)}{\bar\alpha_t^j} + \sum_{i\in\I_t}\frac{G^i_n(x_t,\xi_t)}{\bar\alpha_t^i} - \sum_{i\in\I_t}\frac{\nabla f^i(x_t)}{\alpha_t^i}\right\|\nonumber \\
    & \leq 
    \frac{\max_i\|G^i_n(x_t,\xi_t)\|}{\eta}\left(1 + 2(m - |\I_t|) \right) 
        + \sum_{i\in\I_t}\left(\frac{\hat \sigma_i(n) \sqrt{\ln \frac{1}{\de}} + \hat b_i}{ \alpha_t^i} + \|G^i_n(x_t,\xi_t)\|\left|\frac{1}{\bar \alpha_t^i} -\frac{1}{ \alpha_t^i}\right|\right)\nonumber \\
    & \leq 
    \frac{\max_i\|G^i_n(x_t,\xi_t)\|}{\eta}\left(1 + 2(m - |\I_t|) \right) 
        + \sum_{i\in\I_t}\left(\frac{\hat \sigma_i(n) \sqrt{\ln \frac{1}{\de}}+ \hat b_i}{\alpha_t^i} + \|G^i_n(x_t,\xi_t)\|\frac{\sigma_i(n)\sqrt{\ln\frac{1}{\de}} }{\bar \alpha_t^i\alpha_t^i}\right)\nonumber\\
    & \leq \frac{L}{\eta}\left(2m + 1\right),
\end{align}
for 
    $\hat \sigma_i(n) \leq \frac{\alpha_t^i\max\|G^i_n(x_t,\xi_t)\|}{2\eta\sqrt{\ln \frac{1}{\de}}}$, $\hat b_i \leq \frac{\alpha_t^i\max\|G^i_n(x_t,\xi_t)\|}{2\eta
    }$,
and $\sigma_i(n) \leq \frac{(\alpha_t^i)^2}{2\eta\sqrt{\ln \frac{1}{\de}}}$, implying $\bar \alpha_t^i \geq \alpha_t^i/2$ and using $\|G^i_n(x_t,\xi_t)\| \leq L.$
Then, if $\min \alpha_t^i \leq \bar c\eta$, we have $\sum_{i \in \I_t}\frac{1}{\alpha_t^i} \geq \frac{1}{\bar c \eta} = \frac{L}{l\eta}\left( 2m+1 \right),$ and therefore with high probability $\|B\| \leq \|A\|$. Then we get (\ref{eq:bound1}), that implies
\begin{align}
    \prod_{i\in \I_t}\alpha^i_{t+1} \geq \prod_{i\in \I_t}\alpha^i_{t}.
\end{align}
Moreover, using the same reasoning, we can prove that
\begin{align}
    \prod_{i\in \I}\alpha^i_{t+1} \geq \prod_{i\in \I}\alpha^i_{t}.
\end{align}
for any subset of indices $\I\subseteq [m]$ such that $\I_t \subseteq \I .$
\end{proof}

\subsection{Lower bound on $\gamma_t$}\label{Appendix:5}
Here we assume $\underline{\alpha}_t^i \geq c\eta.$ Recall that 
\begin{align*}
    \gamma_t = \min\left\{\min_{i\in[m]}\left\{\frac{\underline{\alpha}^i_t}{2|\hat \theta^i_t| + \sqrt{\underline{\alpha}_t^i M_i}}\right\}
    \frac{1}{\|g_t\|}, \frac{1}{\hat M_2(x_t)}\right\}.
\end{align*} where 
\begin{align*}
    \hat M_2(x_t) = M_0 + \textcolor{black}{10} \eta \sum_{i=1}^m \frac{ M_i}{\underline{\alpha}^i_{t}} + \textcolor{black}{8} \eta \sum_{i=1}^m  \frac{(\hat \theta^i_t)^2}{(\underline{\alpha}^i_{t})^2}.
\end{align*}
We get the lower bound by constructing a bound on both of the terms inside the minimum.\\ 
1) We have 
$\Prob\left\{\hat M_{2}(x_t) \leq \left(1+\textcolor{black}{10}\frac{m}{c}\right)M + \textcolor{black}{8}\frac{mL^2}{\eta c^2}\right\}\geq 1-\de$ (Due to Lemma \ref{lemma:keeping_distance}, and by definition of $\hat M_2(x_t)$), which implies 
    $$
    \Prob\left\{\frac{1}{\hat M_2(x_t)} \geq \eta\left( \frac{1}{\frac{\textcolor{black}{8} m}{c^2}L^2 + \eta(1+\textcolor{black}{10}\frac{m}{c})M}\right)\right\}\geq 1-\de.
    $$
2) 
Using Lemma \ref{lemma:keeping_distance}
we get 
$\Prob\left\{\|g_t\| \leq L_0 + \sum_{i=1}^m\frac{L_i}{c}\right\}\geq 1-\de.$
 Hence, we can bound       
    $$
        \Prob\left\{\min_{i\in[m]}\left\{\frac{\underline{\alpha}^i_t}{2|\hat \theta^i_t| + \sqrt{\underline{\alpha}_t^i M_i}}\right\}
    \frac{1}{\|g_t\|} \geq \frac{c\eta}{ (2L + \sqrt{M c\eta})L(1 + \frac{m}{c})}\right\}\geq 1-\de.
    $$
 Therefore,
    $$
        \Prob\left\{\gamma_t \geq \frac{\eta}{2}\min 
        \left\{
         \frac{1}{\frac{\textcolor{black}{4} m}{c^2}L^2 + \eta(\textcolor{black}{0.5}+\textcolor{black}{5}\frac{m}{c})M}, \frac{1}{L^2(\frac{1}{c}+\frac{m}{c^2}) + 0.5\sqrt{\frac{M \eta}{cL^2}}L^2(1+\frac{m}{c}) }\right\}
        \right\}\geq 1-\de,
    $$
    $$
        \Prob\left\{\gamma_t \geq \frac{\eta}{2L^2(1+\frac{m}{c})}\min \left\{
         \frac{1}{\frac{\textcolor{black}{4}}{c} + \frac{\textcolor{black}{5}M\eta}{L^2}}, \frac{1}{\frac{1}{c} + \sqrt{\frac{M \eta}{4cL^2}} }\right\}
        \right\}
        \geq 1-\de.
    $$
    $$
        \Prob\left\{\gamma_t \geq \frac{c\eta}{2L^2(1+\frac{m}{c})}\min \left\{
         \frac{1}{\textcolor{black}{4} + \frac{\textcolor{black}{5}Mc\eta}{L^2}}, \frac{1}{1 +  \sqrt{\frac{Mc\eta}{4L^2}} }\right\}
        \right\}
        \geq 1-\de.
    $$
Finally, the bound is
    $$
        \Prob\left\{\gamma_t \geq \eta C 
        \right\}\geq 1-\de.
    $$
    with
    $$
    C := \frac{c\eta}{2L^2(1+\frac{m}{c})}\min \left\{
         \frac{1}{\textcolor{black}{4} + \frac{\textcolor{black}{5}Mc\eta}{L^2}}, \frac{1}{1 +  \sqrt{\frac{Mc\eta}{4L^2}} }\right\}.
    $$
    
\subsection{Proof of Lemma \ref{lemma:connection}}\label{Appendix:proof_lemma:connection}
    \begin{proof}
     From Fact \ref{fact:1} it follows that  $$\forall x\in \mathcal X ~\exists s_x = \frac{x-x_0}{\|x-x_0\|}\in \R^d : ~\la s_x, \nabla f^i(x)\ra \geq\frac{\beta}{2 R} ~~~ \forall i\in \mathcal I_{\beta/2}(x).$$
    Let $\hat x$ be an approximately optimal point for the log barrier: 
    $B_{\eta}(\hat x) - B_{\eta}(x^*_{\eta}) \leq \eta,$ 
    that is equivalent to:
    $$
        f^0(\hat x) + \eta\sum_{i=1}^m -\log (- f^i(\hat x)) - f^0( x^*_{\eta}) - \eta\sum_{i=1}^m -\log (- f^i(x^*_{\eta})) \leq \eta.
    $$
    Then, for the objective function we have the following bound:
    \begin{align}\label{eq:4.1.2.1}  
        f^0(\hat x) - f^0( x^*_{\eta})  \leq  \eta + \eta\sum_{i=1}^m -\log\frac{- f^i(x^*_{\eta})}{ - f^i(\hat x)}.
    \end{align}
    The optimal point for the log barrier $x^*_{\eta}$ must satisfy the stationarity condition
    $$\nabla B_{\eta}(x^*_{\eta}) = \nabla f^0(x^*_{\eta}) + \eta \sum_{i=1}^m \frac{\nabla f^i(x^*_{\eta})}{-f^i(x^*_{\eta})} = 0.$$ 
    By carefully rearranging the above, we obtain
    $$\sum_{i \in \mathcal I_{\beta/2}(x^*_{\eta})} \frac{  \nabla f^i(x^*_{\eta})}{-f^i(x^*_{\eta})} + \sum_{i \notin \mathcal I_{\beta/2}(x^*_{\eta})}\frac{ \nabla f^i(x^*_{\eta})}{-f^i(x^*_{\eta})} = \frac{-\nabla f^0(x^*_{\eta})}{\eta}.$$ 
    By taking a dot product of both sides of the above equation with $s_x = \frac{x^*_{\eta}-x_0}{\|x^*_{\eta}-x_0\|}$, using the Lipschitz continuity we get for $x^*_{\eta}$:
    \begin{align}
        & \frac{1}{\min_i\{-f^i(x^*_{\eta})\}}  \sum_{i \in \mathcal I_{\beta/2}(x^*_{\eta})} \la \nabla f^i(x^*_{\eta}), s_x\ra \frac{ \min_i\{-f^i(x^*_{\eta})\}}{-f^i(x^*_{\eta})} 
        \\ 
        &
        = \frac{\la -\nabla f^0(x^*_{\eta}), s_x\ra}{\eta}  - \sum_{i\notin \mathcal I_{\beta/2}(x^*_{\eta})}\frac{  \la \nabla f^i(x^*_{\eta}), s_x\ra }{-f^i(x^*_{\eta})} \leq \frac{mL}{\eta}.
    \end{align}
    From the above, using Fact \ref{fact:1}, we get
    $$\min\{-f^i(x^*_{\eta})\}  \geq \frac{\eta \beta}{2mLR}.$$
    Hence, combining the above with (\ref{eq:4.1.2.1}) we get the following relation of point $\hat x$ and point $x^*_{\eta}$ optimal for the log barrier:
        \begin{align}\label{eq:1}
        f^0(\hat x) -f^0( x^*_{\eta})  \leq  \eta +  \eta\sum_{i=1}^m \log\frac{ - f^i(\hat x)}{ - f^i(x^*_{\eta})} \leq \eta \left(1 + m\log\left(\frac{2m LR\hat \beta}{\eta \beta}\right)\right).
        \end{align}
Next, note that the Lagrangian $\mathcal L(x,\lambda)$ is a convex function over $x$ and concave over $\lambda$. Hence, for $(x^*_{\eta}, \lambda^*_{\eta}) := \left(x^*_{\eta}, \left[\frac{\eta}{- f^1(x^*_{\eta})},\ldots,\frac{\eta}{- f^m(x^*_{\eta})}\right]^T\right)$ we have
    \begin{align*} 
         &
         \mathcal L(x^*_{\eta},\lambda^*_{\eta}) - \mathcal L(x^*, \lambda^*)
        \leq   \mathcal L(x^*_{\eta},\lambda^*_{\eta}) - \mathcal L(x^*, \lambda^*_{\eta})
        \leq \la \nabla_{x} \mathcal  L( x^*_{\eta},\lambda^*_{\eta}), \lambda^*_{\eta}  - x^*\ra
        \leq 0.
    \end{align*}
Expressing $\mathcal L(x^*_{\eta},\lambda^*_{\eta})$ and $\mathcal L( x^*,\lambda^*)$ and exploiting the fact that $\nabla B_{\eta}(x^*_{\eta}) = \nabla_{x} \mathcal  L( x^*_{\eta},\lambda^*_{\eta}) = 0$,
 we obtain
    $
        \mathcal L(x^*_{\eta},\lambda^*_{\eta}) - \mathcal L(x^*, \lambda^*) = f^0(x^*_{\eta}) - f^0(x^*) - m\eta \leq 0.
    $    
    Consequently, we have    
       $ f^0(x^*_{\eta}) - f^0(x^*) \leq m \eta.$
   Combining the above and (\ref{eq:1}), we get
        $$ f^0(\hat x) - \min_{x\in \mathcal X}f^0( x)   \leq \eta + \eta m\log\left(\frac{2mLR \hat\beta}{\eta  \beta}\right) + m\eta.$$
\end{proof}

\subsection{Zeroth-order estimator properties proof}\label{Appendix:zero-order-estimator}
The  deviation of the gradient estimators $ G^i(x_t,\nu) - \nabla f^i_{\nu}(x_t)$, by definition can be expressed as follows for $i = 0,\ldots,m$
\begin{align}\label{eq:delta_0}
    G^i(x_t,\nu) - \nabla f^i_{\nu}(x_t)
    = \frac{1}{n_t}\sum_{j=1}^{n_t}
    \left[\underbrace{\left(d\frac{f^i(x_{k}+\nu s_{tj}) - f^i(x_t)}{\nu}s_{tj} - \nabla f^i_{\nu}(x_t)\right)}_{v_j^i} +  \underbrace{d\frac{\xi_{tj}^{i+} - \xi_{tj}^{i-} }{\nu}s_{tj}}_{u_j^i}\right],
\end{align}
where the first term under the summation $v_j^i$ is dependent only on random $s_{tj}$, however the second term is dependent on both random variables coming from the noise $\xi^{i\pm}_{tj}$ and from the direction $s_{tj}$. 

Then, using the fact that the additive noise $\xi_{tj}^{i\pm}$ is zero-mean and independent on $s_{tj}$, we get:
\begin{align}\label{eq:delta_0}
    \E\left\|G^i_{\nu,n}(x_t,\xi) - \nabla f^i_{\nu}(x_t)\right\|^2
    = \E \left\|\frac{1}{n}\sum_{j = 1}^{n} v_j^i\right\|^2 +  \E \left\|\frac{1}{n}\sum_{j = 1}^{n} u_j^i\right\|^2 
\end{align}
Using the result of Lemma 2.10 \citep{Berahas_2021}, we can bound the first part of the above expression $\E \left\|\frac{1}{n}\sum_{j = 1}^{n} v_j^i\right\|^2$:
\begin{align}
        \E \left\|\frac{1}{n}\sum_{j = 1}^{n} v_j^i\right\|^2 
         \leq  \frac{3d^2}{n} \left(\frac{\|\nabla f^i(x)\|^2}{d} +  \frac{ M_i^2 \nu^2}{4}\right).
        ~\forall i\in \{0,\ldots,m\}.
\end{align} 
The second part $u_j^i$ is zero-mean, hence does not influence the bias. Indeed, using the independence of $\xi^{j\pm}_{tj}$ and $s_{tj}$ we derive
\begin{align}\label{eq:E u_j^i}
    \E \sum_{j = 1}^{n_t} u_j^i = \frac{d}{\nu} \E\left(\sum_{j = 1}^{n_t} (\xi^{i+}_{tj}  -  \xi^{i-}_{tj}) s_{tj}\right) = 0.
\end{align}
 Its variance can be bounded as follows,
 using $\|s_{tj}\| = 1$:
\begin{align}\label{eq:E||u_j^i||}
    \E \left\|\frac{1}{n}\sum_{j = 1}^{n} u_j^i\right\|^2 = \E \frac{d^2}{\nu^2n^2}\left\|\sum_{j = 1}^{n} (\xi^{i+}_{tj} - \xi^{i-}_{tj} ) s_{tj}\right\|^2 \leq 4 \frac{d^2}{\nu^2 n^2}\sum_{j = 1}^{n} \E\|\xi^{i+}_{tj}\|^2\|s_{tj} \|^2 \leq 4\frac{d^2\sigma^2}{\nu^2 n}.
\end{align}
From the above, and Lemma 2.10 \citep{Berahas_2021} the statement of the Lemma follows directly.
\newpage
\addcontentsline{toc}{section}{Bibliography}
\bibliography{bibliography}{}

\begin{thebibliography}{61}
\providecommand{\natexlab}[1]{#1}
\providecommand{\url}[1]{\texttt{#1}}
\expandafter\ifx\csname urlstyle\endcsname\relax
  \providecommand{\doi}[1]{doi: #1}\else
  \providecommand{\doi}{doi: \begingroup \urlstyle{rm}\Url}\fi

\bibitem[Achiam et~al.(2017)Achiam, Held, Tamar, and
  Abbeel]{achiam2017constrained}
Joshua Achiam, David Held, Aviv Tamar, and Pieter Abbeel.
\newblock Constrained policy optimization, 2017.

\bibitem[Altman(1999)]{altman-constrainedMDP}
E.~Altman.
\newblock \emph{Constrained Markov Decision Processes}.
\newblock Chapman and Hall, 1999.

\bibitem[Amani et~al.(2019)Amani, Alizadeh, and Thrampoulidis]{amani2019linear}
Sanae Amani, Mahnoosh Alizadeh, and Christos Thrampoulidis.
\newblock Linear stochastic bandits under safety constraints.
\newblock \emph{arXiv preprint arXiv:1908.05814}, 2019.

\bibitem[Arjevani et~al.(2019)Arjevani, Carmon, Duchi, Foster, Srebro, and
  Woodworth]{arjevani2019lower}
Yossi Arjevani, Yair Carmon, John~C Duchi, Dylan~J Foster, Nathan Srebro, and
  Blake Woodworth.
\newblock Lower bounds for non-convex stochastic optimization.
\newblock \emph{arXiv preprint arXiv:1912.02365}, 2019.

\bibitem[As et~al.(2022)As, Usmanova, Curi, and Krause]{As2022Constrained}
Yarden As, Ilnura Usmanova, Sebastian Curi, and Andreas Krause.
\newblock Constrained policy optimization via bayesian world models.
\newblock \emph{ArXiv}, 2022.
\newblock URL \url{https://arxiv.org/abs/2201.09802}.

\bibitem[Bach and Perchet(2016)]{bach2016highly}
Francis Bach and Vianney Perchet.
\newblock Highly-smooth zero-th order online optimization.
\newblock In \emph{Conference on Learning Theory}, pages 257--283, 2016.

\bibitem[Balasubramanian and Ghadimi(2018)]{balasubramanian2018zeroth}
Krishnakumar Balasubramanian and Saeed Ghadimi.
\newblock Zeroth-order (non)-convex stochastic optimization via conditional
  gradient and gradient updates.
\newblock In \emph{Advances in Neural Information Processing Systems}, pages
  3455--3464, 2018.

\bibitem[Berahas et~al.(2021)Berahas, Cao, Choromanski, and
  Scheinberg]{Berahas_2021}
Albert Berahas, Liyuan Cao, Krzysztof Choromanski, and Katya Scheinberg.
\newblock A theoretical and empirical comparison of gradient approximations in
  derivative-free optimization.
\newblock \emph{Foundations of Computational Mathematics}, 05 2021.
\newblock \doi{10.1007/s10208-021-09513-z}.

\bibitem[Berkenkamp et~al.(2016{\natexlab{a}})Berkenkamp, Krause, and
  Schoellig]{berkenkamp2016bayesian}
Felix Berkenkamp, Andreas Krause, and Angela~P Schoellig.
\newblock Bayesian optimization with safety constraints: safe and automatic
  parameter tuning in robotics.
\newblock \emph{arXiv preprint arXiv:1602.04450}, 2016{\natexlab{a}}.

\bibitem[Berkenkamp et~al.(2016{\natexlab{b}})Berkenkamp, Schoellig, and
  Krause]{7487170}
Felix Berkenkamp, Angela~P. Schoellig, and Andreas Krause.
\newblock Safe controller optimization for quadrotors with gaussian processes.
\newblock In \emph{2016 IEEE International Conference on Robotics and
  Automation (ICRA)}, pages 491--496, 2016{\natexlab{b}}.
\newblock \doi{10.1109/ICRA.2016.7487170}.

\bibitem[Berkenkamp et~al.(2017)Berkenkamp, Turchetta, Schoellig, and
  Krause]{berkenkamp2017safe}
Felix Berkenkamp, Matteo Turchetta, Angela~P. Schoellig, and Andreas Krause.
\newblock Safe model-based reinforcement learning with stability guarantees,
  2017.

\bibitem[Berkenkamp et~al.(2020)Berkenkamp, Krause, and
  Schoellig]{berkenkamp2020bayesian}
Felix Berkenkamp, Andreas Krause, and Angela~P. Schoellig.
\newblock Bayesian optimization with safety constraints: Safe and automatic
  parameter tuning in robotics, 2020.

\bibitem[Bubeck et~al.(2017)Bubeck, Lee, and Eldan]{bubeck2017kernel}
S{\'e}bastien Bubeck, Yin~Tat Lee, and Ronen Eldan.
\newblock Kernel-based methods for bandit convex optimization.
\newblock In \emph{Proceedings of the 49th Annual ACM SIGACT Symposium on
  Theory of Computing}, pages 72--85, 2017.

\bibitem[Chen et~al.(2019)Chen, Zhang, and Karbasi]{chen2019projection}
Lin Chen, Mingrui Zhang, and Amin Karbasi.
\newblock Projection-free bandit convex optimization.
\newblock In \emph{The 22nd International Conference on Artificial Intelligence
  and Statistics}, pages 2047--2056. PMLR, 2019.

\bibitem[Chow et~al.(2015)Chow, Ghavamzadeh, Janson, and
  Pavone]{DBLP:journals/corr/ChowGJP15}
Yinlam Chow, Mohammad Ghavamzadeh, Lucas Janson, and Marco Pavone.
\newblock Risk-constrained reinforcement learning with percentile risk
  criteria.
\newblock \emph{CoRR}, abs/1512.01629, 2015.
\newblock URL \url{http://arxiv.org/abs/1512.01629}.

\bibitem[Chua et~al.(2018)Chua, Calandra, McAllister, and
  Levine]{DBLP:journals/corr/abs-1805-12114}
Kurtland Chua, Roberto Calandra, Rowan McAllister, and Sergey Levine.
\newblock Deep reinforcement learning in a handful of trials using
  probabilistic dynamics models.
\newblock \emph{CoRR}, abs/1805.12114, 2018.
\newblock URL \url{http://arxiv.org/abs/1805.12114}.

\bibitem[Clevert et~al.(2015)Clevert, Unterthiner, and Hochreiter]{elupaper}
Djork-Arné Clevert, Thomas Unterthiner, and Sepp Hochreiter.
\newblock Fast and accurate deep network learning by exponential linear units
  (elus), 2015.
\newblock URL \url{https://arxiv.org/abs/1511.07289}.

\bibitem[Dalal et~al.(2018)Dalal, Dvijotham, Vecerik, Hester, Paduraru, and
  Tassa]{dalal2018safe}
Gal Dalal, Krishnamurthy Dvijotham, Matej Vecerik, Todd Hester, Cosmin
  Paduraru, and Yuval Tassa.
\newblock Safe exploration in continuous action spaces, 2018.

\bibitem[Deisenroth and Rasmussen(2011)]{10.5555/3104482.3104541}
Marc~Peter Deisenroth and Carl~Edward Rasmussen.
\newblock Pilco: A model-based and data-efficient approach to policy search.
\newblock In \emph{Proceedings of the 28th International Conference on
  International Conference on Machine Learning}, ICML’11, page 465–472,
  Madison, WI, USA, 2011. Omnipress.
\newblock ISBN 9781450306195.

\bibitem[Duchi et~al.(2015)Duchi, Jordan, Wainwright, and Wibisono]{Duchi2015}
John~C. Duchi, Michael~I. Jordan, Martin~J. Wainwright, and Andre Wibisono.
\newblock Optimal rates for zero-order convex optimization: The power of two
  function evaluations.
\newblock \emph{IEEE Transactions on Information Theory}, 61\penalty0
  (5):\penalty0 2788--2806, 2015.
\newblock \doi{10.1109/TIT.2015.2409256}.

\bibitem[Díaz-Francés and
  Rubio(2013)]{RePEc:spr:stpapr:v:54:y:2013:i:2:p:309-323}
Eloísa Díaz-Francés and Francisco Rubio.
\newblock {On the existence of a normal approximation to the distribution of
  the ratio of two independent normal random variables}.
\newblock \emph{Statistical Papers}, 54\penalty0 (2):\penalty0 309--323, May
  2013.
\newblock \doi{10.1007/s00362-012-0429-2}.
\newblock URL
  \url{https://ideas.repec.org/a/spr/stpapr/v54y2013i2p309-323.html}.

\bibitem[Eriksson and
  Jankowiak(2021)]{https://doi.org/10.48550/arxiv.2103.00349}
David Eriksson and Martin Jankowiak.
\newblock High-dimensional bayesian optimization with sparse axis-aligned
  subspaces, 2021.
\newblock URL \url{https://arxiv.org/abs/2103.00349}.

\bibitem[Fazlyab et~al.(2019)Fazlyab, Robey, Hassani, Morari, and
  Pappas]{fazlyab2019efficient}
Mahyar Fazlyab, Alexander Robey, Hamed Hassani, Manfred Morari, and George
  Pappas.
\newblock Efficient and accurate estimation of lipschitz constants for deep
  neural networks.
\newblock \emph{Advances in Neural Information Processing Systems}, 32, 2019.

\bibitem[Fereydounian et~al.(2020)Fereydounian, Shen, Mokhtari, Karbasi, and
  Hassani]{fereydounian2020safe}
Mohammad Fereydounian, Zebang Shen, Aryan Mokhtari, Amin Karbasi, and Hamed
  Hassani.
\newblock Safe learning under uncertain objectives and constraints.
\newblock \emph{arXiv preprint arXiv:2006.13326}, 2020.

\bibitem[Flaxman et~al.(2005)Flaxman, Kalai, and McMahan]{flaxman2005online}
Abraham~D Flaxman, Adam~Tauman Kalai, and H~Brendan McMahan.
\newblock Online convex optimization in the bandit setting: gradient descent
  without a gradient.
\newblock In \emph{Proceedings of the sixteenth annual ACM-SIAM symposium on
  Discrete algorithms}, pages 385--394. Society for Industrial and Applied
  Mathematics, 2005.

\bibitem[Frazier(2018)]{https://doi.org/10.48550/arxiv.1807.02811}
Peter~I. Frazier.
\newblock A tutorial on bayesian optimization, 2018.
\newblock URL \url{https://arxiv.org/abs/1807.02811}.

\bibitem[Garber and Kretzu(2020)]{garber2020improved}
Dan Garber and Ben Kretzu.
\newblock Improved regret bounds for projection-free bandit convex
  optimization.
\newblock In \emph{International Conference on Artificial Intelligence and
  Statistics}, pages 2196--2206. PMLR, 2020.

\bibitem[Hafner et~al.(2021)Hafner, Lillicrap, Norouzi, and
  Ba]{hafner2021mastering}
Danijar Hafner, Timothy Lillicrap, Mohammad Norouzi, and Jimmy Ba.
\newblock Mastering atari with discrete world models, 2021.

\bibitem[Hansen and Ostermeier(2001)]{journals/ec/HansenO01}
Nikolaus Hansen and Andreas Ostermeier.
\newblock Completely derandomized self-adaptation in evolution strategies.
\newblock \emph{Evol. Comput.}, 9\penalty0 (2):\penalty0 159--195, 2001.
\newblock URL \url{http://dblp.uni-trier.de/db/journals/ec/ec9.html#HansenO01}.

\bibitem[Hazan and Luo(2016)]{hazan2016variance}
Elad Hazan and Haipeng Luo.
\newblock Variance-reduced and projection-free stochastic optimization.
\newblock In \emph{International Conference on Machine Learning}, pages
  1263--1271, 2016.

\bibitem[Hinder and Ye(2018)]{hinder2018one}
Oliver Hinder and Yinyu Ye.
\newblock A one-phase interior point method for nonconvex optimization.
\newblock \emph{arXiv preprint arXiv:1801.03072}, 2018.

\bibitem[Hinder and Ye(2019)]{hinder2019poly}
Oliver Hinder and Yinyu Ye.
\newblock A polynomial time log barrier method for problems with nonconvex
  constraints.
\newblock \emph{arXiv preprint: https://arxiv.org/pdf/1807.00404.pdf}, 2019.

\bibitem[Juditsky et~al.(2013)Juditsky, Lan, Nemirovski, and
  Shapiro]{juditsky:hal-00853911}
Anatoli~B. Juditsky, Guanghui Lan, Arkadii~S. Nemirovski, and Alexander
  Shapiro.
\newblock {Stochastic Approximation approach to Stochastic Programming}.
\newblock Research report, {LJK}, 2013.
\newblock URL \url{https://hal.archives-ouvertes.fr/hal-00853911}.
\newblock http://www.optimization-online.org/DB\_HTML/2007/09/1787.html.

\bibitem[Kennedy and Eberhart(1995)]{kennedy95particle}
James Kennedy and Russell~C. Eberhart.
\newblock Particle swarm optimization.
\newblock In \emph{Proceedings of the IEEE International Conference on Neural
  Networks}, pages 1942--1948, 1995.

\bibitem[Kirschner et~al.(2019)Kirschner, Mutn{\`y}, Hiller, Ischebeck, and
  Krause]{kirschner2019adaptive}
Johannes Kirschner, Mojmir Mutn{\`y}, Nicole Hiller, Rasmus Ischebeck, and
  Andreas Krause.
\newblock Adaptive and safe bayesian optimization in high dimensions via
  one-dimensional subspaces.
\newblock \emph{arXiv preprint arXiv:1902.03229}, 2019.

\bibitem[Koller et~al.(2018)Koller, Berkenkamp, Turchetta, and
  Krause]{koller2018learning}
Torsten Koller, Felix Berkenkamp, Matteo Turchetta, and Andreas Krause.
\newblock Learning-based model predictive control for safe exploration.
\newblock In \emph{2018 IEEE Conference on Decision and Control (CDC)}, pages
  6059--6066. IEEE, 2018.

\bibitem[Lan(2020)]{Lan_book}
Guanghui Lan.
\newblock \emph{First-order and Stochastic Optimization Methods for Machine
  Learning}.
\newblock 01 2020.
\newblock ISBN 978-3-030-39567-4.
\newblock \doi{10.1007/978-3-030-39568-1}.

\bibitem[Lillicrap et~al.(2015)Lillicrap, Hunt, Pritzel, Heess, Erez, Tassa,
  Silver, and Wierstra]{https://doi.org/10.48550/arxiv.1509.02971}
Timothy~P. Lillicrap, Jonathan~J. Hunt, Alexander Pritzel, Nicolas Heess, Tom
  Erez, Yuval Tassa, David Silver, and Daan Wierstra.
\newblock Continuous control with deep reinforcement learning, 2015.
\newblock URL \url{https://arxiv.org/abs/1509.02971}.

\bibitem[Luersen et~al.(2004)Luersen, Le~Riche, and Guyon]{luersen2004}
M.~Luersen, R.~Le~Riche, and F.~A Guyon.
\newblock Constrained, globalized, and bounded nelder–mead method for
  engineering optimization.
\newblock \emph{Struct Multidisc Optim 27}, page 43–54, 2004.
\newblock \doi{10.1007/s00158-003-0320-9}.

\bibitem[Mangasarian and Fromovitz(1967)]{MANGASARIAN196737}
O.L Mangasarian and S~Fromovitz.
\newblock The fritz john necessary optimality conditions in the presence of
  equality and inequality constraints.
\newblock \emph{Journal of Mathematical Analysis and Applications}, 17\penalty0
  (1):\penalty0 37--47, 1967.
\newblock ISSN 0022-247X.
\newblock \doi{https://doi.org/10.1016/0022-247X(67)90163-1}.
\newblock URL
  \url{https://www.sciencedirect.com/science/article/pii/0022247X67901631}.

\bibitem[Moriconi et~al.(2019)Moriconi, Deisenroth, and
  Kumar]{https://doi.org/10.48550/arxiv.1902.10675}
Riccardo Moriconi, Marc~P. Deisenroth, and K.~S.~Sesh Kumar.
\newblock High-dimensional bayesian optimization using low-dimensional feature
  spaces, 2019.
\newblock URL \url{https://arxiv.org/abs/1902.10675}.

\bibitem[Nelder and Mead(1965)]{10.1093/comjnl/7.4.308}
J.~A. Nelder and R.~Mead.
\newblock {A Simplex Method for Function Minimization}.
\newblock \emph{The Computer Journal}, 7\penalty0 (4):\penalty0 308--313, 01
  1965.
\newblock ISSN 0010-4620.
\newblock \doi{10.1093/comjnl/7.4.308}.
\newblock URL \url{https://doi.org/10.1093/comjnl/7.4.308}.

\bibitem[Nemirovsky and Yudin(1985)]{Nemirovsky_Yudin}
A.~S. Nemirovsky and D.~B. Yudin.
\newblock Problem complexity and method efficiency in optimization.
\newblock \emph{SIAM Review}, 27\penalty0 (2):\penalty0 264--265, 1985.
\newblock \doi{10.1137/1027074}.
\newblock URL \url{https://doi.org/10.1137/1027074}.

\bibitem[Nocedal and Wright(2006)]{nocedal2006numerical}
Jorge Nocedal and Stephen Wright.
\newblock \emph{Numerical optimization}.
\newblock Springer Science \& Business Media, 2006.

\bibitem[Price(2019)]{NM-barriers}
Christopher Price.
\newblock A modified nelder-mead barrier method for constrained optimization.
\newblock \emph{Numerical Algebra, Control, and Optimization}, 11, 01 2019.
\newblock \doi{10.3934/naco.2020058}.

\bibitem[Rasmussen and Williams(2005)]{10.5555/1162254}
Carl~Edward Rasmussen and Christopher K.~I. Williams.
\newblock \emph{Gaussian Processes for Machine Learning (Adaptive Computation
  and Machine Learning)}.
\newblock The MIT Press, 2005.
\newblock ISBN 026218253X.

\bibitem[Ray et~al.(2019)Ray, Achiam, and Amodei]{Ray2019}
Alex Ray, Joshua Achiam, and Dario Amodei.
\newblock {Benchmarking Safe Exploration in Deep Reinforcement Learning}.
\newblock 2019.

\bibitem[Rechenberg(1989)]{10.1007/978-3-642-83814-9_6}
Ingo Rechenberg.
\newblock Evolution strategy: Nature's way of optimization.
\newblock In H.~W. Bergmann, editor, \emph{Optimization: Methods and
  Applications, Possibilities and Limitations}, pages 106--126, Berlin,
  Heidelberg, 1989. Springer Berlin Heidelberg.
\newblock ISBN 978-3-642-83814-9.

\bibitem[Schulman et~al.(2015)Schulman, Moritz, Levine, Jordan, and
  Abbeel]{https://doi.org/10.48550/arxiv.1506.02438}
John Schulman, Philipp Moritz, Sergey Levine, Michael Jordan, and Pieter
  Abbeel.
\newblock High-dimensional continuous control using generalized advantage
  estimation, 2015.
\newblock URL \url{https://arxiv.org/abs/1506.02438}.

\bibitem[Shamir(2013)]{shamir2013complexity}
Ohad Shamir.
\newblock On the complexity of bandit and derivative-free stochastic convex
  optimization.
\newblock In \emph{Conference on Learning Theory}, pages 3--24, 2013.

\bibitem[Shapiro et~al.(2009)Shapiro, Dentcheva, and
  Ruszczy\'nski]{ShapDentRusz09}
Alexander Shapiro, Darinka Dentcheva, and Andrzej Ruszczy\'nski.
\newblock \emph{Lectures on Stochastic Programming}.
\newblock SIAM, Philadelphia, PA, USA, 2009.

\bibitem[Snoek et~al.(2015)Snoek, Rippel, Swersky, Kiros, Satish, Sundaram,
  Patwary, {Prabhat}, and Adams]{https://doi.org/10.48550/arxiv.1502.05700}
Jasper Snoek, Oren Rippel, Kevin Swersky, Ryan Kiros, Nadathur Satish,
  Narayanan Sundaram, Md. Mostofa~Ali Patwary, {Prabhat}, and Ryan~P. Adams.
\newblock Scalable bayesian optimization using deep neural networks, 2015.
\newblock URL \url{https://arxiv.org/abs/1502.05700}.

\bibitem[Srinivas et~al.(2012)Srinivas, Krause, Kakade, and
  Seeger]{srinivas12information}
Niranjan Srinivas, Andreas Krause, Sham Kakade, and Matthias Seeger.
\newblock Information-theoretic regret bounds for gaussian process optimization
  in the bandit setting.
\newblock \emph{IEEE Transactions on Information Theory}, 58\penalty0
  (5):\penalty0 3250--3265, May 2012.
\newblock \doi{10.1109/TIT.2011.2182033}.

\bibitem[Storn and Price(1997)]{journals/jgo/StornP97}
Rainer Storn and Kenneth~V. Price.
\newblock Differential evolution - a simple and efficient heuristic for global
  optimization over continuous spaces.
\newblock \emph{J. Glob. Optim.}, 11\penalty0 (4):\penalty0 341--359, 1997.
\newblock URL
  \url{http://dblp.uni-trier.de/db/journals/jgo/jgo11.html#StornP97}.

\bibitem[Sui et~al.(2015{\natexlab{a}})Sui, Gotovos, Burdick, and
  Krause]{pmlr-v37-sui15}
Yanan Sui, Alkis Gotovos, Joel Burdick, and Andreas Krause.
\newblock Safe exploration for optimization with gaussian processes.
\newblock In Francis Bach and David Blei, editors, \emph{Proceedings of the
  32nd International Conference on Machine Learning}, volume~37 of
  \emph{Proceedings of Machine Learning Research}, pages 997--1005, Lille,
  France, 07--09 Jul 2015{\natexlab{a}}. PMLR.
\newblock URL \url{https://proceedings.mlr.press/v37/sui15.html}.

\bibitem[Sui et~al.(2015{\natexlab{b}})Sui, Gotovos, Burdick, and
  Krause]{sui2015safe}
Yanan Sui, Alkis Gotovos, Joel Burdick, and Andreas Krause.
\newblock Safe exploration for optimization with gaussian processes.
\newblock In \emph{International Conference on Machine Learning}, pages
  997--1005, 2015{\natexlab{b}}.

\bibitem[Sutton et~al.(2000)Sutton, McAllester, Singh, and
  Mansour]{NIPS1999_464d828b}
Richard~S Sutton, David McAllester, Satinder Singh, and Yishay Mansour.
\newblock Policy gradient methods for reinforcement learning with function
  approximation.
\newblock In S.~Solla, T.~Leen, and K.~M\"{u}ller, editors, \emph{Advances in
  Neural Information Processing Systems}, volume~12. MIT Press, 2000.
\newblock URL
  \url{https://proceedings.neurips.cc/paper/1999/file/464d828b85b0bed98e80ade0a5c43b0f-Paper.pdf}.

\bibitem[Usmanova et~al.(2019)Usmanova, Krause, and
  Kamgarpour]{usmanova2019safe}
Ilnura Usmanova, Andreas Krause, and Maryam Kamgarpour.
\newblock Safe convex learning under uncertain constraints.
\newblock In \emph{The 22nd International Conference on Artificial Intelligence
  and Statistics}, pages 2106--2114, 2019.

\bibitem[Usmanova et~al.(2020)Usmanova, Krause, and
  Kamgarpour]{usmanova2020safe}
Ilnura Usmanova, Andreas Krause, and Maryam Kamgarpour.
\newblock Safe non-smooth black-box optimization with application to policy
  search.
\newblock In \emph{Learning for Dynamics and Control}, pages 980--989, 2020.

\bibitem[Vaswani et~al.(2021)Vaswani, Dubois-Taine, and
  Babanezhad]{vaswani:hal-03456663}
Sharan Vaswani, Benjamin Dubois-Taine, and Reza Babanezhad.
\newblock {Towards Noise-adaptive, Problem-adaptive Stochastic Gradient
  Descent}.
\newblock working paper or preprint, November 2021.
\newblock URL \url{https://hal.archives-ouvertes.fr/hal-03456663}.

\bibitem[Wabersich and Zeilinger(2021)]{wabersich2021predictive}
Kim~P. Wabersich and Melanie~N. Zeilinger.
\newblock A predictive safety filter for learning-based control of constrained
  nonlinear dynamical systems, 2021.

\end{thebibliography}
\newpage
\end{document}